\documentclass{amsart}
\usepackage{amsthm}
\usepackage{amsfonts}
\usepackage{amsmath}
\usepackage{amssymb}
\usepackage{tikz-cd}
\usepackage[colorlinks=true,linkcolor=black, citecolor=black]{hyperref}
\calclayout

\newtheorem{theo}{Theorem}[section]
\newtheorem{prop}[theo]{Proposition}
\newtheorem{coro}[theo]{Corollary}
\newtheorem{lemm}[theo]{Lemma}
\theoremstyle{remark}
\newtheorem{rema}[theo]{Remark}
\theoremstyle{definition}
\newtheorem{defi}[theo]{Definition}
\newtheorem{exam}[theo]{Example}
\newtheorem{claim}{Claim}[theo]

\numberwithin{equation}{section}

\newcommand{\Gr}{\text{Gr}(2,\mathfrak{sl}(3,\mathbb{R}))}
\newcommand{\s}{\mathfrak{sl}(3,\mathbb{R})}
\newcommand{\G}{\mathrm{SL}(3,\mathbb{R})}
\DeclareMathOperator{\spn}{span}
\DeclareMathOperator{\Ima}{Im}

\begin{document}

\title{Flat compact Hermite-Lorentz manifolds in dimension 4}

\author{Bianca Barucchieri}

\begin{abstract}
	We give a classification, up to finite cover, of flat compact complete Hermite-Lorentz manifolds up to complex dimension $4$. 
\end{abstract}

\maketitle

\section{Introduction}
For a vector space $V$, subgroups $\Gamma$ of the affine group $\mathrm{Aff}(V)=\mathrm{GL}(V)\ltimes V$ that act properly discontinuously and cocompactly on $V$ are called crystallographic groups and have a long history. The first results date back to Bieberbach around the years 1911-12, when he studied the case $\Gamma\le \mathrm{O}(n)\ltimes\mathbb{R}^n$. His results show that such groups are, up to finite index, abelian groups made of translations.
Later the Lorentzian case, $\Gamma\le\mathrm{O}(n,1)\ltimes\mathbb{R}^{n+1}$, was studied. For this case the first results were obtained by Auslander and Markus in dimension $3$, \cite{AuslMark}, and, more generally, Fried and Goldman studied all $3$-dimensional crystallographic groups in \cite{FG}. Still for the Lorentzian case, in the years 1987-88  Fried investigated the dimension $4$, \cite{Fried}, and afterwards results for all dimensions were given by Grunewald and Margulis, \cite{Margulis}. In this article we are interested in the Hermite-Lorentz case, namely $\Gamma\le\mathrm{U}(n,1)\ltimes\mathbb{C}^{n+1}$. Following the strategy used by Grunewald and Margulis in \cite{Margulis} we give a classification, up to finite index, of these crystallographic groups for $n\le3$. These groups are the fundamental groups of flat compact complete Hermite-Lorentz manifolds and they determine the manifold. Indeed, after a classical result of Mostow \cite{Mostow}, if two such groups are isomorphic the corresponding manifolds are diffeomorphic. Since we are interested in classifying the manifolds we will only be concerned with isomorphism classes of these groups and not with the different ways one can realise them as subgroups of $\mathrm{U}(n,1)\ltimes\mathbb{C}^{n+1}$. Therefore, according to Mostow's Theorem, the classification, up to finite index, of crystallographic subgroups $\Gamma\le\mathrm{U}(n,1)\ltimes\mathbb{C}^{n+1}$ given in this article corresponds to a classification, up to finite covering, and for complex dimension at most $4$, of flat compact complete Hermite-Lorentz manifolds. Indeed, let $\mathfrak{a}(\mathbb{C}^{n+1})$ be the affine space associated to the complex vector space $\mathbb{C}^{n+1}$ endowed with a Hermitian form of signature $(n,1)$ and $\mathrm{U}(n,1)$ be the subgroup of the general linear group preserving this Hermitian form. Flat compact complete Hermite-Lorentz manifolds can then be presented as the quotient $\Gamma \backslash \mathfrak{a}(\mathbb{C}^{n+1})$ where  the group $\Gamma$ is a subgroup of $\mathcal{H}(n,1)=\mathrm{U}(n,1)\ltimes \mathbb{C}^{n+1}$ acting properly discontinuously and cocompactly on $\mathfrak{a}(\mathbb{C}^{n+1})$. These manifolds can be thought as analogue of both Hermitian manifolds (definite positive) in complex geometry and Lorentzian manifolds in real differential geometry, see \cite{Zeghib}. Let us notice that Scheuneman in \cite{Sch1} already studied the case of compact complete affine complex surfaces.

About crystallographic groups there is a long-standing conjecture due to Auslander that says that such groups are virtually solvable. Otherwise said, from the manifold point of view, the fundamental group of every compact complete flat affine manifold is virtually solvable. Sometimes this statement is formulated using the term polycyclic instead of solvable. 

\begin{rema}
A polycyclic group is solvable. The converse is not true in general. But every discrete solvable subgroup of $\mathrm{GL}(n,\mathbb{R})$ is polycyclic, \cite[Proposition 3.10]{Rag}. Hence we will use the word polycyclic or solvable interchangeably.
\end{rema}

In some special cases the Auslander conjecture is known to be true. One such case is the Riemannian case as we have seen in Bieberbach's result. Another such case is the Lorentzian case, which was proved in \cite{Goldman} by Goldman and Kamishima, see also \cite{Carr}. Finally, after a result of Grunewald and Margulis, \cite{Margulis}, also the case that interests us, namely the Hermite-Lorentz case $\Gamma\le \mathrm{U}(n,1)\ltimes \mathbb{C}^{n+1}$, is known to satisfy the Auslander conjecture.

For virtually solvable crystallographic groups Fried and Goldman gave the following generalisation of Bieberbach's result.

\begin{theo}[{\cite[ Corollary 1.5]{FG}}]\label{FG}
Let $V$ be a finite dimensional real vector space and $\Gamma\le\mathrm{Aff}(V)$ be a virtually solvable group acting properly discontinuously and cocompactly on $\mathfrak{a}(V)$ then there exists a subgroup $H\le\mathrm{Aff}(V)$ such that
\begin{enumerate}
\item $H$ acts simply transitively on $\mathfrak{a}(V)$;
\item $\Gamma\cap H$ has finite index in $\Gamma$; 
\item $\Gamma\cap H$ is a lattice in $H$.
\end{enumerate}
\end{theo}

Indeed, when the linear part of $\Gamma$ preserves a positive definite bilinear form, Bieberbach's results tell us that the simply transitive group of the theorem is the group of pure translations. The simply transitive groups appearing in this theorem are sometimes called connected crystallographic hulls.
This theorem will be our starting point for the investigation of the Hermite-Lorentz manifolds $\Gamma \backslash \mathfrak{a}(\mathbb{C}^{n+1})$. More precisely, in order to study crystallographic subgroups $\Gamma\le\mathrm{U}(n,1)\ltimes \mathbb{C}^{n+1}$ we will study lattices in simply transitive affine groups $H\le \mathrm{U}(n,1)\ltimes \mathbb{C}^{n+1}$. Hence, even though the Hermite-Lorentz manifolds are complex manifolds, we see them as real objects with an integrable complex structure. More precisely we will look at these manifolds, up to finite index, as $\Gamma\backslash H$ where $H$ is a simply transitive subgroup of $\mathrm{U}(n,1)\ltimes \mathbb{C}^{n+1}$ and $\Gamma$ is a lattice in it. The Lie groups $H$ that act simply transitively on $\mathbb{C}^{n+1}$ are real Lie groups which inherit an integrable left invariant complex structure from $\mathbb{C}^{n+1}$. The integrability condition, see for example \cite{Salamon}, is more general than the condition for a Lie group to be a complex Lie group. Indeed the Lie groups we are looking at are not complex Lie groups. \\

This article is organised as follows. In Section \ref{UnSimTr} we will study unipotent groups $H$ acting simply transitively on $\mathfrak{a}(\mathbb{C}^{n+1})$. Let us remember the following definition.

\begin{defi}
A subgroup $U$ of the affine group $\mathrm{Aff}(V)$ is \emph{unipotent} if its linear part consists of endomorphisms whose eigenvalues are all $1$.
\end{defi}

\begin{rema}
A unipotent group is, in particular, a nilpotent group. The converse in general is not true. But every simply transitive nilpotent affine group of motion that appears as a connected crystallographic hull is unipotent, \cite[Theorem A]{FGH}. 
\end{rema}

Thus in Section \ref{UnSimTr} we give a presentation, valid in any dimension, of unipotent simply transitive groups of Hermite-Lorentz affine motion and prove a proposition about their conjugacy classes. Let us remark that, although we haven't adopted this point of view at all, Lie groups acting simply transitively on $\mathfrak{a}(\mathbb{C}^{n+1})$ can be seen as real Lie groups that are endowed with a left invariant flat Hermite-Lorentz metric. There is a vast literature that treats similar questions with the point of view of metric Lie algebras. For example, \cite{Medina} and \cite{Boucetta} give construction results for Lorentzian and pseudo-Riemannian nilpotent Lie algebras of signature $(2,n-2)$. But in these articles they are not interested in the classification problem. The Lie algebras that appear in this article are particular cases of pseudo-Riemannian nilpotent Lie algebras of signature $(2,n-2)$. A posteriori we can show, but we will not do it in this article, that the Lie algebras we have found can be obtained with a double extension process as in \cite{Boucetta}. \\

In Section \ref{ProperlyDisc} we reduce, as for the Lorentzian case, the study of crystallographic groups to the study of lattices in simply transitive unipotent Lie groups. Let us define some terminology.

\begin{defi}
	A group $\Gamma$ is said to be \emph{abelian by cyclic} if it has a normal abelian subgroup $\Gamma_1$ such that $\Gamma/\Gamma_1$ is cyclic.
\end{defi}

We then prove the following.

\begin{theo}\label{cryst}
	Let $\Gamma$ be a subgroup of $\mathcal{H}(n,1)$ that acts properly discontinuously and cocompactly on $\mathfrak{a}(V)$ then $\Gamma$ is either virtually nilpotent or contains a subgroup of finite index that is abelian by cyclic.
\end{theo}

According to this theorem, the unipotent hypothesis we made in Section \ref{UnSimTr} only leaves out the easy abelian by cyclic case.
We end this section with the classification of the latter case.\\

In Section \ref{IsomDeg} we start the classification, up to isomorphism, of those $H$ found in Section \ref{UnSimTr}. We complete their classification for the dimensions $2$ and $3$. In dimension $4$ we give the classification for some particular cases that we have called degenerate cases. \\

In Section \ref{IsomGen} we finish the classification of the unipotent simply transitive groups $H$ in dimension $4$. For this more general case we are left with a classification problem of $8$-dimensional nilpotent Lie algebras defined by three parameters. Since there are no complete classifications of nilpotent Lie algebras in dimension bigger than $7$ we introduce an ad hoc method to study our particular family of Lie algebras. Using the fact that these Lie algebras are Carnot, see Subsection \ref{Carnot}, we can identify their isomorphism classes with the orbits of the $\G$-action on the Grassmannian of $2$-dimensional subspaces of $\s$ induced by the adjoint action. Hence studying the orbits for this action is the same as studying the isomorphism classes for the family of Lie algebras.
In contrast with the Lorentzian case where, for every fixed dimension, there are finitely many isomorphism classes of unipotent subgroups of $\mathrm{O}(n,1)\ltimes\mathbb{R} ^{n+1}$ acting simply transitively on $\mathfrak{a}(\mathbb{R}^{n+1})$ we find that these Lie algebras constitute an infinite family of pairwise non isomorphic Lie algebras of Lie groups acting simply transitively on $\mathfrak{a}(\mathbb{C}^{3+1})$. In order to classify unipotent simply transitive subgroups of $\mathrm{U}(3,1)\ltimes \mathbb{C}^{3+1}$ up to isomorphism it is sufficient to classify their Lie algebras up to Lie algebra isomorphism. Thus in this section we complete the proof of the following.

\begin{theo}\label{ClassLieAlg}
	The list given in Appendix \ref{LieBracket} with $\mathbb{K}=\mathbb{R}$ is a complete non redundant list of isomorphism classes of unipotent simply transitive subgroups of $\mathrm{U}(3,1)\ltimes \mathbb{C}^{3+1}$. 
\end{theo}

In Section \ref{lattices} we study the Hermite-Lorentz crystallographic groups in the nilpotent case. Indeed we have reduced it to the study of lattices in simply transitive unipotent groups. Since in Sections \ref{IsomDeg} and \ref{IsomGen} we have given the complete list of those simply transitive groups we just need to understand the abstract commensurability classes of lattices of these groups. Let us give the following definition. 

\begin{defi}
Two groups $\Gamma_1$ and $\Gamma_2$ are said to be \emph{abstractly commensurable} if there exist two subgroups $\Delta_1\le \Gamma_1$ and $\Delta_2\le \Gamma_2$ of finite index such that $\Delta_1$ is isomorphic to $\Delta_2$.
\end{defi}

Lattices in nilpotent Lie groups are described by Malcev's theorems, see Section \ref{lattices}  and for more details see \cite{Cor} or \cite[Section II]{Rag}. These theorems say that a nilpotent Lie group $G$ admits a lattice if and only if its Lie algebra $\mathfrak{g}$ admits a $\mathbb{Q}$-form, that is a rational subalgebra $\mathfrak{g}_0$ such that $\mathfrak{g}_0\otimes\mathbb{R}$ is isomorphic to $\mathfrak{g}$. Furthermore abstract commensurability classes of lattices in $G$ are in correspondence with $\mathbb{Q}$-isomorphism classes of $\mathbb{Q}$-forms of $\mathfrak{g}$. To be more precise let us remember that for a simply connected nilpotent Lie group the exponential map from its Lie algebra is a diffeomorphism and we will call $\log$ its inverse. Then from a lattice $\Gamma$ in $G$, $\spn_\mathbb{Q}\{\log\Gamma\}$ provides a $\mathbb{Q}$-form of $\mathfrak{g}$. On the other hand let $\mathfrak{g}$ be a Lie algebra with a basis with respect to which the structure constants are rational and let $\mathfrak{g}_0$ be the $\mathbb{Q}$-span of this basis. If $\mathcal{L}$ is any lattice of $\mathfrak{g}$ contained in $\mathfrak{g}_0$ the group generated by $\exp\mathcal{L}$ is a lattice in $G$.
The main result of this section is then the following.

\begin{theo}\label{ClassLattic}
	The list given in  Appendix \ref{LieBracket} with $\mathbb{K}=\mathbb{Q}$ is a complete list without repetition of abstract commensurability classes of nilpotent crystallographic subgroups of  $\mathrm{U}(3,1)\ltimes \mathbb{C}^{3+1}$.
\end{theo}

Finally in Section \ref{topology} we give some topological considerations about the manifolds $\Gamma \backslash \mathfrak{a}(\mathbb{C}^{n+1})$ that are virtually nilpotent, namely that they are finitely covered by torus bundles over tori.

\section*{Acknowledgement}
The author would like to warmly thank her thesis advisors Vincent Koziarz and Pierre Mounoud. Each one contributed with his style to this article. Without the several hours of discussion they dedicated and their encouragement this article would not have seen the light of day. Also she would like to thank Yves Cornulier who introduced her to the concept of Carnot Lie algebras.

\section{Unipotent simply transitive groups of Hermite-Lorentz affine motion}\label{UnSimTr}

In this section we look at unipotent subgroups of $\mathrm{U}(n,1)\ltimes \mathbb{C}^{n+1}$ acting simply transitively on $\mathfrak{a}(\mathbb{C}^{n+1})$. The reason why we have the unipotent hypothesis, that simplifies our study, is that, as we will see in the following section, the study of crystallographic groups will be reduced to the study of lattices in unipotent simply transitive Lie groups. \\

Let $V$ be a complex vector space of dimension $n+1$ endowed with a Hermitian form $h$ of signature $(n,1)$, such an Hermitian form will be called Hermite-Lorentz. Also we denote by $\mathfrak{a}(V)$ the affine space associated to $V$ and by $\mathrm{Aff}(V)$ the group of affine transformations of $\mathfrak{a}(V)$. Remember that the affine group can be seen as a linear group as follows 
\[\mathrm{Aff}(V)=\left\{\begin{pmatrix}
A & v \\ 0 & 1 
\end{pmatrix}\ |\ A\in\mathrm{GL}(V),v\in V\right\}.\]
Denote furthermore by $L:\mathrm{Aff}(V) \rightarrow \mathrm{GL}(V)$ the homomorphism that associates to each affine transformation its linear part, then its kernel consists of pure translations and let us denote it by 
\[T=\left\{\begin{pmatrix}
\mathrm{Id} & v \\ 0 & 1 
\end{pmatrix}\ |\ v\in V\right\}.\] 
For a subgroup $G\le \mathrm{Aff}(V)$, let us denote $T_G=G\cap T$ the subgroup of pure translations in $G$.\\
Let furthermore $\mathrm{U}(n,1)$ be the group of linear transformations of $V$ that preserve $h$ and $\mathcal{H}(n,1)$ the group of affine transformations of $\mathfrak{a}(V)$ whose linear part is in $\mathrm{U}(n,1)$. 

\begin{rema}
Let us notice that the pair $(V,h)$ where $V$ is a complex vector space and $h$ is a Hermite-Lorentz form defined on it is equivalent to the triple $(V(\mathbb{R}),J,\langle \cdot,\cdot\rangle)$. Here $V(\mathbb{R})$ is the associated real vector space, $J:V(\mathbb{R})\rightarrow V(\mathbb{R})$ is a complex structure on $V(\mathbb{R})$, i.e.~a linear endomorphism such that $J^2=-\mathrm{Id}$ and $\langle \cdot,\cdot\rangle$ is a pseudo-Riemannian metric on $V(\mathbb{R})$ of signature $(2n,2)$ such that, if $\omega(u,v):=\langle u,Jv\rangle=-\langle Ju,v\rangle $, we have 
\[h(u,v)=\langle u,v\rangle+i\omega(u,v).\]
\end{rema}

Let us also fix the following notation: for $z\in\mathbb{C}$ we denote by $\Re(z)$ and $\Im(z)$ its real and imaginary part respectively.

\begin{lemm}\label{FixNullVect}
A unipotent subgroup of $\mathrm{U}(n,1)$ fixes a null vector for $h$.
\end{lemm}
\begin{proof}
Let $U$ be a unipotent subgroup of $\mathrm{U}(n,1)$. From Engel's theorem we know that $U$ fixes a vector $v_0\in V$. If $v_0$ is timelike then $U$ fixes $v_0^\perp$ that is spacelike, hence it is contained in $\mathrm{U}(n)$, but a unipotent unitary matrix is trivial hence $U$ would be trivial. If instead $v_0$ were spacelike then, if $U$ does not fix a null vector in $v_0^\perp$, it would have to fix a spacelike vector $v_1$ in $v_0^\perp$. This implies that $U$ preserves the Hermite-Lorentz space $\spn \{v_0,v_1\}^\perp$. By induction we get that then $U$ is contained in $\mathrm{U}(1,1)$ and we can see that a unipotent element in $\mathrm{U}(1,1)$ is trivial.
\end{proof}

\begin{lemm}\label{unipotent}
We can choose a basis for $V$ with respect to which a maximal unipotent subgroup of $\mathrm{U}(n,1)$ is written as
\[\mathcal{U}=\left\{\begin{pmatrix}
1 & -\overline{v}^t & -\frac{1}{2}\overline{v}^tv+ib \\
0 & \mathrm{Id}_{n-1} & v \\
0 & 0 & 1
\end{pmatrix}\ |\ v\in \mathbb{C}^{n-1}, b\in\mathbb{R}\right\}.\]
\end{lemm}
\begin{proof}
Let $\mathcal{U}$ be a maximal unipotent subgroup of $\mathrm{U}(n,1)$.  We might choose the fixed null vector $v_0$ of $\mathcal{U}$ from Lemma \ref{FixNullVect} to be the first vector of a base of $V$. Let $(r,u,s)\in\mathbb{C}\times\mathbb{C}^{n-1}\times\mathbb{C}$ be the coordinates with respect to this basis. Then we write the quadratic form $q$ associated to the Hermitian form $h$ as 
\[q(r,u,s)=r\overline{s}+\overline{r}s+u^t\overline{u}.\]
With respect to this basis we can write an element of $\mathcal{U}$ as $A=\begin{pmatrix}
1 & v^t & a \\
0 & M & w \\
0 & 0 & 1
\end{pmatrix}$ with $v,w\in \mathbb{C}^{n-1}$, $a\in\mathbb{C}$ and $M$ an upper triangular matrix with $1$'s on the diagonal.
Imposing the condition $\bar{A}^tHA=H$, where $H$ is the matrix associated to $h$, we get that $A$ has the desired form.
\end{proof}

\begin{lemm}\label{Parabolic}
We can choose a basis for $V$ with respect to which a parabolic subgroup of $\mathrm{U}(n,1)$ is written as \[\mathcal{P}=\left\{\begin{pmatrix}
\overline{\lambda} & -\overline{v}^t & a \\
0 & \sigma & \lambda^{-1}\sigma v \\
0 & 0 & \lambda^{-1}
\end{pmatrix}\ |\ \lambda\in\mathbb{C}^*, \sigma\in \mathrm{U}(n-1), v\in \mathbb{C}^{n-1}, \Re(\overline{\lambda a})=-\frac{1}{2}\overline{v}^tv\right\}.\]
A minimal parabolic subgroup of $\mathrm{U}(n,1)$, that is a Borel subgroup $\mathcal{B}$, is given by elements $A\in\mathcal{P}$ with $\sigma=\text{diag}(e^{i\theta_1},\ldots,e^{i\theta_{n-1}})$.
\end{lemm}
\begin{proof}
By definition a Borel subgroup is a solvable subgroup of $\mathrm{U}(n,1)$. By Lie's theorem this implies that all the elements of $\mathcal{B}$ have a common eigenvector. As in Lemma \ref{FixNullVect} we can see that this common vector is isotropic and hence we may choose it to be the first vector of a basis. With respect to this basis we write the quadratic form associated to $h$ as $q(r,u,s)=r\overline{s}+\overline{r}s+u^t\overline{u}$ with $(r,u,s)\in\mathbb{C}\times\mathbb{C}^{n-1}\times\mathbb{C}$. Let us apply again Lie's theorem to the orthogonal space of the first vector of the basis. After noticing that a solvable subgroup of $\mathrm{U}(n-1)$ is abelian and imposing the equation $\bar{A}^tHA=H$ where $H$ is the matrix associated to $q$, we find the desired form. Finally an element $A$ in a parabolic subgroup should have a block upper triangular form. 
\end{proof}

\begin{prop}\label{UnipSimplTrans}
Let $U$ be a unipotent subgroup of $\mathcal{H}(n,1)$ that is acting simply transitively and affinely on $\mathfrak{a}(V)$. Then $U=\exp(\mathfrak{u})$ where $\mathfrak{u}$ is a nilpotent Lie algebra. In suitable coordinates $(r,u,s)\in V=\mathbb{C}\times W\times \mathbb{C}$, where $W=\mathbb{C}^{n-1}$, the quadratic form $q$ associated to the Hermitian form $h$ is given by
\[q(r,u,s)=r\overline{s}+\overline{r}s+u^t\overline{u}\]
and $\mathfrak{u}=\mathfrak{u}(\gamma_2,\gamma_3,b_2,b_3)$ has the expression
\[\left\{\begin{pmatrix}
0 & -(\overline{\gamma_2(u)}^t+\overline{\gamma_3(s)}^t) & i(b_2(u)+b_3(s)) & r \\
0 & 0 & \gamma_2(u)+\gamma_3(s) & u \\
0 & 0 & 0 & s \\
0 & 0 & 0 & 0
\end{pmatrix}\ |\ (r,u,s)^t\in\mathbb{C}\times W\times \mathbb{C}\right\}\]
where 
\begin{enumerate}
\item $\gamma_2:W\rightarrow W$ is an $\mathbb{R}$-linear map such that \[\Ima\gamma_2\oplus J\Ima\gamma_2\subseteq\ker\gamma_2\text{ and } \omega(\Ima\gamma_2,\Ima\gamma_2)=0,\]\label{cond g_2}
\item $\gamma_3:\mathbb{C}\rightarrow W$ is an $\mathbb{R}$-linear map such that $\gamma_3(is)-J\gamma_3(s)\in\ker\gamma_2$ for all $s\in\mathbb{C}$, \label{cond w_0}
\item $b_2:W\rightarrow \mathbb{R}$ is an $\mathbb{R}$-linear map such that\[b_2(s\gamma_2(u))=2\omega(\gamma_2(u),\gamma_3(s))\text{ and }\] \[b_2(\gamma_3(is)-J\gamma_3(s))=2\omega(\gamma_3(is),\gamma_3(s)) \text{ for all } u\in W,s\in\mathbb{C},\] \label{cond b_2}
\item $b_3:\mathbb{C}\rightarrow \mathbb{R}$ is an $\mathbb{R}$-linear map.
\end{enumerate}
\end{prop}
\begin{proof}
Since $U$ is acting simply transitively on $\mathfrak{a}(V)$ it is simply connected, hence $U=\exp(\mathfrak{u})$ where $\mathfrak{u}$ is its Lie algebra.
The linear part of $U$ can be conjugated to be as in Lemma \ref{unipotent}. Equivalently we can find coordinates $(r,u,s)\in V=\mathbb{C}\times W\times \mathbb{C}$, where $W=\mathbb{C}^{n-1}$, with respect to which the quadratic form $q$ associated to the Hermitian form $h$ reads as
\[q(r,u,s)=r\overline{s}+\overline{r}s+u^t\overline{u}\]
and the Lie algebra of the linear part of $U$ is
\[L(\mathfrak{u})=\left\{\begin{pmatrix}0 & -\overline{\gamma}^t & ib \\ 0 & 0 & \gamma \\
0 & 0 & 0\end{pmatrix}\ |\ \gamma\in W, b\in\mathbb{R}\right\}.\]

Since $\mathfrak{u}$ is the Lie algebra of $U$ that is acting simply transitively on $\mathfrak{a}(V)$ there exist two $\mathbb{R}$-linear maps
\[\gamma:\mathbb{C}\times W\times \mathbb{C}\rightarrow \mathbb{C}^{n-1}\text{ and }\]
\[b:\mathbb{C}\times W\times \mathbb{C}\rightarrow \mathbb{R},\]
such that 
\[\mathfrak{u}=\left\{\begin{pmatrix}0 & -\overline{\gamma(v)}^t & ib(v) & r\\ 0 & 0 & \gamma(v) & u \\
0 & 0 & 0 & s \\ 
0 & 0 & 0 & 0\end{pmatrix}\ |\ v=(r,u,s)\in\mathbb{C}\times W\times \mathbb{C}\right\}.\]
Let us write $(A(v),v)$ for an element in $\mathfrak{u}$ where $A(v)$ denotes the linear part and $v$ the translation part.
If we compute the commutator bracket of two elements in $\mathfrak{u}$ we get 
\[[(A(v),v),(A(v'),v')]=\begin{pmatrix}
0 & 0 & \overline{\gamma(v')}^t\gamma(v)-\overline{\gamma(v)}^t\gamma(v') & \overline{\gamma(v')}^tu-\overline{\gamma(v)}^tu'+i(s'b(v)-sb(v')) \\
0 & 0 & 0 & s'\gamma(v)-s\gamma(v') \\
0 & 0 & 0 & 0 \\
0 & 0 & 0 & 0
\end{pmatrix}.\]
Since $\mathfrak{u}$ is a Lie algebra the commutator bracket between two of its element should belong to $\mathfrak{u}$. Hence we must have $\gamma(v'')=0$ and $b(v'')=\overline{\gamma(v')}^t\gamma(v)-\overline{\gamma(v)}^t\gamma(v')=2\Im(\overline{\gamma(v')}^t\gamma(v))$ where \[v''=(\overline{\gamma(v')}^tu-\overline{\gamma(v)}^tu'+i(s'b(v)-sb(v')),s'\gamma(v)-s\gamma(v'),0).\]
So let us write $\gamma(v)=\gamma_1(r)+\gamma_2(u)+\gamma_3(s)$ and $b(v)=b_1(r)+b_2(u)+b_3(s)$.
\begin{claim}\label{lem1}
We have $\gamma_1= 0$, $\Ima\gamma_2\oplus J\Ima\gamma_2\subseteq\ker\gamma_2$ and that property \ref{cond w_0} holds.
\end{claim}
\begin{proof}
Let $s=s'=0$ then $\gamma(v'')=\gamma_1(\overline{\gamma(v')}^tu-\overline{\gamma(v)}^tu')=0$. Now let $u'=0$ we get $\gamma(v'')=\gamma_1(\overline{\gamma_1(r')}^tu)=0$ letting $r'$ and $u$ vary we see that $\gamma_1=0$. Now take $s=0$ we have $\gamma(v'')=\gamma_2(s'\gamma(v))=\gamma_2(s'\gamma_2(u))=0$ hence if we let $s'$ be real we obtain $(\gamma_2)^2=0$ and if we let $s'$ be purely imaginary we obtain $\gamma_2(J\gamma_2)=0$. Now the equation $\gamma(v'')=0$ becomes $\gamma_2(s'\gamma_3(s))-\gamma_2(s\gamma_3(s'))=0$. Let $s=is'$ and the lemma follows.
\end{proof}

\begin{claim}\label{lem2}
We have $b_1=0$, $\omega(\Ima\gamma_2,\Ima\gamma_2)=0$ and that property \ref{cond b_2} holds.
\end{claim}
\begin{proof}
Let $u=u'=0$ then $b(v'')=b_1(i(s'b(v)-sb(v')))+b_2(s'\gamma_3(s)-s\gamma_3(s'))=2\Im(\overline{\gamma_3(s')}^t\gamma_3(s))$. Now let $s=0$ we get $b(v'')=b_1(is'b(v))=b_1(is'b_1(r))=0$ hence $b_1=0$. Looking again at the equation $b(v'')=2\Im(\overline{\gamma(v')}^t\gamma(v))$ and letting $s=s'=0$ we get $\Im(\overline{\gamma_2(u')}^t\gamma_2(u))=0$ for all $u,u'$. So finally the equation $b(v'')=2\Im(\overline{\gamma(v')}^t\gamma(v))$ becomes 
\[b_2(s'\gamma_2(u)-s\gamma_2(u')+s'\gamma_3(s)-s\gamma_3(s'))=2\Im(\overline{\gamma_2(u')}^t\gamma_3(s)+\overline{\gamma_3(s')}^t\gamma_2(u)+\overline{\gamma_3(s')}^t\gamma_3(s)).\]
Now take $s=0$ we get $b_2(s'\gamma_2(u))=2\Im(\overline{\gamma_3(s')}^t\gamma_2(u))$. Finally if we consider $u=u'=0$ we get $b_2(s'\gamma_3(s)-s\gamma_3(s'))=2\Im(\overline{\gamma_3(s')}^t\gamma_3(s))$, now if we let $s=is'$ the lemma follows.
\end{proof}

We can notice that these conditions are also sufficient in order to have $\gamma(v'')=0$ and $b(v'')=2\Im(\overline{\gamma(v')}^t\gamma(v))$. The proof of Proposition \ref{UnipSimplTrans} is now complete.
\end{proof}

Let us decompose $W$ as a real vector space as follows \begin{equation} \label{decomposition} W=\Ima\gamma_2\oplus J\Ima\gamma_2\oplus S\oplus T\end{equation} where $S\oplus T$ is orthogonal to $\Ima\gamma_2\oplus J\Ima\gamma_2$ with respect to $h_{|W}$, $T$ is orthogonal to $S$ with respect to $\langle \cdot,\cdot\rangle_{|W}$ and \[\ker\gamma_2=\Ima\gamma_2\oplus J\Ima\gamma_2 \oplus S.\]
Write $\pi_i$ for the projection of $\ker\gamma_2$ on the $i-th$ factor with $i=1,2,3$.

\begin{prop}\label{uppper bound k}
If we think $\gamma_2$ as a real linear map we have an upper bound $\mathrm{rank}(\gamma_2)\le\frac{2n-2}{3}$.
\end{prop}
\begin{proof}
We have $\Ima\gamma_2\cap J\Ima\gamma_2=\{0\}$. This is because if $v_1=Jv_2$ with $v_1,v_2\in \Ima\gamma_2$ then $\langle v_1,v_1\rangle=\langle v_1,Jv_2\rangle=\omega(v_1,v_2)=0$ then $v_1=0$. So finally, since both $\Ima\gamma_2$ and $J\Ima\gamma_2$ are contained in $\ker\gamma_2$, we have $2\mathrm{rank}(\gamma_2)\le\dim \ker\gamma_2$ then $3\mathrm{rank}(\gamma_2)\le \mathrm{rank}(\gamma_2)+\dim \ker\gamma_2= 2n-2$.
\end{proof}

\begin{rema}\label{pi_3}
We can notice that $\pi_3(\gamma_3(is)-J\gamma_3(s))=0$ for some $s\in\mathbb{C}$, $s\ne 0$, if and only if $\pi_3(\gamma_3(is)-J\gamma_3(s))=0$ for all $s\in\mathbb{C}$, $s\ne 0$. Indeed it suffices to notice that $\gamma_3(is)-J\gamma_3(s)=\bar{s}(\gamma_3(i)-J\gamma_3(1))$ for all $s\in\mathbb{C}$. Hence $\gamma_3(is)-J\gamma_3(s)\in\Ima\gamma_2\oplus J\Ima\gamma_2$ if and only if $\gamma_3(i)-J\gamma_3(1)$ does.
\end{rema}

\begin{prop}\label{b2}
If there exists $s\in\mathbb{C}$, $s\ne0$, such that $\pi_3(\gamma_3(is)-J\gamma_3(s))=0$ then $\Ima\gamma_3\subseteq \Ima\gamma_2$ and hence $b_2$ is $0$ on $\Ima\gamma_2\oplus J\Ima\gamma_2$.
\end{prop}
\begin{proof}
From Remark \ref{pi_3} if for some $s\in\mathbb{C}$, $s\ne 0$ we have $\pi_3(\gamma_3(is)-J\gamma_3(s))=0$ then also $\pi_3(\gamma_3(i)-J\gamma_3(1))=0$. Hence let $w_0=\gamma_3(i)-J\gamma_3(1)\in\ker\gamma_2$. If $\pi_3(w_0)=0$ then write  $w_0=w_1+Jw_2$ with $w_1,w_2\in \Ima\gamma_2$. Condition (\ref{cond b_2}) of Proposition \ref{UnipSimplTrans} implies that $2\omega(\gamma_3(i),\gamma_3(1))=b_2(w_0)=b_2(w_1+Jw_2)=2\omega(w_1,\gamma_3(1))+2\omega(w_2,\gamma_3(i))$ hence $\omega(w_1-\gamma_3(i),\gamma_3(1))+\omega(w_2,\gamma_3(i))=0$ since $w_1-\gamma_3(i)=-J(w_2+\gamma_3(1))$ and $\gamma_3(i)=w_1+Jw_2+J\gamma_3(1)$ substituting them in the previous expression we obtain $\omega(J\gamma_3(1),\gamma_3(1))+\omega(Jw_2,w_2)+2\omega(Jw_2,\gamma_3(1))=0$, i.e.~$\|\gamma_3(1)\|^2+2\langle w_2,\gamma_3(1)\rangle+\|w_2\|^2=0$. This implies $\gamma_3(1)=-w_2$, in other words $\Ima\gamma_3\subseteq \Ima\gamma_2$ hence $b_2$ is $0$ on $\Ima\gamma_2\oplus J\Ima\gamma_2$. 
\end{proof}

As an abstract algebra $\mathfrak{u}(\gamma_2,\gamma_3,b_2,b_3)$ can be described as the real vector space $\mathbb{C}\times W\times \mathbb{C}$, with $W=\mathbb{C}^{n-1}$, with Lie brackets
\[ [(r,u,s),(r',u',s')]= ( h(u,\gamma(u',s'))-h(u',\gamma(u,s))+i(s'b(u,s)-sb(u',s')),  s'\gamma(u,s)-s\gamma(u',s'),0)\]
and with $\gamma:W\times \mathbb{C}\rightarrow W$ and $b:W\times \mathbb{C}\rightarrow \mathbb{R}$ linear maps, $\gamma=\gamma_2+\gamma_3$ and $b=b_2+b_3$ satisfying the conditions of Proposition \ref{UnipSimplTrans}.

\begin{rema}
From Remark \ref{pi_3} we saw that the condition $\pi_3(\gamma_3(is)-J\gamma_3(s))=0$ does not depend on the basis we have chosen for $\mathbb{C}$ as a real vector space. Hence we will write this condition as $\pi_3(\gamma_3(i\xi)-J\gamma_3(\xi))=0$ where $\{\xi,i\xi\}$ is any basis of $\mathbb{C}$ as a real vector space.
\end{rema}

\begin{defi}
We say that a Lie algebra $\mathfrak{g}$ is a $k$-step nilpotent Lie algebra if $\mathcal{C}^{k+1}\mathfrak{u}=\{0\}$ and $\mathcal{C}^k\mathfrak{u}\not=\{0\}$ where $\mathcal{C}^k\mathfrak{u}=[\mathfrak{u},\mathcal{C}^{k-1}\mathfrak{u}]$ and $\mathcal{C}^1\mathfrak{u}=\mathfrak{u}$.
\end{defi}

\begin{prop}
The Lie algebras $\mathfrak{u}(\gamma_2,\gamma_3,b_2,b_3)$ are at most $3$-step nilpotent.
The lower central series looks like \[\mathfrak{u}(\gamma_2,\gamma_3,b_2,b_3)\supseteq \mathcal{C}^2\mathfrak{u}(\gamma_2,\gamma_3,b_2,b_3)\supseteq\mathcal{C}^3\mathfrak{u}(\gamma_2,\gamma_3,b_2,b_3)\supseteq\{0\}\] where  $\mathcal{C}^2\mathfrak{u}(\gamma_2,\gamma_3,b_2,b_3)\subseteq\mathbb{C}\oplus\mathbb{C}\Ima(\gamma_2)\oplus\mathbb{R}\pi_3(\gamma_3(i\xi)-J\gamma_3(\xi))$, $\mathcal{C}^3\mathfrak{u}(\gamma_2,\gamma_3,b_2,b_3)\subseteq \mathbb{C}$.
\end{prop}

\begin{rema}
More precisely if $\gamma_2\ne 0$ then the Lie algebras $\mathfrak{u}(\gamma_2,\gamma_3,b_2,b_3)$ are $3$-step nilpotent and indeed we have equalities $\mathcal{C}^2\mathfrak{u}(\gamma_2,\gamma_3,b_2,b_3)=\mathbb{C}\oplus\mathbb{C}\Ima(\gamma_2)\oplus\mathbb{R}\pi_3(\gamma_3(i\xi)-J\gamma_3(\xi))$ and  $\mathcal{C}^3\mathfrak{u}(\gamma_2,\gamma_3,b_2,b_3)=\mathbb{C}$. Instead if $\gamma_2= 0$ then all possibilities can occur, see Appendix \ref{LieBracket} for the case $n=3$.
\end{rema}

\subsection{Classification up to conjugation}
We now write the classification of the unipotent subgroup $U(\gamma_2,\gamma_3,b_2,b_3)=\exp \mathfrak{u}(\gamma_2,\gamma_3,b_2,b_3)$ of $\mathcal{H}(n,1)$ up to conjugation. This will be useful for the proof of Proposition \ref{nonunip2}.

\begin{prop}\label{conj}
There exists $g\in\mathcal{H}(n,1)$ such that $gU(\gamma_2,\gamma_3,b_2,b_3)g^{-1}=U(\gamma_2',\gamma_3',b_2',b_3')$ if and only if there exist $\lambda\in\mathbb{C}^*, \sigma\in \mathrm{U}(n-1), v\in W, s_1\in \mathbb{C}$ such that 
\begin{enumerate}
\item $\gamma_2'(u)=\lambda\sigma\gamma_2(\sigma^{-1}u)$, \item $\gamma_3'(s)=\lambda\sigma\gamma_3(\lambda s)+\lambda\sigma\gamma_2
(\tilde{s})$, 
\item $b_2'(u)=|\lambda|^2b_2(\sigma^{-1}u-\lambda s_1\gamma_2(\sigma^{-1}u))-2\omega(\lambda\gamma_2(\sigma^{-1}u),v)$, 
\item $b'_3(s)=|\lambda|^2b_3(\lambda s)+|\lambda|^2b_2(\tilde{s}-\lambda s_1\gamma_2(\tilde{s}))-2\omega(\lambda (\gamma_2(\tilde{s})+\gamma_3(\lambda s)),v)$
\end{enumerate} 
where $\tilde{s}=\lambda s_1\gamma_3(\lambda s)-sv$.
\end{prop}
\begin{proof}
Since a parabolic subgroup is self normalizing if there exist $L(g)\in \mathrm{U}(n,1)$ and $L(h)\in \mathcal{U}$ such that $L(ghg^{-1})\in\mathcal{U}$ then $L(g)\in\mathcal{P}$. Hence we may assume that
\[g=\begin{pmatrix}
\overline{\lambda} & -\overline{v}^t & a & r_1 \\
0 & \sigma & \lambda^{-1}\sigma v & u_1 \\
0 & 0 & \lambda^{-1} & s_1 \\
0 & 0 & 0 & 1
\end{pmatrix}\] 
with $\lambda\in\mathbb{C}^*, \sigma\in \mathrm{U}(n-1), v\in W, a\in\mathbb{C}$ and $\overline{\lambda a} +\lambda a=-\overline{v}^tv$.\\
Notice that the inverse of the linear part of $g$ is of the form 
\[L(g)^{-1}=\begin{pmatrix}
\overline{\lambda}^{-1} & \overline{\lambda}^{-1}\overline{v}^t\overline{\sigma}^t & \overline{a} \\
0 & \overline{\sigma}^t & -v \\
0 & 0 & \lambda
\end{pmatrix}.\] 
Since $U(\gamma_2,\gamma_3,b_2,b_3)=\exp(\mathfrak{u}(\gamma_2,\gamma_3,b_2,b_3))$ and 
\[g\exp(X)g^{-1}=\exp(gXg^{-1})\] we can work on the level of the Lie algebra. 
We let 
\[X=\begin{pmatrix}
0 & -(\overline{\gamma_2(u)}^t+\overline{\gamma_3(s)}^t) & i(b_2(u)+b_3(s)) & r \\
0 & 0 & \gamma_2(u)+\gamma_3(s) & u \\
0 & 0 & 0 & s \\
0 & 0 & 0 & 0
\end{pmatrix}\in \mathfrak{u}(\gamma_2,\gamma_3,b_2,b_3),\]
then the linear part of $gXg^{-1}$ is 
\[\begin{pmatrix}
0 & -\overline{\lambda}(\overline{\gamma_2(u)}^t+\overline{\gamma_3(s)}^t)\overline{\sigma}^t & \overline{\lambda}(\overline{\gamma_2(u)}^t+\overline{\gamma_3(s)}^t)v+i|\lambda|^2(b_2(u)+b_3(s))-\lambda \overline{v}^t(\gamma_2(u)+\gamma_3(s)) \\
0 & 0 & \lambda\sigma(\gamma_2(u)+\gamma_3(s)) \\
0 & 0 & 0
\end{pmatrix}\]
and the translation part of $gXg^{-1}$ is 
\[\begin{pmatrix}
* \\ -\lambda s_1\sigma(\gamma_2(u)+\gamma_3(s))+\sigma u+\lambda^{-1}s\sigma v \\ 
\lambda^{-1}s
\end{pmatrix}.\]
Let \[(r',u',s')=(*,  -\lambda s_1\sigma(\gamma_2(u)+\gamma_3(s))+\sigma u+\lambda^{-1}s\sigma v, \lambda^{-1}s).\]
When $s'=0$ then $s=0$ and we get $\gamma_2'(\sigma u-\lambda s_1\sigma(\gamma_2(u)))=\lambda\sigma\gamma_2(u)$. Take $u=\gamma_2(\tilde{u})$ for some $\tilde{u}\in W$ then we get $\gamma_2'(\sigma(\gamma_2(\tilde{u})))=0$ hence $\gamma_2'(\sigma u)=\lambda\sigma\gamma_2(u)$ that is $\gamma_2'=\lambda\sigma\gamma_2\sigma^{-1}$. \\
When $u'=0$, $u=\lambda s_1 \gamma_2(u)+\lambda s_1\gamma_3(s)-\lambda^{-1}sv$ and we get $\gamma_3'(\lambda^{-1}s)=\lambda\sigma\gamma_2(u)+\lambda\sigma\gamma_3(s)$, hence $\gamma_3(\lambda^{-1}s)=\lambda\sigma\gamma_3(s)+\lambda\sigma\gamma_2
(\lambda s_1\gamma_3(s)-\lambda^{-1}sv)$ that is $\gamma_3'(s)=\lambda\sigma\gamma_3(\lambda s)+\lambda\sigma\gamma_2
(\lambda s_1\gamma_3(\lambda s)-sv)$. \\
Now consider the equation
\[|\lambda|^2b_2(u)+|\lambda|^2b_3(s)-2\Im(\lambda\overline{v}^t(\gamma_2(u)+\gamma_3(s)))=b_2'(-\lambda s_1\sigma(\gamma_2(u)+\gamma_3(s))+\sigma u+\lambda^{-1}s\sigma v)+b_3'(\lambda^{-1}s).\]
Let $s=0$ then we are left with \[|\lambda|^2b_2(u)-2\Im(\lambda\overline{v}^t\gamma_2(u))=b_2'(-\lambda s_1\sigma\gamma_2(u)+\sigma u).\]
Take $u=\gamma_2(\tilde{u})$ then we have $|\lambda|^2b_2(\gamma_2(\tilde{u}))=b_2'(\sigma \gamma_2(\tilde{u}))$. Hence the equation becomes $b_2'(\sigma u)=|\lambda|^2b_2(u-\lambda s_1\gamma_2(u))-2\Im(\lambda\overline{v}^t\gamma_2(u))$.\\
Let $u=0$ then the equation becomes 
\[|\lambda|^2b_3(s)-2\Im(\lambda\overline{v}^t\gamma_3(s))=b_2'(-\lambda s_1\sigma\gamma_3(s)+\lambda^{-1}s\sigma v)+b_3'(\lambda^{-1}s),\]
substituting $b_2'$ we get \[|\lambda|^2b_3(s)-2\Im(\lambda\overline{v}^t\gamma_3(s))=|\lambda|^2b_2(-\lambda s_1\gamma_3(s)+\lambda^{-1}s v-\lambda s_1\gamma_2(-\lambda s_1\gamma_3(s)+\lambda^{-1}s v))\]\[-2\Im(\lambda\overline{v}^t\gamma_2(-\lambda s_1\gamma_3(s)+\lambda^{-1}s v))+b_3'(\lambda^{-1}s).\]
Hence the proposition follows.
\end{proof}

\begin{rema}\label{semplification}
We saw in Proposition \ref{b2} that if $\pi_3(\gamma_3(i\xi)-J\gamma_3(\xi))=0$ then $b_2$ is $0$ on $\Ima\gamma_2\oplus J \Ima\gamma_2$. On the other hand when $\pi_3(\gamma_3(i\xi)-J\gamma_3(\xi))\ne 0$, if $\gamma_2\ne 0$, using Proposition \ref{conj} we can conjugate the Lie group $U(\gamma_2,\gamma_3,b_2,b_3)$ so that $\pi_2(\gamma_3(\xi))=0$. In other words we can always suppose that $b_2$ is $0$ on $\Ima\gamma_2$.
\end{rema}

\section{Properly discontinuous and cocompact groups of Hermite-Lorentz affine motions}\label{ProperlyDisc}

In this section we go back to our original question about crystallographic subgroups of $\mathrm{U}(n,1)\ltimes \mathbb{C}^{n+1}$ and prove the main result of this section, namely Theorem \ref{cryst}.

\begin{defi}
Consider the following subgroups of the Borel subgroup $\mathcal{B}$ of $\mathrm{U}(n,1)$
\[\mathcal{D}=\left\{\begin{pmatrix}
\overline{\lambda} & 0 & 0 \\
0 & \sigma & 0 \\
0 & 0 & \lambda^{-1}
\end{pmatrix}\ |\ \sigma=\text{diag}(e^{i\theta_1},\ldots,e^{i\theta_{n-1}})\right\}\]
and 
\[\hat{\mathcal{B}}=\left\{\begin{pmatrix}
\lambda & -\overline{v}^t & a \\
0 & \sigma & \lambda\sigma v \\
0 & 0 & \lambda
\end{pmatrix}\ |\ |\lambda|^2=1,\sigma=\text{diag}(e^{i\theta_1},\ldots,e^{i\theta_{n-1}}), \Re(\overline{\lambda}  a)=-\frac{1}{2}\overline{v}^tv\right\}.\]
\end{defi}

\begin{defi}
The \emph{unipotent radical} of a linear group $G$ is the set of all unipotent elements in the largest connected solvable normal subgroup of $G$. It may be characterised as the largest connected unipotent normal subgroup of $G$.
\end{defi}

\begin{prop}\label{nonunip1}
Let $H\le \mathcal{H}(n,1)$ be a subgroup that acts simply transitively on $\mathfrak{a}(V)$ then $H$ is $\mathcal{H}(n,1)$-conjugate to a subgroup whose linear part is in $\mathcal{B}$.
\end{prop}
\begin{proof}
Since $H$ acts simply transitively by \cite[ Theorem I.1]{Auslander} $H$ is solvable. Then its linear part $L(H)\le U(n,1)$ is solvable and the Zariski closure of $L(H)$ is solvable as well. Thus $H$ is conjugate in $\mathcal{H}(n,1)$ to a group whose linear part is a subgroup of the Borel group $\mathcal{B}$.
\end{proof}

\begin{prop}\label{nonunip2}
Let $H\le\mathcal{H}(n,1)$ be a group acting simply transitively on $\mathfrak{a}(V)$ such that $L(H)\le \mathcal{B}$. Then either $L(H)\le \hat{\mathcal{B}}$ or $L(H)$ is $\mathcal{B}$-conjugate to a subgroup of $\mathcal{D}$.
\end{prop}
\begin{proof}
Let $U$ be the unipotent radical of the Zariski closure $\widetilde{H}$ of $H$, then by \cite[ Theorem III.1]{Auslander} $U$ also acts simply transitively on $\mathfrak{a}(V)$. Hence by Proposition \ref{UnipSimplTrans} we have that $U=U(\gamma_2,\gamma_3,b_2,b_3)$ for some $\gamma_2,\gamma_3,b_2,b_3$. By definition $U$ is a normal subgroup of $\widetilde{H}$ hence $H$ normalizes $U$. That is for all $h\in H$ we have $hU(\gamma_2,\gamma_3,b_2,b_3)h^{-1}=U(\gamma_2,\gamma_3,b_2,b_3)$. Now for every fixed $h\in H$ its linear part is of the form 
\[\begin{pmatrix} \overline{\lambda} & -\overline{v}^t & a \\ 0 & \sigma & \lambda^{-1}\sigma v \\
0 & 0 & \lambda^{-1} \end{pmatrix}\]
and it follows from Proposition \ref{conj} that $\gamma_2(u)=\lambda\sigma\gamma_2(\sigma^{-1}u)$, hence if $\gamma_2\ne0$ then $|\lambda|^2=1$. If $\gamma_2= 0$ then $\gamma_3(s)=\lambda\sigma\gamma_3(\lambda s)$ and $b_2(u)=|\lambda|^2b_2(\sigma^{-1}u)$ so if $b_2$ is non zero then $|\lambda|^2=1$ if instead $\gamma_3$ is not $0$ then $\lambda^2=1$. Finally if all of $\gamma_2,\gamma_3,b_2$ are identically $0$ then $b_3(s)=|\lambda|^2b_3(\lambda s)$ and hence if it is non zero $\lambda=1$. Hence if at least one of $\gamma_2,\gamma_3,b_2,b_3$ is non zero then $|\lambda|^2=1$ and $L(H)\le \hat{\mathcal{B}}$, otherwise all of $\gamma_2,\gamma_3,b_2,b_3$ are $0$ and hence $L(U)=\{Id\}$. Since $L(U)$ is the unipotent radical of $\widetilde{L(H)}$ then this implies that $\widetilde{L(H)}$ is reductive and being solvable and connected it is a torus. Hence $L(H)$ is $\mathcal{B}$-conjugate to a subgroup of $\mathcal{D}$.
\end{proof}

\begin{defi}
The \emph{nilradical}, $N$, of a Lie group $G$ is the largest connected normal nilpotent subgroup of $G$. Analogously the nilradical, $\mathfrak{n}$, of a Lie algebra $\mathfrak{g}$ is the largest nilpotent ideal of $\mathfrak{g}$. Then we have that $N=\exp(\mathfrak{n})$.
\end{defi}

\begin{lemm}\label{nilpotentRad}
Let $H\le \mathcal{H}(n,1)$ be a group acting simply transitively on affine space and let $N$ be its nilradical. Consider furthermore the Zariski closure of $H$, $\widetilde{H}$, inside $\mathcal{H}(n,1)$ and let $U$ be the unipotent radical of $\widetilde{H}$, then $(H\cap U)^0=N$.
\end{lemm}
\begin{proof}
Clearly $(H\cap U)^0\subseteq N$. On the other side from \cite[Corollary III.3]{Auslander} we have that $N\subseteq U$, hence the lemma follows.
\end{proof}

\begin{lemm}\label{Mostow}
Let $H$ be a simply connected connected solvable Lie group and $N$ its nilradical. Let $\Gamma$ be a lattice in $H$ then 
$\Gamma N/N\subseteq H/N$ is discrete.
\end{lemm}
\begin{proof}
From \cite[Mostow Theorem]{Ausl2} we know that $\overline{\Gamma N}^0=N$ hence $N$ is open in $\Gamma N$ and the lemma follows.
\end{proof}

We will now prove the main technical part of Theorem \ref{cryst}. 

\begin{prop}\label{VirtNil}
Let $H\le \mathcal{H}(n,1)$ be a group acting simply transitively on $\mathfrak{a}(V)$ with $L(H)\le \widehat{\mathcal{B}}$. Let $\Gamma$ be a lattice in $H$. Then $\Gamma$ is virtually nilpotent.
\end{prop}
\begin{proof}
Let us consider $N$, the nilradical of $H$ and let $\mathfrak{h}$ and $\mathfrak{n}$ be the corresponding Lie algebras. From \cite[Corolary 3.5 and Corolary 1 of Theorem 2.3]{Rag} we know that $\Gamma\cap N$ is a lattice in $N$ and that $\Gamma\cap\mathcal{C}^iN$ is a lattice in $\mathcal{C}^iN$, where $\mathcal{C}^iN$ are the groups of the lower central series of $N$. Furthermore if $U$ is the unipotent radical of the Zariski closure of $H$, we have seen in Lemma \ref{nilpotentRad} that $(H\cap U)^0=N$. Since by \cite[ Theorem III.1]{Auslander} $U$ also acts simply transitively on $\mathfrak{a}(V)$ we have $N\subseteq U(\gamma_2,\gamma_3,b_2,b_3)$ for some $\gamma_2,\gamma_3,b_2,b_3$, using the terminology as in Section \ref{UnSimTr}. The general strategy of the proof is as follows: from \cite[Corollary 1.41]{Knapp} we know that $H/N$ is abelian, hence $\Gamma/\Gamma\cap N$ is abelian. We will show that the action by conjugation of $\Gamma/\Gamma\cap N$ on $\Gamma\cap N$ is virtually unipotent. This implies that $\Gamma$ is virtually nilpotent.\\
Let us start choosing a basis of $\mathfrak{h}$ that contains a basis of $\mathfrak{n}$. Since $H$ acts simply transitively on $\mathfrak{a}(V)$, $\mathfrak{h}$ inherits the Hermite-Lorentz form $h$ of $V$, we denote by $\langle\cdot ,\cdot \rangle$ its real part. We will use the symbol $\perp$ to denote the orthogonal with respect to $\langle\cdot ,\cdot \rangle$. We know from Lemma \ref{Parabolic} that $H$ preserves a complex isotropic line, let $\tau$ be a basis for it over $\mathbb{C}$. Then we choose $W_2$, a space supplementary to $\mathbb{C}\tau\cap \mathcal{C}^2\mathfrak{n}$ in $(\mathbb{C}\tau)^\perp\cap \mathcal{C}^2\mathfrak{n}$ and $W_3$ a supplementary space to $(\mathbb{C}\tau\cap\mathfrak{n})\oplus W_2$ in $(\mathbb{C}\tau)^\perp\cap \mathfrak {n} $. Then choose $W_1$ such that $W=W_1\oplus W_2\oplus W_3$ is a supplementary space to $\mathbb{C}\tau$ in $\mathbb{C}\tau^\perp$. Finally choose $\xi$ to be an isotropic vector orthogonal to $W$ such that $h(\tau,\xi)=1$. Then $\mathfrak{h}=\mathbb{C}\tau\oplus W\oplus\mathbb{C}\xi$ where $W=W_1\oplus W_2\oplus W_3$. Let $k_1=\dim W_1$, $k_2=\dim W_2$ and $k_3=\dim W_3$.
Since $H$ acts simply transitively on $\mathfrak{a}(V)$ and $L(H)\le \widehat{\mathcal{B}}$ we can write the elements in $\mathfrak{h}$ as 
 \[v=(r,u_1,u_2,u_3,s)=\begin{pmatrix} i\varphi(v) & -\overline{\gamma}^{(1)}(v)^t & -\overline{\gamma}^{(2),1}(v)^t & -\overline{\gamma}^{(2),2}(v)^t & ib(v) & r\\ 0 & \rho^{(1)}(v) & 0 & 0 &  \gamma^{(1)}(v) & u_1\\
 0 & 0 & \rho^{(2),1}(v) & 0 &  \gamma^{(2),1}(v) & u_2\\
 0 & 0 & 0 & \rho^{(2),2}(v) &  \gamma^{(2),2}(v) & u_3\\
 0 & 0 & 0 & 0 & i\varphi(v) & s\\
  0 & 0 & 0 & 0 & 0 & 0  \end{pmatrix}\]
with $\varphi:V\rightarrow \mathbb{R}$,  $\gamma=(\gamma^{(1)},\gamma^{(2),1},\gamma^{(2),2}):V\rightarrow W_1\oplus W_2\oplus W_3$, $b:V\rightarrow \mathbb{R}$ and $ \rho=(\rho^{(1)},\rho^{(2),1},\rho^{(2),2}):V\rightarrow \mathfrak{o}(k_1)\times \mathfrak{o}(k_2)\times \mathfrak{o}(k_3)$ such that $\rho=\mathrm{diag}(i\theta_1,\ldots,i\theta_{n-1})$.\\
The reason why we can write the central block, $L(\mathfrak{h})_{|W}$ corestricted to $W$, in a diagonal form that preserves our decomposition is that the adjoint action of $\mathfrak{h}$ leaves $(\mathbb{C}\tau)^\perp/\mathbb{C}\tau$ invariant. Hence its intersections with $\mathfrak{n}$ and $\mathcal{C}^2\mathfrak{n}$ are invariant as well, and furthermore this action is given by multiplication by the central block. We are also making an abuse of notation by writing in a complex form the matrix that is indeed a real linear transformation. We can do so since if a complex line in $W$, on which we act by $i\theta_j$, has a one dimensional real intersection with $\mathfrak{n}$ then $\theta_j=0$. Let us write $\lambda=e^{i\varphi}$, $\sigma=\mathrm{diag}(e^{i\theta_1},\ldots,e^{i\theta_{n-1}})$ and $\sigma^{(1)}$, $\sigma^{(2),1}$, $\sigma^{(2),2}$ the exponentials of $\rho^{(1)},\rho^{(2),1},\rho^{(2),2}$ respectively. Finally we write $g(v)$ for the element in $H$ corresponding to $v\in\mathfrak{h}$. \newline
First of all we remark that if $\lambda$ is not identically $1$ then either $\mathbb{C}\tau\subseteq\mathfrak{n}$ or $\mathbb{C}\tau\cap \mathfrak{n}=\{0\}$ and moreover either $\mathfrak{n}\subseteq\mathbb{C}\tau\oplus W_2\oplus W_3$ or there exist $\xi_1,\xi_2\in\mathfrak{n}$ such that $\xi_1\equiv \xi\mod (\mathbb{C}\tau)^\perp$ and $\xi_2\equiv i\xi \mod (\mathbb{C}\tau)^\perp$. Indeed assume that $\mathbb{C}\tau\cap\mathfrak{n}\ne \{0\}$ and let $\tau_1$ be a non trivial element of the intersection. Then we know that the element $g(\tau_1)$ is a pure translation, since elements of this form belonging to $U(\gamma_2,\gamma_3,b_2,b_3)$ are so. As $\varphi$ is not identically $0$, there exists $v\in\mathfrak{h}$ be such that $\varphi(v)\ne 0$. Computing the commutator between these two elements we find $[\tau_1,v]=-i\varphi(v)\tau_1\in\mathfrak{n}$.
Since $\varphi(v)\ne 0$ the whole $\mathbb{C}\tau\subseteq\mathfrak{n}$. In the same way we can see that if there exists an element $\xi_1\in\mathfrak{n}$, $\xi_1\notin(\mathbb{C}\tau)^\perp$, then the action of $v$ on $\xi_1$ is via a non trivial rotation in $\mathfrak{n}/((\mathbb{C}\tau)^\perp\cap\mathfrak{n})$ hence $\mathfrak{n}$ should contain another element $\xi_2$ linearly independent from $\xi_1$ and not belonging to $(\mathbb{C}\tau)^\perp$. \\
In order to prove that the action of $\Gamma/\Gamma\cap N$ on $\Gamma\cap N$ is unipotent we will prove that the action by conjugation of $H$ on each $\mathcal{C}^iN/\mathcal{C}^{i-1}N$ for $i=3,2,1$ is given by a matrix with eigenvalues of modulus $1$. This implies that when we look at the corresponding action by $(\Gamma\cap\mathcal{C}^iN)/(\Gamma\cap\mathcal{C}^{i-1}N)$ on the torus $(\mathcal{C}^iN/\mathcal{C}^{i-1}N)/(\Gamma\cap\mathcal{C}^iN/(\Gamma\cap\mathcal{C}^{i-1}N))$ and we denote by $p(x)$ the characteristic polynomial of the matrix of the action we know that all the roots of $p(x)\in\mathbb{Z}[x]$ have modulus $1$. By a theorem of Kronecker, \cite{Greiter}, all the roots of $p(x)$ are roots of unity. Then it follows that that action of $\Gamma$ is virtually unipotent. 

If $\mathcal{C}^3\mathfrak{n}\ne\{0\}$ we look at the action of $\Gamma$ on the torus $\mathcal{C}^3N/(\mathcal{C}^3N\cap \Gamma)$ by conjugation. As $\mathcal{C}^3\mathfrak{n}\subseteq\mathbb{C}\tau $ we see that this action is given by multiplication by $\lambda$ or it is trivial. Since automorphisms of Riemannian tori are finite we can assume, up to finite index, that $\lambda_{|\Gamma}=1$. Then we can consider the action of $\Gamma\cap\mathcal{C}^2N/(\Gamma\cap\mathcal{C}^3N)$ on the torus $\left(\mathcal{C}^2N/\mathcal{C}^3N\right)/\left(\Gamma\cap\mathcal{C}^2N/(\Gamma\cap\mathcal{C}^3N)\right)$. Notice that $\mathcal{C}^2\mathfrak{n}\subseteq\mathbb{C}\tau\oplus\mathbb{C}\mathrm{Im}\gamma_2\oplus\mathbb{R}\pi_3(w_0)\subseteq\mathbb{C}\tau\oplus W_2$. Hence the action of $\mathfrak{h}$ on $\mathcal{C}^2\mathfrak{n}$, written in the decomposition $\mathbb{C}\tau\oplus W_2$, is given by 
\[v\rightarrow\begin{pmatrix} i\varphi(v) & -\overline{\gamma}^{(2),1}(v)-isb \\ 0 & \rho^{(2),1}(v)\end{pmatrix}\] 
and we see that the eigenvalues of the exponential of this matrix have modulus $1$. We can now consider the action of $\Gamma/(\Gamma\cap \mathcal{C}^2N)$ on $\left(N/\mathcal{C}^2N\right)/\left(\Gamma\cap N/\Gamma\cap \mathcal{C}^2N\right)$. In order to understand this action let us compute the commutator between a generic element $v=(r,u_1,u_2,u_3,s)$ of $\mathfrak{h}$ and one, $v'=(r',u_1',u_2',u_3',s')$, of $\mathfrak{n}$. We have $[v,v']=$
\[\begin{pmatrix} 0 & * & * & * & * & (1) \\ 
0 & 0 & 0 & 0 & * & *\\
0 & 0 & 0 & 0 & * & *\\
0 & 0 & 0 & 0 & \rho^{(2),2}(v)\gamma^{(2),2}(v')-i\varphi(v)\gamma^{(2),2}(v') & \rho^{(2),2}(v)u_3'+s'\gamma^{(2),2}(v)-s\gamma^{(2),2}(v')\\
0 & 0 & 0 & 0 & 0 & i\varphi(v)s'\\
0 & 0 & 0 & 0 & 0 & 0
 \end{pmatrix}\]
 where \[(1)=i\varphi(v)r'-\overline{\gamma}^{(1)}(v)^t u_1'-\overline{\gamma}^{(2),1}(v)^t u_2'-\overline{\gamma}^{(2),2}(v)^t u_3'+is'b(v)+\]\[\overline{\gamma}^{(1)}(v')^t u_1+\overline{\gamma}^{(2),1}(v')^t u_2+\overline{\gamma}^{(2),2}(v')^t u_3-isb(v').\]
Then the action of $\mathfrak{h}$ on $\mathfrak{n}/\mathcal{C}^2\mathfrak{n}$, written following the decomposition $W_3\oplus \mathbb{C}\xi$ and $\mathbb{C}\tau\oplus W_3\oplus \mathbb{C}\xi$ respectively, is given by 
\[v\rightarrow\begin{pmatrix} \rho^{(2),2}(v)-s\gamma_2 & \gamma^{(2),2}(v)-s\gamma_3 \\ 0 & i\varphi(v)\end{pmatrix}\text{ or }\begin{pmatrix} i\varphi(v) & * & * \\0 & \rho^{(2),2}(v)-s\gamma_2 & \gamma^{(2),2}(v)-s\gamma_3 \\ 0 & 0 & i\varphi(v)\end{pmatrix} \]
where the second possibility occurs when $\mathfrak{n}$ is abelian and $\mathbb{C}\tau\cap \mathfrak{n}\ne\{0\}$. We have denoted the restrictions of $\gamma^{(2),2}$ to $W_3$ and $\mathbb{C}\xi$ by $\gamma_2$ and $\gamma_3$ respectively as in Section \ref{UnSimTr}. As said before if we prove that all the eigenvalues of the exponential of these matrices are of modulus $1$ we are done. If $\varphi\ne 0$ then $\dim(\mathfrak{n}/((\mathbb{C}\tau)^\perp\cap \mathfrak{n}))$ is either $0$ or $2$. If it is $0$ then $\mathcal{C}^3\mathfrak{n}=\{0\}$ and $\gamma_2=0$ restricted to $\mathfrak{n}$. If it is $2$ this implies that $\mathbb{C}\mathrm{Im}\gamma_2\subseteq\mathcal{C}^2\mathfrak{n}$. Hence in each case the action on $N/\mathcal{C}^2N$ is upper triangular with eigenvalues of modulus $1$. If $\varphi=0$, since $\mathfrak{h}$ is a Lie algebra, we must have $\rho^{(2),2}(v)\gamma_2(u'_3)=\gamma_2(\rho^{(2),2}(v)u_3'+s'\gamma^{(2),2}(v))$ looking at this when $s'=0$ we see that it implies that $\rho^{(2),2}$ and $\gamma_2$ commutes so that $\rho^{(2),2}+s\gamma_2$ have the same eigenvalues as $\rho^{(2),2}$ and we conclude as before. 
\end{proof}

We are now ready to prove the main result of this section.

\begin{proof}[Proof of Theorem \ref{cryst}] 
Since $\Gamma\le\mathcal{H}(n,1)$ acts properly discontinuously and cocompactly on $\mathfrak{a}(V)$ then $\Gamma$ is virtually polycyclic, see \cite[Theorem 3.1]{Margulis}. Let $H$ be a subgroup of $\text{Aff}(V)$ acting simply transitively on $\mathfrak{a}(V)$ coming from Theorem \ref{FG}. Then we know that $\Gamma\cap H$ has finite index in $\Gamma$ and is a lattice in $H$. Hence, after replacing $\Gamma$ with $\Gamma\cap H$, we can assume that $\Gamma$ is a lattice in $H$. Notice that actually $H$ is the connected component of the identity of a crystallographic hull of $\Gamma$ and after \cite[Theorem 1.4]{FG} the Zariski closure of $\Gamma$ is the same as the one of its crystallographic hull. Hence $H$ is contained in the Zariski closure of $\Gamma$ that lies inside $\mathcal{H}(n,1)$, it follows that $H\le \mathcal{H}(n,1)$. \\
From Proposition \ref{nonunip1} $H$ can be conjugated to a subgroup whose linear part is in $\mathcal{B}$ and after Proposition \ref{nonunip2} either $L(H)\le \hat{\mathcal{B}}$ or $L(H)$ is conjugated to a subgroup of $\mathcal{D}$.  

If $L(H)\le \widehat{\mathcal{B}}$ Proposition \ref{VirtNil} implies that $\Gamma$ is virtually nilpotent. 

We are left to treat the case $L(H)\le \mathcal{D}$. We denote $\widetilde{H}$ the Zariski closure of $H$. Then the unipotent radical of $\widetilde{H}$ is equal to the group $T$ of pure translations of $\mathcal{H}(n,1)$, see the proof of Proposition \ref{nonunip2}. From Lemma \ref{nilpotentRad} the nilradical of $H$ is equal to $T\cap H$, the subgroup of pure translations that lies in $H$. 
From Lemma \ref{Mostow} the group $\Gamma/\Gamma\cap N$ is isomorphic to $\Gamma N/N$, that is discrete in $H/N$. Being $H/N$ isomorphic to $L(H)$, it is contained in $\mathbb{C}^*\times \mathrm{U}(n-1)$. Let us call $\Delta$ the group $\Gamma/\Gamma\cap N$ for convenience. Since $S^1\times \mathrm{U}(n-1)$ is compact, $\Delta/(\Delta\cap (S^1\times \mathrm{U}(n-1)))$ is discrete in $\mathbb{R}^*$ hence cyclic. Furthermore $\Delta\cap (S^1\times \mathrm{U}(n-1))$ is discrete in $S^1\times \mathrm{U}(n-1)$ with the induced topology hence finite. Then, up to take a finite index subgroup of $\Gamma$, we can assume that $\Delta\cap (S^1\times\mathrm{U}(n-1))$ is trivial and so that $\Delta=\Gamma/\Gamma\cap N$ is cyclic. Finally $\Gamma\cap N$ is discrete in $N$ which is some group of translations so $\Gamma\cap N$ is isomorphic to $\mathbb{Z}^m$ for some $m$ and then for dimension reasons we have $m=2n+1$ or $\Gamma$ is abelian. So in this case $\Gamma$ has a finite index subgroup that is abelian by cyclic. 
\end{proof}

Let us now focus on the abelian by cyclic case.

\begin{defi}
Let $n\ge 1$ and $A\in\mathrm{GL}(2n+1,\mathbb{Z})$ diagonalisable with eigenvalues $\lambda$, $\lambda^{-1}, a_i$ with $\overline{\lambda},\lambda^{-1}\in\mathbb{C}^*$ of multiplicity $1$ and $|a_i|=1$. Define  $\widetilde{\Gamma}(2n+2,A)=\mathbb{Z}\ltimes_A\mathbb{Z}^{2n+1}$.
\end{defi}

Then the following propositions provide a classification in the abelian by cyclic case. 

\begin{prop}
Let $n\ge 1$ and $A,A'\in\mathrm{GL}(2n+1,\mathbb{Z})$ then 
\begin{itemize}
\item $\Gamma(2n+2,A)\cong \Gamma(2n+2,A')$ if and only if $A$ is $\mathrm{GL}(2n+1,\mathbb{Z})$-conjugated to either $A'$ or $A'^{-1}$.
\item $\Gamma(2n+2,A)$ is commensurable with $\Gamma(2n+2,A')$ if and only if $A^r $ is $\mathrm{GL}(2n+1,\mathbb{Q})$-conjugated to $A'^s$ for some $r,s\in\mathbb{Z}\setminus\{0\}$.
\end{itemize}
\end{prop}
\begin{proof}
Assume that $A$ does not have finite order, otherwise up to finite index the group $\Gamma(2n+2,A)$ is abelian and the result follows. For the first claim let us notice that $\mathbb{Z}^{2n+1}$ is the unique maximal abelian subgroup of $\Gamma(2n+2,A)$ of rank $2n+1$. Indeed if $(r,u)\in\Gamma(2n+2,A)$, with $r\not=0$, commutes with $(0,u')$ for all $u'\in\mathbb{Z}^{2n+1}$ then $(A^r-\mathrm{Id})u'=0$ but then $A$ would be unipotent and being diagonalisable it is the identity which is a contradiction. Hence indeed $\mathbb{Z}^{2n+1}$ is a maximal abelian subgroup. Let now $\Gamma$ be another one distinct from $\mathbb{Z}^{2n+1}$. Then $\Gamma$ contains an element of the form $(r,u)$ with $r\ne 0$. The elements that commutes with this element are exactly the one corresponding with the eigenspace associated with the eigenvalue $1$ of $A^r$. Then $(r,u)$ commutes with at most $2n-1$ elements. Hence $\Gamma$ can have rank at most $2n$. This implies that $\mathbb{Z}^{2n+1}$ should be preserved and we can write an isomorphism between $\Gamma(2n+2,A)$ and $\Gamma(2n+2,A')$ as $\varphi(r,u)=(\varphi_1(r),\varphi_2^1(r)+\varphi_2^2(u))$. Then in order for $\varphi$ to be a group homomorphism we need \[\varphi_2^1(r+r')=A'^{\varphi_1(r)}\varphi_2^1(r')+\varphi_2^1(r)\text{ and } \varphi_2^2(A^ru'+u)=A'^{\varphi_1(r)}\varphi_2^2(u')+\varphi_2^2(u).\] Taking then $u=0$ we get $\varphi_2^2(A^ru')=A'^{\varphi_1(r)}\varphi_2^2(u')$. Since $\varphi_2^2\in\mathrm{GL}(2n+1,\mathbb{Z})$ and, being $\varphi_1:\mathbb{Z}\rightarrow\mathbb{Z}$ an isomorphism, $\varphi_1$ sends $1$ to $\pm1$, this means that $A$ is $\mathrm{GL}(2n+1,\mathbb{Z})$-conjugated to $A'$ or $A'^{-1}$. For the second claim just notice that a finite index subgroup $\Gamma_0$ of $\Gamma(2n+2,A)$ is of the form $\mathbb{Z}\ltimes_B \mathbb{Z}^{2n+1}$ for some $B\in\mathrm{GL}(2n+1,\mathbb{Z})$. In order to see it as a finite index subgroup of $\Gamma(2n+2,A)$ we have to give an injective morphism $\varphi:\mathbb{Z}\ltimes_B \mathbb{Z}^{2n+1}\rightarrow \Gamma(2n+2,A)$. Using the same notation as before this means that $\varphi_1(r)=mr$ with $m\in\mathbb{Z}\setminus\{0\}$ and if we denote by $\{v_i\}_i$ a basis for  $\mathbb{Z}^{2n+1}$ we have $\varphi_2^2(v_i)=m_iv_i$ i.e. $\varphi_2\in\mathrm{GL}(2n+1,\mathbb{Q})$. As before being $\varphi$ a morphism we have $\varphi_2^2(B^ru')=A^{mr}\varphi_2^2(u')$. That means that $B$ is $\mathrm{GL}(2n+1,\mathbb{Q})$-conjugated to $A^m$ for some $m\in\mathbb{Z}\setminus\{0\}$. Apply then the first part of the lemma.
\end{proof}

\begin{prop} Let $n\ge 1$ then we have the following.
\begin{itemize}
\item If $\Gamma\le \mathcal{H}(n,1)$ is a crystallographic group and $\Gamma$ is not virtually nilpotent then $\Gamma$ contains a subgroup of finite index $\Gamma_0$ that is isomorphic to $\Gamma(2n+2,A)$ for some $A\in\mathrm{GL}(2n+1, \mathbb{Z})$ diagonalisable with eigenvalues $1's$ and $\lambda$, $\lambda^{-1}$ with multiplicity $2$.
\item Every $\Gamma(2n+2,A)$ as before can be realised as a crystallographic group. 
\end{itemize} 
\end{prop}
\begin{proof}
The first claim follows from the proof of Theorem \ref{cryst}. Indeed we are in the case where $\Gamma$ is a lattice in $H$ with $L(H)\le\mathcal{D}$.  For the second claim we can just realise the group $\Gamma(2n+2,A)$ as an affine group as follows
\[\left\{\begin{pmatrix} A^{u_1} & v \\ 0 & 1\end{pmatrix}\ |\ v=(r_1,r_2,u_2,\ldots,u_{2n-2},s_1,s_2)\in\mathbb{Z}^2\times \mathbb{Z}^{2n-3}\times \mathbb{Z}^2, u_1\in\mathbb{Z}\right\}.\]
Then since $A$ is diagonalisable with eigenvalues $1$'s and $\lambda,\lambda^{-1}$ of multiplicity $2$ this means that we can conjugate $A$ to a matrix belonging to $\mathrm{U}(n,1)$ and hence conjugate the whole group to a subgroup of $\mathcal{H}(n,1)$.
\end{proof}

\section{Classification up to isomorphism, dimension $2$, $3$ and $4$ the degenerate cases}\label{IsomDeg}

The classification of the groups $U(\gamma_2,\gamma_3,b_2,b_3) $ up to isomorphism translates to the classification of the Lie algebras $\mathfrak{u}(\gamma_2,\gamma_3,b_2,b_3) $ up to isomorphism. For the Lie algebras that appear in dimension $2$ and $3$ we will use the terminology of \cite{Graaf}. All the other Lie algebras that appear in this section are defined in Appendix \ref{LieBracket} where we also elucidate the correspondence with the terminology of \cite{Graaf} for isomorphism classes of nilpotent Lie algebras up to dimension $6$ and of \cite{Gong} for dimension $7$. The notation used in \cite{Graaf} for the list of Lie algebras up to isomorphism is $L_{i,j}$ where $i$ denotes the dimension of the Lie algebra and $j$ is just an index. While in \cite{Gong} each Lie algebra is labelled by its upper central series dimensions plus an additional letter to distinguish non-isomorphic Lie algebras. Let us start with the classification in small dimensions. 

\begin{exam}
For $n+1=2$ we see that, up to isomorphism, we have just one non abelian Lie algebra. In fact with respect to the basis $\{\tau,i\tau,\xi,i\xi\}$ of $\mathbb{C}\times\mathbb{C}$ the Lie brackets are given by $[\xi,i\xi]=-b_3(\xi)\tau-b_3(i\xi)i\tau$. And we can see that if $b_3\ne 0$ these Lie algebras are isomorphic to a Lie algebra that is the direct sum of the $3$-dimensional Heisenberg Lie algebra, $L_{3,2}$, and a one dimensional abelian ideal.
\end{exam}

\begin{exam}\label{dim6}
For $n+1=3$ let us notice that from Lemma \ref{uppper bound k} we have $\gamma_2=0$. Let us recall, from Section \ref{UnSimTr}, the following notations: $W=S=\mathbb{C}$, $b_2:W\rightarrow\mathbb{R}$, $\gamma_3:\mathbb{C}\rightarrow W$ and $b_3:\mathbb{C}\rightarrow \mathbb{R}$. Let $\{\tau,i\tau,g_1,g_2,\xi,i\xi\}$ be a basis of the Lie algebra. Let us denote by $w_0:=\gamma_3(i\xi)-J\gamma_3(\xi)$.
\begin{enumerate}

\item Assume $w_0=0$. 
\begin{enumerate}
\item Assume $b_2= 0$. Then it is clear that the Lie algebra is either abelian if $b_3= 0$ or isomorphic to $L_{3,2}\oplus\mathbb{R}^3$.

\item Assume $b_2\ne 0$. Then using Proposition \ref{conj} we can conjugate the group in order to replace $b_2$ by $b_2'(u)=|\lambda|^2b_2(\sigma^{-1}u)$ with $\sigma\in \mathrm{U}(1)$ and $\lambda\in\mathbb{C}^*$ so that $b_2(u)=\omega(u,g_1)$ for $u\in W$. Furthermore we can modify $b_3$ by $b'_3(s)=b_3(s)+b_2(-sx)$ for some $x\in \mathbb{C}$, hence we may assume $b_3= 0$. Hence the only non zero brackets are $[g_2,\xi]=i\tau, [g_2,i\xi]=-\tau$. Then in this case the Lie algebra is isomorphic to $L_{5,8}\oplus\mathbb{R}$. 
\end{enumerate}

\item Assume $w_0\not= 0$.

\begin{enumerate}
\item Assume that the real rank of $\gamma_3$ is 1 and $b_2= 0$. Then we might change $\gamma_3 $ by $\gamma'_3(s)=\lambda \sigma\gamma_3(\lambda s)$ with $\lambda\in\mathbb{C}^*$ and $\sigma\in \mathrm{U}(1)$ so that $\gamma_3(1)=\varepsilon g_1$ and $\gamma_3(i)=0$, notice that since $\gamma_3\ne 0$ we have $\varepsilon\ne 0$. Then the non zero brackets are $[g_1,\xi]=\varepsilon \tau, [g_2,\xi]=\varepsilon i\tau, [\xi,i\xi]=-b_3(\xi)\tau-b_3(i\xi) i\tau-\varepsilon g_2$. We can then see that in this case the Lie algebra is isomorphic to $L_{6,25}$.

\item Assume that the real rank of $\gamma_3$ is 1 and $b_2\ne 0$. As before we conjugate the group so that $\gamma_3(1)=\varepsilon g_1$ and $\gamma_3(i)=0$ with $\varepsilon\ne 0$. Then since $b_2(w_0)=2\omega(\gamma_3(i),\gamma_3(1))=0$ we have that $b_2(g_2)=0$ and since $b_2\ne0$ then $b_2(g_1)\ne 0$. Hence the Lie brackets are $[g_1,\xi]=\varepsilon\tau+b_2(g_1)i\tau$, $[g_2,\xi]=\varepsilon i\tau$, $[g_1,i\xi]=-b_2(g_1)\tau$, $[\xi,i\xi]=-b_3(\xi)\tau-b_3(i\xi)i\tau-\varepsilon g_2$. We can see they are all isomorphic to $L_{6,19}(0)$.

\item Assume that the real rank of $\gamma_3$ is $2$. Notice that nevertheless $\gamma_3(\xi)$ and $\gamma_3(i\xi)$ are linearly dependent over $\mathbb{C}$ hence $\gamma_3(i\xi)=\overline{\lambda}\gamma_3(\xi)$ with $\lambda=\lambda_1+i\lambda_2\in\mathbb{C}$. Let $v_1=\frac{\gamma_3(\xi)}{\|\gamma_3(\xi)\|^2}\in W$  so that $h(v_1,\gamma_3(\xi))=1$ and $h(v_1,\gamma_3(i\xi))=\lambda$. Let $\alpha=\|\lambda_3(\xi)\|^2$ and $\{x_1,y_1\}$ be a basis for $W$ as a real vector space, we have $w_0=\alpha\lambda_1x_1-\alpha(\lambda_2+1)y_1$, then the Lie brackets are 
\[[x_1,\xi_1]=\tau_1+b_2(x_1)\tau_2, \qquad [y_1,\xi_1]=(1+b_2(y_1))\tau_2,\]
\[[x_1,\xi_2]=(\lambda_1-b_2(x_1))\tau_1+\lambda_2\tau_2, \qquad [y_1,\xi_2]=(-\lambda_2-b_2(y_1))\tau_1+\lambda_1\tau_2,\]
\[[\xi_1,\xi_2]=-b_3(1)\tau_1-b_3(i)\tau_2-w_0.\]
Notice that we have  
\[[w_0,\xi_1]=\alpha\lambda_1\tau_1-\alpha(1+3\lambda_2)\tau_2,\]
\[[w_0,\xi_2]=\alpha(\lambda_1^2+\lambda_2^2+3\lambda_2)\tau_1-\alpha\lambda_1\tau_2.\]
Now if the transformation $\tau_1'=\alpha\lambda_1\tau_1+(2\alpha\lambda_2-\alpha(\lambda_2+1))\tau_2, \tau_2'=(\alpha(\lambda_1^2+\lambda_2^2+\lambda_2)-2\alpha\lambda_2)\tau_1-\alpha\lambda_1\tau_2$ has non zero determinant we can bring the Lie algebra to the following form
\[[x_1,x_2]=x_3,\qquad [x_1,x_3]=x_5,\qquad [x_2,x_3]=x_6,\]
\[[x_1,x_4]=ax_5+bx_6,\qquad [x_2,x_4]=cx_5+dx_6\] that furthermore can be brought to $L_{6,24}(\varepsilon)$. If instead the determinant is $0$ we can bring it to the standard form $L_{6,23}$ or one of $L_{6,25}$ or $L_{6,19}(0)$ that we already encountered.
\end{enumerate}
\end{enumerate}
\end{exam}

Now we want to give representatives of the isomorphism classes of the groups $U(\gamma_2,\gamma_3,b_2,b_3)$ when either $\pi_3(\gamma_3(i\xi)-J\gamma_3(\xi))=0$ or $\mathrm{rank}(\gamma_2)=0$ in dimension $4$. The dimension $4$ is the smallest dimension where $\gamma_2$ can be non trivial. 

\begin{prop}\label{w_0=0}
The list of Lie algebras $L_j$ for $j=1,\ldots,6$ given in Appendix \ref{LieBracket} with $\mathbb{K}=\mathbb{R}$ is a complete non redundant list of isomorphism classes of the Lie algebras $\mathfrak{u}(\gamma_2,\gamma_3,b_2,b_3)$ when $\pi_3(\gamma_3(i\xi)-J\gamma_3(\xi))=0$ and $n+1=4$.
\end{prop}
\begin{proof}
Let us call $w_0:=\gamma_3(i\xi)-J\gamma_3(\xi)$. If $\pi_3(w_0)=0$ we saw in Proposition \ref{b2} that $\Ima\gamma_3\subseteq \Ima\gamma_2$ and $b_2= 0$ on $\Ima\gamma_2\oplus J\Ima\gamma_2$. If $\gamma_2= 0$ the result follows easily and we find the Lie algebras $L_i$ with $i=1,2,3$ that are $2$-step nilpotent. Hence let us suppose $\mathrm{rank}(\gamma_2)=1$ and let $\{\tau,i\tau,e,Je,g,f,\xi,i\xi\}$ be a basis for $\mathbb{C}\times W\times \mathbb{C}$ that respects the decomposition (\ref{decomposition}). Then $\gamma_2(f)=\delta e$, $\gamma_3(\xi)=\nu e$ and $\gamma_3(i\xi)=\mu e$ with $\delta,\nu,\mu\in\mathbb{R}$. If $f'=\delta^{-1}f$ the Lie brackets become 
\[[e,f']=\tau,\qquad [e,\xi]=\nu\tau,\qquad[e,i\xi]=\mu\tau,\]
\[[Je,f']=i\tau,\qquad [Je,\xi]=\nu i\tau,\quad [Je,i\xi]=\mu i\tau,\]
\[[g,\xi]=b_2(g)i\tau, \qquad[g,i\xi]=-b_2(g)\tau,\]
\[[f',\xi]=b_2(f')i\tau+e,\qquad [f',i\xi]=-b_2(f')\tau+Je,\]
\[[\xi,i\xi]=-b_3(\xi)\tau-b_3(i\xi)i\tau-w_0.\]
After replacing $e$ by $b_2(f')\tau_2+e$ and $Je$ by $-b_2(f')\tau_1+Je$ we can assume that $[f',\xi]=e$ and $[f',i\xi]=Je$.
Furthermore, letting $\xi_1= \xi-\nu f'$ and $\xi_2=i\xi -\mu f'$, we have $[e,\xi_1]=[e,\xi_2]=[Je,\xi_1]=[Je,\xi_2]=0$  and $[\xi_1,\xi_2]=b\tau+ci\tau$ for some $b,c\in\mathbb{R}$. Now, if $b_2(g)=0$ and $b_3= 0$, then we can find an isomorphism between the Lie algebra $\mathfrak{u}(\gamma_2,\gamma_3,b_2,b_3)$ and $L_4$. If instead $b_2(g)=0$ but $b\ne 0$ then defining $x_1=f',x_2=-(b\xi_1+c\xi_2),x_3= -b^{-1}\xi_2,x_4=-( be+cJe),x_5= -b^{-1}Je, x_6=b\tau+ci\tau,x_7= b^{-1}i\tau$ we see that the Lie algebra $\mathfrak{u}(\gamma_2,\gamma_3,b_2,b_3)$ is isomorphic to $L_5$.
Finally, assuming $b_2(g)\ne 0$, we can define $x_1=f',x_2=\xi+\frac{b}{b_2(g)}g, x_3=i\xi+\frac{c}{b_2(g)}g,x_4=e,x_5=Je,x_6=\frac{1}{b_2(g)}g,x_7=-\tau$ and $x_8=-i\tau$ in order to see that $\mathfrak{u}(\gamma_2,\gamma_3,b_2,b_3)$ is isomorphic to $L_6(1)$. 
\end{proof}

Next we give the classification in dimension $4$ for the case $\pi_3(\gamma_3(i\xi)-J\gamma_3(\xi))\not =0$ and $\gamma_2=0$.

\begin{prop}\label{gamma_2=0}
The list of Lie algebras $N_j$ for $j=1,\ldots,13$ given in Appendix \ref{LieBracket} with $\mathbb{K}=\mathbb{R}$ is a complete non redundant list of isomorphism classes of the Lie algebras $\mathfrak{u}(\gamma_2,\gamma_3,b_2,b_3)$ when $\pi(\gamma_3(i\xi)-J\gamma_3(\xi))\ne 0$, $\gamma_2= 0$ and $n+1=4$.
\end{prop}
\begin{proof}
Let $\{\tau,i\tau,g_1,g_2,g_3,g_4,\xi,i\xi\}$ be a basis for the real vector space $\mathbb{C}\times W\times \mathbb{C}$. Let us call $w_0:=\gamma_3(i\xi)-J\gamma_3(\xi)$ then the Lie brackets of $\mathfrak{u}(0,\gamma_3,b_2,b_3)$ read as \[[g_j,\xi]=\langle g_j,\gamma_3(\xi)\rangle\tau+\left(\omega(g_j,\gamma_3(\xi))+b_2(g_j)\right)i\tau,\]\[[g_j,i\xi]=(\langle g_j,\gamma_3(i\xi)\rangle-b_2(g_j))\tau+\omega(g_j,\gamma_3(i\xi))i\tau,\]\[ [\xi,i\xi]=-b_3(\xi)\tau-b_3(i\xi)i\tau-w_0.\]
With 
\[[w_0,\xi]=\langle\gamma_3(i\xi),\gamma_3(\xi)\rangle \tau+\left(3\omega(\gamma_3(i\xi),\gamma_3(\xi))-\|\gamma_3(\xi)\|^2\right)i\tau\]
\[[w_0,i\xi]=\left(\|\gamma_3(i\xi)\|^2-3\omega(\gamma_3(i\xi),\gamma_3(\xi)\right)\tau-\langle\gamma_3(i\xi),\gamma_3(\xi)\rangle i\tau.\]
First of all by redefining $w_0$ we might assume that $[\xi,i\xi]=w_0$.
Let $b_2(u)=\langle u,v_0\rangle $, for some vector $v_0\in W$. Then we can notice that the restrictions of $\mathrm{ad}(\xi)$ and $\mathrm{ad}(i\xi)$ to $W$ define two linear maps $W\rightarrow \mathbb{C}$ and that taking their real and imaginary part we get four linear forms over the reals. Call these linear forms $\alpha_i$, explicitly $\alpha_1=\langle\cdot, \gamma_3(\xi)\rangle,\alpha_2=\langle\cdot, J\gamma_3(\xi)+v_0\rangle,\alpha_3=\langle\cdot, \gamma_3(i\xi)-v_0\rangle,\alpha_4=\langle\cdot, J\gamma_3(i\xi)\rangle$. For $i=1,\ldots,4$ let us denote by $v_i$ the vector associated with the linear form $\alpha_i$. Notice that, if $\mathcal{Z}$ is the center of $\mathfrak{u}(0,\gamma_3,b_2,b_3)$, the center of $\mathfrak{u}(0,\gamma_3,b_2,b_3)/\mathbb{C}\tau$ is $W$, hence $\dim(W\cap \mathcal{Z})$ is an invariant and we have 
\[\dim(\mathcal{Z}\cap W)=\dim\left(\bigcap_{i=1}^4 \ker\alpha_i\right)=\dim\left(\spn\{v_1,\ldots,v_4\}^\perp\right).\]

\begin{claim}\label{claim2.3.5}
If $w_0\in\mathcal{Z}\cap W$ then $\dim\left(\spn\{v_1,\ldots,v_4\}\right)\ge 3$. If instead  $w_0\notin\mathcal{Z}\cap W$ then we have $\dim\left(\spn\{v_1,\ldots,v_4\}\right)\ge 2$. Hence in any case $0\le \dim(\mathcal{Z}\cap W) \le 2$.
\end{claim}
\begin{proof}
Notice first that since $w_0\ne 0$, $\gamma_3(\xi)$ and $\gamma_3(i\xi)$ cannot be both zero. Furthermore if one of them is zero, say $\gamma_3(\xi)$, then $\spn\{v_1,\ldots,v_4\}=\spn\{v_0,\gamma_3(i\xi),J\gamma_3(i\xi)\}$ has dimension at least two hence the claim follows. Assume now that both $\gamma_3(\xi)$ and $\gamma_3(i\xi)$ are non zero. 
Assume by contradiction that $w_0\notin\mathcal{Z}\cap W$ and $\dim\left(\spn\{v_1,\ldots,v_4\}\right)=1$. Since $\gamma_3(\xi)\ne0$ we can take it as generator of $\spn\{v_1,\ldots,v_4\}$. But $w_0=\gamma_3(i\xi)-J\gamma_3(\xi)=-J(v_4+v_1)=-\lambda J\gamma_3(\xi)$ with $\lambda\in\mathbb{R}$. Hence $w_0$ is orthogonal to all of $\spn\{v_1,\ldots,v_4\}$ and hence $w_0\in\mathcal{Z}\cap W$ and we get a contradiction. Assume now that $w_0\in\mathcal{Z}\cap W$ then we claim that $\gamma_3(\xi)$ and $\gamma_3(i\xi)$ are linearly independent over $\mathbb{C}$. Indeed suppose that $\gamma_3(i\xi)=\lambda\gamma_3(\xi)$ with $\lambda=\lambda_1+i\lambda_2\in\mathbb{C}$. Since $w_0$ is in the center, we must have $\langle \gamma_3(i\xi),\gamma_3(\xi)\rangle =0$ hence $\lambda_1=0$, $3\omega(\gamma_3(i\xi),\gamma_3(\xi))-\|\gamma_3(\xi)\|^2=0$ hence $3\lambda_2-1=0$, finally  $\|\gamma_3(i\xi)\|^2-3\omega(\gamma_3(i\xi),\gamma_3(\xi))=0$ hence $\lambda_2^2-3\lambda_2=0$ and we get a contradiction. Then $v_1=\gamma_3(\xi)$ and $v_4=J\gamma_3(i\xi)$ are linearly independent over $\mathbb{R}$. Furthermore $v_2+v_3=J\gamma_3(\xi)+\gamma_3(i\xi)$ is also independent form $\{v_1,v_4\}$. Indeed if $v_2+v_3=J\gamma_(\xi)+\gamma_3(i\xi)=\lambda\gamma_3(\xi)+\mu\gamma_3(i\xi)$ then we would have $\gamma_3(\xi)$ and $\gamma_3(i\xi)$ linearly depend over $\mathbb{C}$. Hence $\dim\left(\spn\{v_1,\ldots,v_4\}\right)\ge 3$.
\end{proof}

Let us first assume $\dim\left(\spn\{v_1,\ldots,v_4\}\right)=4$, let $(z_1,\ldots,z_4)$ be the dual basis of $(\alpha_1,\ldots,\alpha_4)$ then
\[[z_1,\xi]=\tau,\qquad[z_2,\xi]=i\tau,\qquad[z_3,i\xi]=\tau,\qquad[z_4,i\xi]=i\tau,\qquad[\xi,i\xi]=w_0.\]
If $w_0=\sum a_iz_i$ then $[[\xi,i\xi],\xi]=a_1\tau+a_2i\tau, [[\xi,i\xi],i\xi]=a_3\tau+a_4i\tau$. Let us write the structure constants in a matrix $A=\begin{pmatrix} a_1 & a_2 \\ a_3 & a_4\end{pmatrix}$ and let us represent a change of basis in $\spn\{\xi,i\xi\}$ with $P\in\mathrm{GL}(2,\mathbb{R})$. Then the matrix that represents the structure constants in the new basis is just $\det(P)PA$. Hence, depending on the rank of the matrix $A$, using the left action of $\mathrm{GL}(2,\mathbb{R})$ just defined, we can bring $A$ to one of the three normal forms $\begin{pmatrix} 1 & a \\ 0 & 0 \end{pmatrix}$, $\begin{pmatrix} 0 & 1 \\ 0 & 0 \end{pmatrix}$ and $\begin{pmatrix} 1 & 0 \\0 & 1 \end{pmatrix}$. The first matrix, if we call $z_1'=z_1+az_2$, leads to the following Lie brackets
\[[\xi,i\xi]=z_1',\qquad [z_1',\xi]=\tau+ai\tau,\]
\[[z_2,\xi]=i\tau,\qquad [z_3,i\xi]=\tau, \qquad [z_4,i\xi]=i\tau\] and one can easily see that this Lie algebra is isomorphic to $N_{10}$. The second normal form brings as well to the class $N_{10}$. For the third, after calling $z_1'=z_1+z_4$, we have the following Lie brackets 
\[[\xi,i\xi]=z_1',\qquad [z_1',\xi]=\tau,\qquad[z_1',i\xi]=i\tau,\]
\[[z_2,\xi]=i\tau,\qquad[z_3,i\xi]=\tau,\qquad [z_4,i\xi]=i\tau\]
and this Lie algebra is isomorphic to the Lie algebra $N_{11}$.

Let us now assume that  $\dim\left(\spn\{v_1,\ldots,v_4\}\right)=3$. If $w_0\in \mathcal{Z}\cap W$ suppose then that $\{v_1,v_2,v_3\}$ are linearly independent, indeed the other combinations can be reduced to this case by a change of variables. Complete $\{\alpha_1,\alpha_2,\alpha_3\}$ to $\{\alpha_1,\alpha_2,\alpha_3,\beta\}$ in order to have a basis for $W^*$. Let us take its dual basis, call it $(z_1,\ldots,z_4)$, as basis for $W$. The Lie brackets expressed in this basis become
\[[z_1,\xi]=\tau,\qquad[z_1,i\xi]=\alpha_4(z_1)i\tau,\qquad[z_2,\xi]=i\tau,\qquad[z_2,i\xi]=\alpha_4(z_2)i\tau,\]
\[[z_3,i\xi]=\tau +\alpha_4(z_3)i\tau, \qquad [\xi,i\xi]=w_0=z_4.\]
After a change of basis we can arrive to the following form 
\[[z_1,\xi_1]=\tau_1,\qquad[z_2,\xi_1]=\tau_2,\qquad[z_1,\xi_2]=\mu\tau_2,\qquad[z_3,\xi_2]=\tau_1,\qquad[\xi_1,\xi_2]=z_4\] with $\mu=\alpha_4(z_1)+\alpha_4(z_2)\alpha_4(z_3)$ and, depending on whether $\mu$ is $0$ or not, we find the normal forms $N_{12}$ or $N_{13}$. If instead $w_0\notin \mathcal{Z}\cap W$ there exists $x\in\mathcal{Z}\cap W$ such that $\mathfrak{u}(0,\gamma_3,b_2,b_3)\cong \mathbb{R}x\oplus \mathfrak{g}'$ where $\mathfrak{g}'$ is a $7$-dimensional $3$-nilpotent Lie algebra with a $2$-dimensional center. Notice that $\mathfrak{g}'$ is not decomposable otherwise $\dim(\mathcal{Z}\cap W)>1$. Hence $\mathfrak{g}'$ can be one of $(257\delta)$ with $\delta\in\{A,\ldots,L\}$, or some real form of them, in Gong's list. Some of them can be excluded since we know that $W\oplus\mathbb{C}\tau$ is a $6$-dimensional abelian subalgebra, so we get that the only possibilities are $(257\delta)$ with $\delta\in\{A,B,C,D,I,J,J_1\}$. Hence in this case $\mathfrak{u}(0,\gamma_3,b_2,b_3)$ is isomorphic to one of the $N_j$ with $j=5,\ldots,8$ or $N_{9}(\varepsilon)$ with $\varepsilon\in\{0,1,-1\}$. One can notice that the Lie algebras $N_5,\ldots,N_{13}$ can all be realised.

Let us now assume that  $\dim\left(\spn\{v_1,\ldots,v_4\}\right)=2.$ It follows from Claim \ref{claim2.3.5} below that then  $w_0\notin\mathcal{Z}\cap W$. Then there exist $x,y\in\mathcal{Z}\cap W$ such that $\mathfrak{u}(0,\gamma_3,b_2,b_3)\cong \spn\{x,y\}\oplus\mathfrak{g}'$ where $\mathfrak{g}'$ is a $6$-dimensional $3$-nilpotent Lie algebra with a center of dimension $2$, hence one of $L_{6,23},L_{6,24}(\varepsilon),L_{6,25},L_{6,27}$ or $L_{5,5}\oplus\mathbb{R}$ in de Graaf's list. One can see that all of these Lie algebras, except $L_{5,5}\oplus\mathbb{R}$, appeared already in dimension $6$ hence they can be realised. Instead let us see that $L_{5,5}\oplus\mathbb{R}^3$ does not appear as Hermite-Lorentz Lie algebra. Indeed this Lie algebra is characterised by the fact that $\dim\left(\spn\{v_1,\ldots,v_4\}\right)=2$ and that the dimension of $\mathrm{ad}(w_0)$ is $1$. Notice that if $v_1-v_4=0$ then $\gamma_3(\xi)=J\gamma_3(i\xi)$ but then the dimension of $\mathrm{ad}(w_0)$ being $1$ would imply $\gamma_3(i\xi)=0$ and hence $w_0=0$. Hence $v_1-v_4\ne 0$ and since $v_2+v_3=J(v_1-v_4)$ this implies that $\dim\spn\{v_1-v_4,v_2+v_3\}=2$ and we can take $\{v_1-v_4,v_2+v_3\}$ as generators of $\{v_1,\ldots,v_4\}$. Then $v_1$ depends linearly on them, i.e.~$\gamma_3(\xi)=\lambda(\gamma_3(\xi)-J\gamma_3(i\xi))+\mu(J\gamma_3(\xi)+\gamma_3(i\xi))$ for some $\lambda,\mu\in\mathbb{R}$. Hence $\gamma_3(\xi)$ and $\gamma_3(i\xi)$ are linearly dependent over $\mathbb{C}$. Hence $\spn\{v_2+v_3,v_1-v_4\}=\spn\{\gamma_3(\xi),J\gamma_3(\xi)\}$. Since $v_2=J\gamma_3(\xi)+v_0\in\spn\{v_1-v_4,v_2+v_3\}=\spn\{\gamma_3(\xi),J\gamma_3(\xi)\}$ this implied that also $v_0\in\spn\{\gamma_3(\xi),J\gamma_3(\xi)\}$. Then everything is contained in the plane $\spn\{\gamma_3(\xi),J\gamma_3(\xi)\}$ and hence the Lie algebra $L_{5,5}\oplus\mathbb{R}$ should have appeared in dimension $6$. 
Hence $\mathfrak{u}(0,\gamma_3,b_2,b_3)$ is isomorphic to one of $N_j$ with $j=1,\ldots,4$. 
\end{proof}

\section{Classification up to isomorphism, dimension $4$ the general case}\label{IsomGen}

In this section we will treat the most general case $\pi_3(\gamma_3(i\xi)-J\gamma_3(\xi))\not =0$ and $\gamma_2\ne 0$ for dimension $4$. We will first see how we can bring these Lie algebras to a simpler form. Then we will introduce the concept of Carnot Lie algebras and see how the classification up to isomorphism is translated to the classification of the orbits of the $\G$-adjoint action on $\Gr$. Finally we will study separately a particular case, namely the case $\alpha=0$ before studying the most general case.

\subsection{Reduction}

Let us introduce the following family of Lie algebras to which the Lie algebras $\mathfrak{u}(\gamma_2,\gamma_3,b_2,b_3)$ are isomorphic in the case $\pi_3(\gamma_3(i\xi)-J\gamma_3(\xi))\not=0$ and $\mathrm{rank}(\gamma_2)=1$ in dimension $4$.

\begin{defi}
Let $\alpha,a,b,c\in\mathbb{R}$ and define the family of Lie algebras $\{\mathfrak{g}_\mathbb{R}(\alpha,a,b,c)\}_{\alpha,a,b,c}$ by the following non zero Lie brackets, expressed in the basis $\{x_1,\ldots,x_8\}$ of $\mathbb{R}^8$: 
\[[x_1,x_2]=x_4,\qquad [x_1,x_3]=x_5,\qquad [x_2,x_3]=x_6,\]
\[[x_1,x_4]=x_7,\qquad [x_1,x_5]=x_8,\]
\[[x_2,x_4]=\alpha x_7,\qquad [x_2,x_5]=-\alpha x_8,\]
\[ [x_3,x_4]=-\alpha x_8,\qquad [x_3,x_5]=3\alpha x_7,\]
\[[x_2,x_6]=ax_7+bx_8,\qquad [x_3,x_6]=cx_7-ax_8.\]
\end{defi}

\begin{prop}
If $\pi_3(\gamma_3(i\xi)-J\gamma_3(\xi))\not =0 $, $n+1=4$ and  $\mathrm{rank}(\gamma_2)=1$ then the Lie algebras $\mathfrak{u}(\gamma_2,\gamma_3,b_2,b_3)$ are isomorphic to $\mathfrak{g}_\mathbb{R}(\alpha,a,b,c)$ for some $\alpha,a,b,c$.
\end{prop}
\begin{proof}
Let $\{\tau,i\tau,e,Je,g,f,\xi,i\xi\}$ be a basis for $\mathbb{C}\times W\times \mathbb{C}$ adapted to the decomposition (\ref{decomposition}) and let $w_0:=\gamma_3(i\xi)-J\gamma_3(\xi)$. The Lie brackets of $\mathfrak{u}(\gamma_2,\gamma_3,b_2,b_3)$ expressed in this basis read as
	\[[e,f]=h(e,\gamma_2(f))\tau,\qquad[Je,f]=h(Je,\gamma_2(f))\tau,\]
	\[[e,\xi]=(h(e,\gamma_3(\xi))+ib_2(e))\tau,\qquad[e,i\xi]=(h(e,\gamma_3(i\xi))-b_2(e))\tau,\]
	\[[Je,\xi]=(h(Je,\gamma_3(\xi))+ib_2(Je))\tau,\qquad[Je,i\xi]=(h(Je,\gamma_3(i\xi))-b_2(Je))\tau,\]  
	\[[g,\xi]=(h(g,\gamma_3(\xi))+ib_2(g))\tau,\qquad[g,i\xi]=(h(g,\gamma_3(i\xi))-b_2(g))\tau,\]
	\[[f,\xi]=(h(f,\gamma_3(\xi))+ib_2(f))\tau+\gamma_2(f),\qquad[f,i\xi]=(h(f,\gamma_3(i\xi))-b_2(f))\tau+J\gamma_2(f),\]
	\[[\xi,i\xi]=(-b_3(\xi)-ib_3(i\xi))\tau-w_0.\]
First of all, up to conjugating the group, we can assume that $b_2(e)=0$ i.e.~$\langle e,J\gamma_3(\xi)\rangle=0$ and $b_2(f)=0$. If $\gamma_2(f)=\delta e$, with $\delta\in\mathbb{R}^*$, making the following change of variables 
\[\begin{cases}
\xi_1=\xi-\frac{1}{\delta}\left(\langle e,\gamma_3(\xi)\rangle-\langle Je,\gamma_3(i\xi)\rangle\right)f\\
\xi_2=i\xi-\frac{1}{\delta}\langle e,\gamma_3(i\xi)\rangle f
\end{cases}\]
we have that $\gamma_3(\xi_1)=\gamma_3(\xi)-(\langle e,\gamma_3(\xi)\rangle -\langle Je,\gamma_3(i\xi)\rangle )e$ and $ \gamma_3(\xi_2)=\gamma_3(i\xi)-\langle e,\gamma_3(i\xi)\rangle e$, in other words $\langle \gamma_3(\xi_2),e\rangle =0$ and $\langle \gamma_3(\xi_1),e\rangle=\langle \gamma_3(\xi_2),Je\rangle$. Then $[\xi_1,\xi_2]=-b_3(\xi_1)\tau-b_3(\xi_2)i\tau-w_0'$ with $w_0'=\gamma_3(\xi_2)-J\gamma_3(\xi_1)$ and  $\langle w_0',e\rangle =\langle w_0',Je\rangle =0$, hence $ w_0'=\beta g$ for some $\beta\in\mathbb{R}^*$. Defining now $x_1=\delta^{-1}f=f',x_2=\xi_1$ and $x_3=\xi_2$ we have 
\[[x_1,x_2]=h(f',\gamma_3(\xi_1))\tau+e,\]
\[[x_1,x_3]=h(f',\gamma_3(\xi_2))\tau+Je,\]
\[[x_2,x_3]=-b_3(\xi_1)\tau-b_3(\xi_2)i\tau-w_0'.\]
Hence let us define $x_4=h(f',\gamma_3(\xi_1))\tau+e,x_5=h(f',\gamma_3(\xi_2))\tau+Je,x_6=-b_3(\xi_1)\tau-b_3(\xi_2)i\tau-w_0', x_7=-\tau$ and $x_8=-i\tau$. Then, remembering that $b_2(Je)=2\langle e,J\gamma_3(\xi_2)\rangle =-2\langle \gamma_3(\xi_2),Je \rangle$, we have 
\[[x_1,x_4]=x_7,\qquad[x_2,x_4]=\langle \gamma_3(\xi_2),Je\rangle x_7,\qquad[x_3,x_4]=-\langle \gamma_3(\xi_2),Je\rangle x_8,\]
\[[x_1,x_5]=x_8,\qquad[x_2,x_5]=-\langle \gamma_3(\xi_2),Je\rangle x_8,\qquad[x_3,x_5]=3\langle \gamma_3(\xi_2),Je\rangle x_7,\]
\[[x_2,x_6]=-(\langle \gamma_3(\xi_2),\gamma_3(\xi_1)\rangle x_7+(3\langle \gamma_3(\xi_2),J\gamma_3(\xi_1)\rangle-\|\gamma_3(\xi_1)\|^2)x_8),\]
\[[x_3,x_6]=-((\|\gamma_3(\xi_2)\|^2-3\langle \gamma_3(\xi_2),J\gamma_3(\xi_1)\rangle)x_7-\langle \gamma_3(\xi_2),\gamma_3(\xi_1)\rangle x_8).\]
Calling $\alpha=\langle \gamma_3(\xi_2),Je\rangle,$ \[a=-\langle \gamma_3(\xi_2),\gamma_3(\xi_1)\rangle,\]\[b=-(3\langle \gamma_3(\xi_2),J\gamma_3(\xi_1)\rangle-\|\gamma_3(\xi_1)\|^2),\]\[c=-(\|\gamma_3(\xi_2)\|^2-3\langle \gamma_3(\xi_2),J\gamma_3(\xi_1)\rangle)\]
we can see the isomorphism with $\mathfrak{g}_\mathbb{R}(\alpha,a,b,c)$.
\end{proof}
\begin{rema}
If $a=b=c=0$ the center of the corresponding Lie algebras has dimension $3$ and is generated by  $\{x_6,x_7,x_8\}$, otherwise it is of dimension $2$ and is generated by $\{x_7,x_8\}$.
\end{rema}

\subsection{Carnot Lie algebras}\label{Carnot}

\begin{defi}
A \emph{Carnot grading} on a Lie algebra $\mathfrak{g}$ is an algebra grading of $\mathfrak{g}$, $\mathfrak{g}=\oplus_{i\ge 1}\mathfrak{g}_i$ where the $\mathfrak{g}_i$ are Lie subalgebras of $\mathfrak{g}$ and $[\mathfrak{g}_i,\mathfrak{g}_j]\subseteq\mathfrak{g}_{i+j}$, such that $\mathfrak{g}$ is generated by $\mathfrak{g}_1$. A Lie algebra is \emph{Carnot graded} if it is endowed with a Carnot grading and \emph{Carnot} if it admits a Carnot grading. 
\end{defi}

\begin{defi}
Given a Lie algebra $\mathfrak{g}$ the direct sum $\mathrm{Car}(\mathfrak{g})=\bigoplus_{i\ge 1}\mathfrak{v}_i$, where $\mathfrak{v}_i=\mathfrak{g}^i/\mathfrak{g}^{i+1}$ and $\mathfrak{g}^i$ is the sequence of subalgebras defined by $\mathfrak{g}^{i+1}=[\mathfrak{g},\mathfrak{g}^{i}]$ and $\mathfrak{g}^1=\mathfrak{g}$, endowed with the Lie brackets induced on each quotient by the ones of the Lie algebra $\mathfrak{g}$ is called the \emph{associated Carnot-graded} Lie algebra to the Lie algebra $\mathfrak{g}$.
\end{defi}

\begin{prop}[{\cite[Proposition 3.5]{Cornulier}}]
A Lie algebra is Carnot if and only if it is isomorphic, as a Lie algebra, to its associated Carnot-graded Lie algebra. Furthermore if these conditions hold, then: 
\begin{itemize}
\item for any Carnot grading on $\mathfrak{g}$, the graded Lie algebras $\mathfrak{g}$ and $\mathrm{Car}(\mathfrak{g})$ are isomorphic,
\item for any two Carnot gradings on $\mathfrak{g}$, there is a unique automorphism mapping the first to the second and inducing the identity modulo $[\mathfrak{g},\mathfrak{g}]$.
\end{itemize}
\end{prop}

\begin{coro}[{\cite[Corollary 3.6]{Cornulier}}]
Let $\mathfrak{g}$ be a Carnot graded Lie algebra. Denote by $Aut(\mathfrak{g})$ its automorphism group as a Lie algebra and $Aut(\mathfrak{g})_0$ its automorphism group as a graded Lie algebra. Let $Aut(\mathfrak{g})_{\ge 1}$ be the group of automorphism of the Lie algebra $\mathfrak{g}$ inducing the identity on $\mathfrak{g}/[\mathfrak{g},\mathfrak{g}]$. Then $Aut(\mathfrak{g})_{\ge 1}$ is a normal subgroup and 
\[Aut(\mathfrak{g})=Aut(\mathfrak{g})_0\ltimes Aut(\mathfrak{g})_{\ge 1}\]
\end{coro}

\begin{exam}
Here is an example of a nilpotent Lie algebra that is not Carnot. It is given by the following non zero Lie brackets on the basis $\{x_1,\ldots,x_5\}$ 
\[[x_1,x_3]=x_4\]
\[[x_1,x_4]=[x_2,x_3]=x_5.\]
\end{exam}

\begin{rema}
Each Lie algebra $\mathfrak{g}_\mathbb{R}(\alpha,a,b,c)$ is Carnot with grading $\mathfrak{v}_1=\spn_\mathbb{R}\{x_1,x_2,x_3\}, \mathfrak{v}_2=\spn_\mathbb{R}\{x_4,x_5,x_6\}$ and $ \mathfrak{v}_3=\spn_\mathbb{R}\{x_7,x_8\}$.
\end{rema}

\begin{exam}
Every free $k$-step nilpotent Lie algebra of rank $n$ is Carnot.
\end{exam}

\begin{prop}\label{quotientOfFree}
Every quotient of the free $k$-step nilpotent Lie algebra of rank $n$ by a graded ideal is a Carnot nilpotent Lie algebra.
\end{prop}
\begin{proof}
Denote by $F_{k,n}$ the free $k$-step nilpotent Lie algebra of rank $n$ and
let $\mathfrak{h}$ be an homogeneous ideal of $F_{k,n}$, if $F_{k,n}=\bigoplus_{i=1}^k F_i$ we have $\mathfrak{h}=\bigoplus_i (\mathfrak{h}\cap F_i)=\bigoplus_i \mathfrak{h}_i$. Then the quotient $\mathfrak{g}=F_{k,n}/\mathfrak{h}$ inherits the grading and it will be generated by $F_1/\mathfrak{h}_1$ since $\mathfrak{h}$ is an ideal. Indeed if an element $x\in\mathfrak{g}$ is obtained as the brackets of elements in $F_1$ with some of them in $\mathfrak{h}_1$ then $x$ is itself in $\mathfrak{h}_1$. Finally the quotient will be a nilpotent Lie algebra.
\end{proof}

Let us introduce a general setting in which we will see our family of Lie algebras $\mathfrak{g}_\mathbb{R}(\alpha,a,b,c)$.

\begin{defi}
Let $F$ be the free $3$-step nilpotent Lie algebra on $3$ generators. Then with respect to the basis $\{y_1,\ldots y_{14}\}$ of $\mathbb{R}^{14}$ the Lie brackets are as follows
\[[y_1,y_2]=y_4,\qquad[y_1,y_3]=y_5,\qquad[y_2,y_3]=y_6,\]
\[[y_1,y_4]=y_7,\qquad[y_1,y_5]=y_8,\qquad[y_1,y_6]=y_9,\]
\[[y_2,y_4]=y_{10},\qquad[y_2,y_5]=y_{11},\qquad[y_2,y_6]=y_{12},\]
\[[y_3,y_4]=y_{11}-y_9,\qquad[y_3,y_5]=y_{13},\qquad[y_3,y_6]=y_{14}.\]
We also have a grading of $F$ as $F\cong\oplus_{i=1}^3 F_i$ where $F_i=F^i/F^{i+1}$. Notice that $F_1=\spn_\mathbb{R}\{y_1,y_2,y_3\}, F_2=\spn_\mathbb{R}\{y_4,y_5,y_6\}$ and $F_3=\spn_\mathbb{R}\{y_7,\ldots,y_{14}\}$, where we are using an abuse of notation, thinking the elements of the basis of $F_i$ as equivalence classes. 
\end{defi}

\begin{prop}
We have an action of $\mathrm{GL}(3,\mathbb{R})$ on $F_3$.
\end{prop}
\begin{proof}
By definition each element of the basis of $F_3$ is uniquely determined as the Lie brackets of an element of $F$ and an element of the derived algebra of $F$, $F^2$. Then we can define the action of $\mathrm{GL}(3,\mathbb{R})$ on $F_3$ as follows, let $M\in \mathrm{GL}(3,\mathbb{R})$ and $v$ an element of the basis of $F_3$ then  $M*v:=[Mu_1,[Mu_2,Mu_3]]$ where $u_1,u_2,u_3\in F_1$ are the unique vectors such that $v=[u_1,[u_2,u_3]]$, then we extend this action on $F_3$ by linearity.\end{proof}

\begin{rema}
We have hence a representation $\rho:\mathrm{GL}(3,\mathbb{R})\rightarrow \mathrm{GL}(F_3)$, and by abuse of notation let us call $\rho$ also the induced representation $\rho:\mathrm{SL}(3,\mathbb{R})\rightarrow \mathrm{GL}(F_3)$.
\end{rema}

\begin{prop}
There exists an $\mathrm{SL}(3,\mathbb{R})$-equivariant isomorphism $\varphi:\s\rightarrow F_3$, where $\mathrm{SL}(3,\mathbb{R})$ acts on $\s$ via the adjoint action and on $F_3$ via $\rho$.
\end{prop}
\begin{proof}
Let us denote by $E_{ij}$ the $3\times 3$ matrix whose $(i,j)$-th entry is $1$ and all the rest is 0. Then the matrices $E_i, i=1,\ldots,8$, where $E_1=E_{11}-E_{22}, E_2=E_{22}-E_{33}, E_3=E_{12}, E_4=E_{13}, E_{5}=E_{21}, E_6=E_{23}, E_7=E_{31}, E_8=E_{32}$, form a basis of $\s$ and the isomorphism $\varphi:\s\rightarrow F_3$ is given by $\varphi(E_1)=y_9+y_{11}, \varphi(E_2)=y_9-2y_{11}, \varphi(E_3)=-y_8, \varphi(E_4)=y_7, \varphi(E_5)=y_{12}, \varphi(E_6)=y_{10}, \varphi(E_7)=y_{14}, \varphi(E_8)=-y_{13}$. In order to check the equivariance of $\varphi$ it is sufficient to look at the level of the Lie algebras. If we let $\rho^*$ be the induced representation $\rho^*:\s\rightarrow \text{End}(F_3)$ defined by $\rho^*(E_i)v=\frac{d}{dt}_{|t=0}\rho(\exp(tE_i))v$ then the equivariant condition becomes $\varphi(\text{ad}(E_i)E_j)=\rho^*(E_i)\varphi(E_j)$ for all $i,j=1,\ldots,8$. This can be verified by easy calculations, for example $\varphi(\text{ad}(E_1)E_4)=\varphi(E_4)=y_7$ and $\rho^*(E_1)y_7=\frac{d}{dt}_{|t=0}[e^ty_1,[e^ty_1,e^{-t}y_2]]=y_7$.
\end{proof}

\begin{defi}\label{QuotientAlg}
For any $V\in\mathrm{Gr}(6,F_3)$ let us define the Lie algebra $\mathfrak{g}(V)$ to be the vector space $F_1\oplus F_2\oplus F_3/V$ with the Lie algebra structure induced by the one of $F$. Let us also define $P\in\Gr$ by $P:=\varphi^{-1}(V)^\perp$ where the orthogonal space is taken with respect to the Killing form on $\s$.
\end{defi}

\begin{rema}
 The Lie algebra $\mathfrak{g}(V)$ is a $8$-dimensional $3$-step nilpotent Lie algebra that is Carnot.  
\end{rema}

\begin{prop}\label{Isom&Orbit}
We have a bijection

\[\left\{ \begin{tabular}{c} Isomorphism classes of  \\ the Lie algebras $\mathfrak{g}(V)$ \end{tabular} \right\} \longleftrightarrow \left\{\begin{tabular}{c}     Orbits of elements  $P:=\varphi^{-1}(V)^\perp\in\Gr$ \\ under the adjoint action \end{tabular}  \right\}.\]

\end{prop}
\begin{proof}
As we have seen, being the Lie algebras $\mathfrak{g}(V)$ Carnot, each isomorphism between them is just induced by a linear isomorphism between their homogeneous parts of degree $1$. So, after fixing a basis for the two Lie algebras, an isomorphism between $\mathfrak{g}(V)$ and $\mathfrak{g}(V')$ is induced by an element of $\mathrm{GL}(3,\mathbb{R})$ that sends the $V$ to $V'$. Since scalar multiples of the identity act trivially on $\mathrm{Gr}(6,F_3)$ if there exists an element of  $\mathrm{GL}(3,\mathbb{R})$ that sends $V$ to $V'$ then there exists an element of $\mathrm{SL}(3,\mathbb{R})$ doing the same. Hence, after having identified the $6$-dimensional subspaces of $F_3$ with  elements $P\in\text{Gr}(2,\mathfrak{sl}(3,\mathbb{R}))$ and the action of $\mathrm{SL}(3,\mathbb{R})$ with the adjoint action, the orbit of each $P$ under the $\mathrm{SL}(3,\mathbb{R})$-adjoint action represents the isomorphism class of the associated Lie algebra.
\end{proof}

\subsection{The case $\alpha=0$}

In this case the family of Lie algebras $\mathfrak{g}_\mathbb{R}(\alpha,a,b,c)$ is quite easy to classify. Indeed we have the following.

\begin{prop}\label{alpha=0}
	The isomorphism classes of $\mathfrak{g}_\mathbb{R}(0,a,b,c)$ are  represented by $\mathfrak{g}_\mathbb{R}(0,1,0,0)$ if $a^2+bc>0$, $\mathfrak{g}_\mathbb{R}(0,0,1,-1)$ if $a^2+bc<0$, $\mathfrak{g}_\mathbb{R}(0,0,0,1)$ if $a^2+bc=0$ but $b\ne 0$ or $c\ne 0$ and $\mathfrak{g}_\mathbb{R}(0,0,0,0)$ if $a=b=c=0$.
\end{prop}
\begin{proof}
 Since we know that the Lie algebras $\mathfrak{g}_\mathbb{R}(0,a,b,c)$ are Carnot an isomorphism between them is induced by an element of $\G$.
Hence consider the injection $\iota :\mathrm{GL}(2,\mathbb{R})\rightarrow \G$ that associates to $\widetilde{g}\in \mathrm{GL}(2,\mathbb{R})$ the matrix \[\begin{pmatrix} \frac{1}{\det \widetilde{g}} & 0 \\ 0 & \widetilde{g}\end{pmatrix}\] then the map 
	\begin{align*} \psi: \mathfrak{sl}(2,\mathbb{R}) & \rightarrow \{\mathfrak{g}_\mathbb{R}(0,a,b,c)\} \\
	\begin{pmatrix} a & b \\ c & -a\end{pmatrix}& \mapsto \mathfrak{g}_\mathbb{R}(0,a,b,c)
	\end{align*}
	intertwines the two actions, i.e~for $A\in\mathfrak{sl}(2,\mathbb{R})$ we have $\Delta^3\psi(\widetilde{g}A\widetilde{g}^{-1})={^t\iota(\widetilde{g})}\psi(A)$.
Explicitly if $g\in\G$ is such that $g\cdot\mathfrak{g}_\mathbb{R}(0,a,b,c)=\mathfrak{g}_\mathbb{R}(0,a',b',c')$ we have, letting $\widetilde{g}=\begin{pmatrix} \lambda & \mu \\ \delta & \rho \end{pmatrix}$ and $\Delta=\det\widetilde{g}$,  
	
	\[a'=\Delta^2(\rho(a\lambda+c\delta)-\mu(b\lambda-a\delta)),\]
	\[b'=\Delta^2(b\lambda^2-2a\delta\lambda-c\delta^2),\]
	\[c'=\Delta^2(-b\mu^2+2a\rho\mu+c\rho^2).\]
Depending on the sign of $a^2+bc$ we can bring the matrix in $\mathfrak{sl}(2,\mathbb{R})$ to one of the following normal form 
	\[\begin{pmatrix} \sqrt{a^2+bc} & 0 \\ 0 & -\sqrt{a^2+bc}\end{pmatrix}, \begin{pmatrix} 0 & \sqrt{-(a^2+bc)} \\ -\sqrt{-(a^2+bc)} & 0 \end{pmatrix} ,\begin{pmatrix}
	0 & 0 \\ 1 & 0 
	\end{pmatrix}\text{ or } 0.\] 
	Finally since we are interested in these normal forms projectively we might assume the coefficients to be $1$.
\end{proof}

\begin{rema}
We can present the representatives of the isomorphism classes of $\mathfrak{g}_\mathbb{R}(0,a,b,c)$ in a more compact form as follows. The Lie algebras $\mathfrak{g}_\mathbb{R}(0,0,0,0)$ and $\mathfrak{g}_\mathbb{R}(0,0,\varepsilon,1)$ with $\varepsilon\in\mathbb{R}$ represent the isomorphism classes of $\mathfrak{g}_\mathbb{R}(0,a,b,c)$ with $\mathfrak{g}_\mathbb{R}(0,0,\varepsilon',1)\cong \mathfrak{g}_\mathbb{R}(0,0,\varepsilon,1)$ if and only if there exists $\alpha\in\mathbb{R}^*$ such that $\varepsilon'=\alpha^2\varepsilon$. Indeed notice that $\mathfrak{g}_\mathbb{R}(0,1,0,0)\cong\mathfrak{g}_\mathbb{R}(0,0,1,1)$.
\end{rema}

\subsection{The case $\alpha\not =0$}

In this case let us write the family $\mathfrak{g}_\mathbb{R}(\alpha,a,b,c)$ as follows.

\begin{defi}
When $\alpha\not=0$ a change of variables induced by $x'_1=\alpha x_1$ brings our family $\{\mathfrak{g}_\mathbb{R}(\alpha,a,b,c)\}_{\alpha,a,b,c}$ to one denoted by $\{\mathfrak{g}_\mathbb{R}(a,b,c)\}_{a,b,c}$, where the non zero Lie brackets, expressed in the basis $\{x_1,\ldots,x_8\}$ of $\mathbb{R}^8$ are: 
\[[x_1,x_2]=x_4,\qquad [x_1,x_3]=x_5,\qquad [x_2,x_3]=x_6,\]
\[[x_1,x_4]= x_7,\qquad [x_1,x_5]= x_8,\]
\[[x_2,x_4]= x_7,\qquad [x_2,x_5]=- x_8,\]
\[ [x_3,x_4]=- x_8,\qquad [x_3,x_5]=3 x_7,\]
\[[x_2,x_6]=ax_7+bx_8,\qquad [x_3,x_6]=cx_7-ax_8.\]
\end{defi}

\begin{rema}
	Notice that no Lie algebra $\mathfrak{g}_\mathbb{R}(\alpha,a,b,c)$ with $\alpha\ne 0$ is isomorphic to one of $\mathfrak{g}_\mathbb{R}(0,a,b,c)$. To see this let use define an invariant, called the characteristic sequence, of isomorphism classes of nilpotent Lie algebras.  The characteristic sequence of a nilpotent Lie algebra $\mathfrak{g}$ is defined as $c(\mathfrak{g})=\max \{c(x)\ |\ x\in\mathfrak{g}\setminus\mathfrak{g}^1\}$, where $c(x)$ is the decreasing sequence of dimensions of Jordan blocks of $ad(x)$. 
	Now for the family of Lie algebras $\{\mathfrak{g}_\mathbb{R}(a,b,c)\}$ the characteristic sequence is $(3,2,2,1)$ if one of $a$ or $b$ is $0$ and $(3,3,1,1)$ otherwise. While the family $\{\mathfrak{g}_\mathbb{R}(0,a,b,c)\}$ has characteristic sequence $(2,2,1,1,1,1)$ if either $a=b=0$ or $a=c=0$ and $(3,2,1,1,1)$ otherwise.
\end{rema}

Our family of Lie algebras $\mathfrak{g}_\mathbb{R}(a,b,c)$ is a particular case of the Lie algebras $\mathfrak{g}(V)$ defined in  Definition \ref{QuotientAlg}. We will now find explicitly the submanifold of $\Gr$ to which it corresponds.

\begin{defi}
Let us consider the following $6$-dimensional subspace of $F_3$ \[W(a,b,c)=\spn_\mathbb{R}\{y_9,y_{10}-y_7,y_{11}+y_8,y_{13}-3y_7,y_{12}-ay_7-by_8,y_{14}-cy_7+ay_8\}\subseteq F_3.\]
\end{defi}

\begin{prop}
For all $a,b,c\in\mathbb{R}$ we have an isomorphism of Lie algebras between $\mathfrak{g}_\mathbb{R}(a,b,c)$ and the Lie algebra whose stratification is given by $F_1,F_2$ and $F_3/W(a,b,c)$ and whose structure of Lie algebra is induced by the one of $F$.
\end{prop}
\begin{proof}
By imposing the conditions 
\[[x_1,x_4]= x_7=y_7,\qquad [x_1,x_5]= x_8=y_8,\qquad [x_1,x_6]=0=y_9,\]
\[[x_2,x_4]= x_7=y_{10},\qquad[x_2,x_5]=- x_8=y_{11},\qquad[x_2,x_6]=ax_7+bx_8=y_{12},\]
\[[x_3,x_5]=3 x_7=y_{13},\qquad[x_3,x_6]=cx_7-ax_8=y_{14}\]
we can find the generators for $W(a,b,c)$ and hence define a natural isomorphism.
\end{proof}

\begin{rema}
With this point of view we see our family of Lie algebras as a submanifold of the Grassmannian $\text{Gr}(6,F_3)$. Furthermore we can see that this submanifold is contained in $\{W\in\text{Gr}(6,F_3)\ |\ W_0\subseteq W\}\cong \text{Gr}(2,F_3/W_0)\cong \text{Gr}(2,4)$ where $W_0=\spn_\mathbb{R}\{y_9,y_{10}-y_7,y_{11}+y_8,y_{13}-3y_7\}$ and it corresponds actually to just one chart of $\text{Gr}(2,4)$.
\end{rema}

\begin{rema}
	Under $\varphi$, the isomorphism between $\s$ and $F_3$, the subspace $W(a,b,c)$ of $F_3$ corresponds to the subspace, that we denote by $V(a,b,c)$, of $\s$, whose basis is $\{2E_1+E_2,-E_4+E_6, E_1-E_2-3E_3, bE_3-aE_4+ E_5, -3E_4-E_8,-aE_3-cE_4+ E_7\}$. Using the Killing form on $\s$ we can identify $V(a,b,c)$ with a $2$-dimensional subspace $P(a,b,c)$ of $\s$ spanned by $\{ -E_2-bE_3+aE_4+ E_5, aE_3+cE_4-3 E_6+ E_7+ E_8\}$. 
	\end{rema}

	\begin{defi}\label{ThePlane}
	Let us define $P:=P(a,b,c)\in\text{Gr}(2,\mathfrak{sl}(3,\mathbb{R}))$ the $2$-dimensional subspace spanned by 
	\[u:=u(a,b,c)= -E_2-b E_3+aE_4+ E_5\text{ and }\]
	\[v:=v(a,b,c)=aE_3+cE_4-3 E_6+ E_7+E_8.\]
\end{defi}

The following is just a reformulation of Proposition \ref{Isom&Orbit} in our particular case.

\begin{prop}\label{Bijection}
We have a bijection

\[\left\{ \begin{tabular}{c} Isomorphism classes of  \\ the Lie algebras $\mathfrak{g}_\mathbb{R}(a,b,c)$ \end{tabular} \right\} \longleftrightarrow \left\{\begin{tabular}{c}     Orbits of elements  $P(a,b,c)$ \\ under the adjoint action \end{tabular}  \right\}.\]

\end{prop}

\subsubsection{The $\G$-action on $\{\mathfrak{g}_\mathbb{R}(a,b,c)\}\subseteq\Gr$}\label{Plucker}

In order to better understand the action we will embed the Grassmannian into a projective space and decompose it in $\G$-invariant subspaces.
We embed $\Gr$ in projective space by the Pl\"{u}cker embedding given by:
\[\iota: \Gr\rightarrow \mathbb{P}\left(\bigwedge^2\s\right).\]
Fixing $\{E_i\}_{i=1}^8$ as basis for $\s$ we can represent a generic element $V\in\Gr$ as a $2\times 8$ matrix whose lines are the vectors spanning it, then we have that $\iota(V)=[a_{ij}]$ where the Pl\"{u}cker coordinates $a_{ij}$ are the minors of the $2\times 2$ submatrix of $V$ obtained taking the $i$-th and $j$-th columns.\\

From classical representation theory, or from what we will show later, we have the following decomposition  of $\s$-representations 
\[\bigwedge^2\s=\s\oplus S^3(\mathbb{R}^3)\oplus S^3({\mathbb{R}^3}^*),\]
where $S^3(\mathbb{R}^3)$ is the $3$-rd symmetric power of $\mathbb{R}^3$. Let us call $\pi_1$ the projection to the first factor of the decomposition,
\begin{align*}
\pi_1: \bigwedge^2 & \s  \longrightarrow \s\\
& u_1\wedge u_2\mapsto [u_1,u_2].
\end{align*}

\begin{lemm}
The projection just defined, $\pi_1$, is a morphism of $\s$-representations.
\end{lemm}
\begin{proof}
Take $u_1,u_2\in\s$ and $x\in\s$ we have $\pi_1(x\cdot(u_1\wedge u_2))=\pi_1([x,u_1]\wedge u_2+u_1\wedge [x,u_2])=[[x,u_1],u_2]+[u_1,[x,u_2]]=[x,[u_1,u_2]]=x\cdot \pi_1(u_1\wedge u_2)$.
\end{proof}

In order to define the projection to the second and third factor of the decomposition let us denote by $\times$ the standard cross product on $\mathbb{R}^3$, by $\{e_i\}_{i=1}^3$ the standard basis for $\mathbb{R}^3$ and by $\odot$ the symmetric tensor product on $\mathbb{R}^3$. Let us then define $\pi_2$ as
\begin{align*}
\pi_2: \bigwedge^2 & \s \longrightarrow S^3(\mathbb{R}^3)\\
& u_1\wedge u_2\mapsto \sum_{i,j=1}^3u_1e_i\odot u_2e_j\odot (e_i\times e_j).
\end{align*}

\begin{lemm}
The second projection, $\pi_2$, is a morphism of $\s$-representations.
\end{lemm}
\begin{proof} It is sufficient to check it for $u_1=e_i^*\otimes u_1(e_i)$ and $u_2=e_j^*\otimes u_2(e_j)\in \s$. Let $x\in \s$ then on one side we have \[x\cdot\pi_2(u_1\wedge u_2)=x\cdot(u_1(e_i)\odot u_2(e_j)\odot (e_i\times e_j))=\]\[x(u_1(e_i))\odot u_2(e_j)\odot (e_i\times e_j)+u_1(e_i)\odot x(u_2(e_j))\odot (e_i\times e_j)+u_1(e_i)\odot u_2(e_j)\odot x(e_i\times e_j).\] 
On the other hand \[\pi_2(x\cdot(u_1\wedge u_2))=\pi_2((x\cdot u_1)\wedge u_2+u_1\wedge (x\cdot u_2))\] but $x\cdot u_1=x\cdot (e_i^*\otimes u_1(e_i))=-e_i^*x\otimes u_1(e_i)+e_i^*\otimes xu_1(e_i)$, hence \[\pi_2(-e_i^*x\otimes u_1(e_i)\wedge e_j^*\otimes u_2(e_j)+e_i^*\otimes xu_1(e_i)\wedge e_j^*\otimes u_2(e_j)-e_i^*\otimes u_1(e_i)\wedge e_j^*x\otimes u_2(e_j)+\]\[e_i^*\otimes u_1(e_i)\wedge e_j^*\otimes xu_2(e_j))=-u_1(e_i)\odot u_2(e_j)\odot (x^te_i\times e_j)+\]\[xu_1(e_i)\odot u_2(e_j)\odot (e_i\times e_j)-u_1(e_i)\odot u_2(e_j)\odot (e_i\times x^te_j)+u_1(e_i)\odot xu_2(e_j)\odot (e_i\times e_j).\] So it is left to prove that $x(e_i\times e_j)=-x^te_i\times e_j-e_i\times x^te_j$. For this consider a symmetric bilinear form $\left<\cdot,\cdot\right>$ on $\mathbb{R}^3$ and then by definition of cross product we have $\det(e_i,e_j,v)=\left<e_i\times e_j,v\right>$ for any $e_i,e_j,v\in\mathbb{R}^3$. Taking $g\in\G$ we have $\det(g^te_i,g^te_j,g^tv)=c$ for some constant $c\in\mathbb{R}$ that does not depend on $g$. Hence differentiating this expression at the identity we get for $x\in \s$ that  $\det(x^te_i,e_j,v)+\det(e_i,x^te_j,v)+\det(e_i,e_j,x^tv)=0$ hence $\left<x^te_i\times e_j,v\right>+\left<e_i\times x^te_j,v\right>+\left<e_i\times e_j,x^tv\right>=0$ for any $v\in\mathbb{R}^3$ and the result follows.
\end{proof}

We also define
\begin{align*}
\pi_3: \bigwedge^2 & \s \longrightarrow S^3({\mathbb{R}^3}^*)\\
& u_1\wedge u_2\mapsto \sum_{i,j=1}^3e_i^*\odot e_j^*\odot (u_1e_i\times u_2e_j)^*.
\end{align*}

\begin{rema}
Notice that $\pi_3(u_1\wedge u_2)=\pi_2({^tu_1}\wedge {^tu_2})$. Hence since $\pi_2$ is $\s$-equivariant we get that $\pi_3$ is $\s$-equivariant as well since $\pi_3(gu_1g^{-1}\wedge gu_2g^{-1})=\pi_2({^tg^{-1}} {^tu_1} {^tg}\wedge {^tg^{-1}} {^tu_2} {^tg})= {^tg^{-1}}\pi_2({^tu_1}\wedge {^tu_2})={^tg^{-1}}\pi_3(u_1\wedge u_2)$.
\end{rema}

\begin{rema}\label{ExplicitProject}
Let us identify the elements of both $S^3(\mathbb{R}^3)$ and $S^3({\mathbb{R}^3}^*)$ with ternary cubics in the variables $x,y,z$ for convenience. Let $P\in\Gr$, each projection reads in the Pl\"{u}cker coordinates $a_{ij}$ as follows
\[\pi_1(P)=\begin{pmatrix} a_{35}+a_{47} & 2a_{13}-a_{23}+a_{48} & a_{14}+a_{24}+a_{36} \\ -2a_{15}+a_{25}+a_{67} & -a_{35}+a_{68} & -a_{16}+2a_{26}-a_{54} \\ -a_{17}-a_{27}-a_{58} & a_{18}-2a_{28}-a_{37} & -a_{47}-a_{68}\end{pmatrix},\]
\begin{align*}  \pi_2(P)=\ & a_{34}x^3+((-2a_{14}+a_{24}+a_{36})y+(a_{13}+a_{23}-a_{48})z)x^2 +\\
&((-2a_{16}+a_{26}+a_{45})y^2+(3a_{12}-a_{35}+a_{47}-a_{68})yz  +(a_{18}+a_{28}-a_{37})z^2)x\\ &-a_{56}y^3+(a_{15}-2a_{25}+a_{67})y^2z+(a_{17}-2a_{27}+a_{58})yz^2+a_{78}z^3,
\end{align*}
\begin{align*} \pi_3(P)=\ & a_{57}x^3+((-2a_{17}+a_{27}+a_{58})y+(a_{15}+a_{25}+a_{67})z)x^2+ \\ &((-2a_{18}+a_{28}-a_{37})y^2+ (3a_{12}+a_{35}-a_{47}+a_{68})yz+(a_{16}+a_{26}+a_{45})z^2)x\\ &-a_{38}y^3+(a_{13}-2a_{23}-a_{48})y^2z+(a_{14}-2a_{24}+a_{36})yz^2+a_{46}z^3.
\end{align*}
We will fix also the notations $\pi_2(P)=a_{34}x^3+p_1(y,z)x^2+p_2(y,z)x+p_3(y,z)$ and $\pi_3(P)=a_{57}x^3+q_1(y,z)x^2+q_2(y,z)x+q_3(y,z)$. 
\end{rema}

Using the explicit projections one can see that we have the following.

\begin{prop}
	The morphism $(\pi_1,\pi_2,\pi_3):\bigwedge^2\s\rightarrow \s\oplus S^3(\mathbb{R}^3)\oplus S^3({\mathbb{R}^3}^*)$ is an $\s$-equivariant isomorphism.
\end{prop}

We can now use the identification of Proposition $\ref{Bijection}$ and the decomposition of $\mathrm{Gr}(2,\mathfrak{sl}(3,\mathbb{R}))$ in invariant subspaces under the $\mathrm{SL}(3,\mathbb{R})$-action in order to classify the family of Lie algebras $\mathfrak{g}_\mathbb{R}(a,b,c)$.

\begin{prop}\label{w_0not0,gamma2not0}
	The Lie algebras $\mathfrak{g}_\mathbb{R}(a,b,c)$ and $\mathfrak{g}_\mathbb{R}(a',b',c')$ are isomorphic if and only if $a'=\pm a,b'=b$ and $c'=c$.
\end{prop}
\begin{proof}
Under the identification $\mathbb{P}(\bigwedge^2\mathfrak{sl}(3,\mathbb{R}))=\mathbb{P}\left(\mathfrak{sl}(3,\mathbb{R})\oplus S^3(\mathbb{R}^3)\oplus S^3({\mathbb{R}^3}^*)\right)$ we write the elements $P(a,b,c)$ as $[(M(a,b,c),f_1(a,b,c),f_2(a,b,c))]$ where 
\[M(a,b,c)=\begin{pmatrix}0 & 2a & -c+3b \\ 0 & a & c+6 \\ 0 & b+2 & -a\end{pmatrix},\]
 \begin{align*} & f_1(a,b,c)=  (-a^2-bc)e_1\odot  e_1\odot e_1+ (3b-c) e_1\odot  e_1\odot  e_2 -2a e_1\odot e_1\odot e_3 \\   +(3-c) & e_1\odot e_2\odot e_2 +2a e_1\odot e_2 \odot e_3 +  (b-1) e_1\odot e_3\odot e_3 +3 e_2\odot e_2 \odot e_2 +3 e_2\odot e_3 \odot e_3\end{align*}
 and \begin{align*}& f_2(a,b,c)=e_1^*\odot e_1^*\odot e_1^*+(b-1)  e_1^*\odot e_2^*\odot e_2^*  -2a e_1^*\odot e_2^*\odot e_3^* \\ +(3-c)e_1^*\odot e_3^*\odot e_3^* & +b e_2^*\odot e_2^* \odot e_2^* +a e_2^*\odot e_2^*\odot e_3^*  +(2c+3b) e_2^*\odot e_3^* \odot e_3^* -3 a e_3^*\odot e_3^* \odot e_3^*.\end{align*} 
 Let us identify furthermore the elements of both $S^3(\mathbb{R}^3)$ and $S^3({\mathbb{R}^3}^*)$ with ternary cubics in the variables $x,y,z$ for convenience.
 Now an element $g$ of $\mathrm{SL}(3,\mathbb{R})$ that sends $P(a,b,c)$ to $P(a',b',c')$ should preserve the kernel of $M(a,b,c)$. If $M(a,b,c)$ is of rank $2$ then $\ker M(a,b,c)=\mathbb{R}e_1$. Hence $g$ has the form 
 \[\begin{pmatrix} (\sigma_1\sigma_4-\sigma_2\sigma_3)^{-1} & \mu & \nu \\ 0 & \sigma_1 & \sigma_2 \\ 0 & \sigma_3 & \sigma_4\end{pmatrix}\]
 and we are left with an action of the group $\mathrm{GL}(2,\mathbb{R})\ltimes\mathbb{R}^2$. Let us notice that this group induces an action of $\mathrm{GL}(2,\mathbb{R})$ on $S^3(\mathbb{C}^2)$, the ternary cubics in two variables $y$ and $z$. More precisely, if $P\in\Gr$ and $\pi_2(P)=p_0x^3+p_1(y,z)x^2+p_2(y,z)x+p_3(y,z)$ let us call $\mathrm{pr}(P)=p_3(y,z)$. Then if  $g=(h,v)\in\mathrm{GL}(2,\mathbb{R})\ltimes\mathbb{R}^2$ we have $\pi_2(g\cdot P)=h\cdot \mathrm{pr}(P)$. Let us notice that for all $a,b,c$ we have $\mathrm{pr}(P(a,b,c))=[3(y^3+yz^2)]$.
 Since we search for the $g\in\mathrm{GL}(2,\mathbb{R})\ltimes\mathbb{R}^2$ such that $g\cdot f_1(a,b,c)=f_1(a',b',c')$ then in particular we want $g$ such that $[h\cdot \mathrm{pr}(P(a,b,c))]=[3(y^3+yz^2)]$. Hence since \[h\cdot \mathrm{pr}(P(a,b,c))=3\sigma_1(\sigma_1^2+\sigma_2^2)y^3+3(3\sigma_1^2\sigma_3+2\sigma_1\sigma_2\sigma_4+\sigma_3\sigma_2^2)y^2z\]\[+3(3\sigma_1\sigma_3^2+2\sigma_2\sigma_3\sigma_4+\sigma_1\sigma_4^2)yz^2+3\sigma_3(\sigma_3^2+\sigma_4^2)z^3\] then we should have $3\sigma_3(\sigma_3^2+\sigma_4^2)=0$,  i.e.~$\sigma_3=0$ and $3(3\sigma_1^2\sigma_3+2\sigma_1\sigma_2\sigma_4+\sigma_3\sigma_2^2)=6\sigma_1\sigma_2\sigma_4=0$, i.e.~$\sigma_2=0$, also $3\sigma_1(\sigma_1^2+\sigma_2^2)=3(3\sigma_1\sigma_3^2+2\sigma_2\sigma_3\sigma_4+\sigma_1\sigma_3^2)$ that is $\sigma_1^2=\sigma_4^2$. 
 In a similar way we want $g$ such that $g\cdot f_2(a,b,c)=f_2(a',b',c')$. We can compare the terms with at least one $x^2$ in $f_2(a,b,c)$, that is just $x^3$, and in $g\cdot f_2(a,b,c)$ that are  \[(\sigma_1\sigma_4)^3x^3-3(\sigma_1\sigma_4)^2\mu\sigma_4x^2y-3(\sigma_1\sigma_4)^2\nu\sigma_1x^2z.\]
 Then we must have $\sigma_1^2\sigma_4^3\mu=0$, i.e.~$\mu=0$ and $\sigma_1^3\sigma_4^2\nu=0$, i.e.~$\nu=0$. Finally the coefficient of $y^3$ in $h\cdot \mathrm{pr}(P(a,b,c))$ is $3$ times the one of $x^3$ in $g\cdot f_2(a,b,c)$ hence $3\sigma_1^3=3(\sigma_1\sigma_4)^3$, that implies $\sigma_1^3=(\pm\sigma_1^2)^3$ so $\sigma_1=\pm 1$.\\
  Hence $g\in\mathrm{SL}(3,\mathbb{R})$ is such that $g\cdot P(a,b,c)=P(a',b',c')$ if and only if $g=\mathrm{Id}$ or $g=g_1$ where \[g_1=\begin{pmatrix} -1 & 0 & 0 \\ 0 & -1 & 0 \\ 0 & 0 & 1\end{pmatrix}.\]
  The element $g_1$ sends $P(a,b,c)$ to $P(-a,b,c)$.\\
  Finally let us treat the case where the rank of $M(a,b,c)$ is strictly less than $2$, i.e. when $a=0$ and $b=-2$. If $c=-6$ the matrix $M(0,-2,-6)$ is $0$ then of course $P(0,-2,-6)$ is in the same orbit of no other point. If instead $c\not = -6$ let us act on the point $P(0,-2,c)$ in order to have $M(0,-2,c)$ in a normal form. Then $P(0,-2,c)$ is in the same orbit as
    \[\begin{split}
      [(N,p_1,p_2)]=  \left[\begin{pmatrix} 0 & 1 & 0 \\ 0 & 0 & 0 \\ 0 & 0 & 0\end{pmatrix},\frac{3}{(c+6)^2}x^3+\frac{12-c}{c+6}x^2z+3xy^2-6xyz-3(c-4)xz^2, \right. \\
      \left.3(c-4)xy^2-6xyz-3xz^2+\frac{4-c}{c+6}y^3+\frac{6-c}{c+6}y^2z+\frac{3}{c+6}yz^2+\frac{1}{c+6}z^3\right].
      \end{split}\]
  A point in the same orbit should have the same coefficients of the cubics modulo the actions of the stabiliser of $N$ that is 
  \[\mathrm{Stab}_{\mathrm{SL}(3,\mathbb{R})}\left(\begin{pmatrix} 0 & 1 & 0 \\ 0 & 0 & 0 \\ 0 & 0 & 0\end{pmatrix}\right)=\left\{\begin{pmatrix} \sigma_1 & \sigma_2 & \sigma_3 \\ 0 & \sigma_1 & 0 \\ 0 & \sigma_4 & \sigma_1^{-2}\end{pmatrix}\ \big|\ \sigma_1\in\mathbb{R}^*,\sigma_2,\sigma_3,\sigma_4\in\mathbb{R}\right\}.\] 
  Taking $g=\begin{pmatrix} \sigma_1 & -\frac{\sigma_1}{3(c+6)} & -\frac{\sigma_1}{3(c+6)} \\ 0 & \sigma_1 & 0 \\ 0 & \sigma_1^{-2} & \sigma_1^{-2}\end{pmatrix}\in\mathrm{Stab}_{\mathrm{SL}(3,\mathbb{R})}(N)$ we obtain 
  \[[g\cdot (N,p_1,p_2)]= \left[\begin{pmatrix} 0 & 1 & 0 \\ 0 & 0 & 0 \\ 0 & 0 & 0\end{pmatrix},zx^2+\left(3\sigma_1^3y^2+\frac{9-3c}{\sigma_1^3}z^2\right)x,\left(\frac{3c-9}{\sigma_1^3}y^2-3\sigma_1^3z^2\right)x\right].\] The parameter $c$ is then an invariant and all the points $P(0,-2,c)$ are pairwise not isomorphic.  
\end{proof}

Putting together Proposition \ref{w_0=0}, \ref{gamma_2=0}, \ref{alpha=0} and \ref{w_0not0,gamma2not0} we see that the proof of Theorem \ref{ClassLieAlg} is achieved.

\section{Lattices in unipotent groups}\label{lattices}

In Section \ref{ProperlyDisc} we looked at crystallographic groups $\Gamma$ and studied the abelian by cyclic case. We are then left with the virtually nilpotent case. We know that they are lattices in unipotent simply transitive subgroups of $\mathcal{H}(3,1)$. Since we have listed all the possible unipotent Lie subgroups of $\mathcal{H}(3,1)$ that act simply transitively on $\mathfrak{a}(V)$ we are interested in studying their lattices, in particular the question of existence of lattices and of their classification up to finite index. For lattices in unipotent groups the theory is illustrated by Malcev's theorems. The first one concerns conditions of existence of lattices in nilpotent Lie groups and the second a criterion of classification of lattices up to finite index.

\begin{theo}[{\cite[Section II,Theorem 2.12]{Rag}}]
Let $U$ be a simply connected nilpotent Lie group and $\mathfrak{u}$ its Lie algebra then $U$ admits a lattice if and only if $\mathfrak{u}$ admits a basis with respect to which the structure constants are rational.
\end{theo}

\begin{rema}
The statement is equivalent to say that $\mathfrak{u}$ admits a $\mathbb{Q}$-form, i.e.~a rational Lie subalgebra $\mathfrak{u}_\mathbb{Q}$ such that $\mathfrak{u}_\mathbb{Q}\otimes_\mathbb{Q}\mathbb{R}\cong \mathfrak{u}$. 
\end{rema}

\begin{theo}[{\cite{Cor},\cite{Margulis}}]
Let $\Gamma_1$ and $\Gamma_2$ be lattices in a nilpotent Lie group $G$. Then $\Gamma_1$ and $\Gamma_2$ induce isomorphic rational structures on $\mathfrak{g}$ if and only if they are abstractly commensurable: they have finite index subgroups $\Delta_1\le \Gamma_1$ and $\Delta_2\le\Gamma_2$ that are isomorphic.
\end{theo}

We begin the study of lattices in the unipotent simply transitive subgroups of $\mathcal{H}(3,1)$ by looking at the lattices in the subgroups that correspond to the Lie algebras $\mathfrak{g}_\mathbb{R}(\alpha,a,b,c)$. Thanks to the theorems just stated, in order to classify these lattices up to abstract commensurability, we look at the isomorphism classes of the $\mathbb{Q}$-forms in the Lie algebras $\mathfrak{g}_\mathbb{R}(\alpha,a,b,c)$.

\begin{prop}\label{Q-form general case}
Every $\mathbb{Q}$-form of $\mathfrak{g}_\mathbb{R}(a,b,c)$ is isomorphic to exactly one of the Lie algebras of the form $\mathfrak{g}_\mathbb{Q}(e,f,g,h,j,k,l)$ or $\mathfrak{g}_\mathbb{Q}(0,-2,c)$ of Appendix \ref{LieBracket}.
\end{prop}
\begin{proof}
Let $\mathfrak{h}$ be a $\mathbb{Q}$-form of $\mathfrak{g}_\mathbb{R}(a,b,c)$. Let $P(a,b,c)\in\mathrm{Gr}(2,\mathfrak{sl}(3,\mathbb{R}))$, see Definition \ref{ThePlane}, be the plane associated to $\mathfrak{g}_\mathbb{R}(a,b,c)$ under the bijection of Proposition \ref{Bijection}. By definition $\mathfrak{h}\otimes\mathbb{R}\cong\mathfrak{g}_\mathbb{R}(a,b,c)$. By \cite[Theorem 3.15]{Cornulier} $\mathfrak{h}$ is Carnot over $\mathbb{Q}$ if and only if $\mathfrak{h}\otimes\mathbb{R}$ is Carnot over $\mathbb{R}$. Hence $\mathfrak{h}=\mathfrak{g}(V)$, see Definition \ref{QuotientAlg}, to which we associate an element $P\in\mathrm{Gr}(2,\mathfrak{sl}(3,\mathbb{Q}))$ by Proposition \ref{Bijection}. Furthermore $P$ is in the same real orbit of $P(a,b,c)$ under the action of $\mathrm{SL}(3,\mathbb{R})$. We will use the same terminology as in Section \ref{Plucker} and in particular of Remark \ref{ExplicitProject}. The idea of the proof is the same as in Proposition \ref{w_0not0,gamma2not0} but over $\mathbb{Q}$.

Knowing that $P$ is in the same real orbit of $P(a,b,c)$ implies that $0$ is an eigenvalue of $\pi_1(P)$. We can then act on it in order to have $\pi_1(P)=\left[\begin{pmatrix} 0 & * & * \\ 0 & * & * \\ 0 & * & * \end{pmatrix}\right]$. 

Assuming $P$ is such that $\pi_1(P)$ has rank $2$ allows us to reduce the action to the one of $\mathrm{GL}(2,\mathbb{Q})\ltimes \mathbb{Q}^2$. Using the action of $\mathbb{Q}^2$ we can also achieve that $q_1(y,z)=0$ for $\pi_3(P)$. Hence we reduce the action to the one of the stabiliser of $q_1(y,z)=0$ in $\mathrm{GL}(2,\mathbb{Q})\ltimes \mathbb{Q}^2$, that is $\mathrm{GL}(2,\mathbb{Q})$. The subvariety defined by these conditions is defined by the following equations in the Pl\"{u}cker coordinates $\{a_{35}+a_{47}=0, -2a_{15}+a_{25}+a_{67}=0, a_{17}+a_{27}+a_{58}=0, a_{15}+a_{25}+a_{67}=0, -2a_{17}+a_{27}+a_{57}=0\}$, compare with \ref{ExplicitProject}. The Lie algebra $\mathfrak{g}_\mathbb{R}(a,b,c)$ has also the property that $a_{57}=\mathrm{coeff}(\pi_3(P(a,b,c)),x^3)\ne 0$, hence the same is true for $P$, which is in the same real orbit.  Working in the affine chart of the Grassmannian defined by $a_{57}\ne 0$ and solving the above equations  we find that $P$ is in the orbit of
\begin{equation}\label{matrice}\begin{pmatrix} 0 & -e & -k & j & 1 & -g & 0 & -f \\ 0 & -g & j & l & 0 & -h & 1 & e \end{pmatrix}.\end{equation}
We still have the action of $\mathrm{GL}(2,\mathbb{Q})$. This action induces an action over the binary cubic in the variables $y,z$ of $\pi_2(P)$: $p_3(y,z)=ey^3+3fy^2z+3gyz^2+hz^3\in\mathbb{Q}[y,z]$. This cubic is in the same real orbit of the corresponding cubic of $\pi_2(P(a,b,c))$, that is $3y^3+3yz^2$, hence it is a cubic with negative discriminant. We are then interested in finding normal forms for the action of $\mathrm{GL}(2,\mathbb{Q})$ over binary cubics with negative discriminant. We can proceed as follows. First of all we can assume without loss of generality that $e\ne 0$, using the action of $\mathrm{GL}(2,\mathbb{Q})$. We consider the associated univariate polynomial $p_3(y,1)$. Since the discriminant is negative $p_3(y,1)$ has only one real root, $\rho\in\mathbb{R}$. If $\rho$ is rational, using the action of $\mathrm{GL}(2,\mathbb{Q})$ on $\rho$ by homographies, we can bring $\rho$ to $0$ and reduce $p_3(y,z)$ to $y(ey^2+3fyz+3gz^2)$. Reducing further the quadratic polynomial we find that the normal forms for $p_3(y,z)$ over $\mathbb{Q}$ is $y(ey^2+gz^2)$ with $y(ey^2+gz^2)$ in the same orbit as $y(e'y^2+g'z^2)$ if and only if there exists $\sigma_1,\sigma_2\in\mathbb{Q}^*$ such that $e'=\sigma_1^3e$ and $g'=\sigma_1\sigma_2^2g$. If $p_3(y,1)$ is irreducible over the rationals then we can consider the field $K$, a cubic extension of $\mathbb{Q}$, obtained by adjoining the root $\rho$ to $\mathbb{Q}$. Cubic fields have been classified in \cite{Marques}. Their result, \cite[Corollary 1.3]{Marques}, is that in every such extension one can find an element $\theta$ whose minimal polynomial is either \begin{equation}\label{f1}y^3-t\end{equation} or \begin{equation}\label{f2}y^3-3y-t\end{equation} with $t\in\mathbb{Q}$.
A polynomial as in $(\ref{f1})$ has always negative discriminant and is irreducible over $\mathbb{Q}$ if and only if $t\notin\mathbb{Q}^3$. A polynomial as in $(\ref{f2})$ has negative discriminant if and only if $t^2-4>0$ and is irreducible over $\mathbb{Q}$ if and only if $t$ does not belong to the image of the function $x^3-3x$ defined over the rational. 
Furthermore two polynomials in the form $(\ref{f1})$, $y^3-t_1$ and $y^3-t_2$, generate the same extension if and only if there exists $\mu\in \mathbb{Q}$ such that $t_2=\mu^3t_1^j$ with $j=1,2$. Two polynomials in the form $(\ref{f1})$, $y^3-3y-t_1$ and $y^3-3y-t_2$, generate the same extension if and only if $t_2=-3t_1\alpha^2\beta+t_1\beta^3+6\alpha+\alpha^3t_1^2-8\alpha^3$ with $\alpha,\beta\in\mathbb{Q}$ such that $\alpha^2+t_1\alpha\beta+\beta^2=1$. Finally, a polynomial as in $(\ref{f2})$ generates the same extension as one in the form $(\ref{f1})$ if the polynomial $x^2+tx+1$ has a root in $\mathbb{Q}$, i.e~if $t^2-4\in\mathbb{Q}^2$. So $\rho$ can be written as $m\theta^2+n\theta+r$ with $m,n,r\in\mathbb{Q}$ and $\theta$ satisfying $(\ref{f1})$ or $(\ref{f2})$. Observe that for $g\in\mathrm{GL}(2,\mathbb{Q})$ the element $g\cdot p_3(y,z)$ equals a minimal polynomial of $g\cdot \rho$ where again we act on $\rho$ by homographies. Via this action we can bring $\rho=m\theta^2+n\theta+r$ to either $m\theta^2+\theta$ or $\theta^2$. These two cases correspond to different orbits and different values of $m$ also correspond to different orbits. We are left with finding a minimal polynomial for $m\theta^2+\theta$ and $\theta^2$. If $\theta$ satisfies $(\ref{f1})$ the element $m\theta^2+\theta$ has minimal polynomial \begin{equation}\label{f3}y^3-3mty-m^3t^2-t\end{equation} and the element $\theta^2$ has minimal polynomial \begin{equation}\label{f4}y^3-t^2.\end{equation} If instead $\theta$ is a root of $(\ref{f2})$ then $m\theta^2+\theta$ has minimal polynomial \begin{equation}\label{f5}y^3-6my^2+(9m^2-3tm-3)y+3tm^2-m^3t^2-t\end{equation} and $\theta^2$ has minimal polynomial \begin{equation}\label{f6}y^3-6y^2+9y-t^2.\end{equation} 
The way to find the minimal polynomial of a sum of algebraic numbers is classical and uses the resultant between the two minimal polynomials of the corresponding algebraic numbers. Acting with $g\in\mathrm{GL}(2,\mathbb{Q})$ on $p_3(y,z)$ multiplies the coefficient $a_{57}$ by the inverse of the determinant of $g$. Hence a plane as in $(\ref{matrice})$ is in normal form if there exist $e,\lambda\in\mathbb{Q}^*$ such that $\frac{1}{e}\lambda^3 p_3(y,z)$ is in one of the normal forms $(\ref{f3}),(\ref{f4}),(\ref{f5})$ or $(\ref{f6})$. We still have the action by homotheties $\mathrm{diag}(\mu^{-2},\mu,\mu)\in\mathrm{SL}(3,\mathbb{Q})$ with $\mu\in\mathbb{Q}^*$ that does not change $\rho$ but indeed transform $P$, hence the number $e'=\frac{1}{e}\lambda^3$ is uniquely determined up to multiplication by $\mu^9$ for some $\mu\in\mathbb{Q}^*$. Then notice that once $p_3(y,z)$ is in a reduced form the other coefficients of $P$ cannot be reduced anymore and they describe different Lie algebras. This is because the stabiliser of a binary cubic with non-zero discriminant is trivial. Calculating the Lie algebra associated to $(\ref{matrice})$ we find the families of Lie algebras $\mathfrak{g}_\mathbb{Q}(e,f,g,h,j,k,l)$ as in Appendix \ref{LieBracket}. 

Assume now that $\mathfrak{h}$ is a $\mathbb{Q}$-form of $\mathfrak{g}_\mathbb{R}(0,-2,c)$ with $c\ne -6$. This family of Lie algebras is characterised by the fact that $\pi_1(P(0,-2,c))$ has rank $1$. Remember from Proposition \ref{w_0not0,gamma2not0} that we can bring $P(0,-2,c)$ via the action to
\begin{equation}\label{n1}\left[\begin{pmatrix} 0 & 1 & 0 \\ 0 & 0 & 0 \\ 0 & 0 & 0 \end{pmatrix},\  zx^2+(3y^2+(9-3c)z^2)x,\ ((3c-9)y^2-3z^2)x\right].\end{equation}
We can then assume that the plane $P$ associated to $\mathfrak{h}$ is such that $\pi_1(P)$ is the same as in $(\ref{n1})$. Then we are left with the action of the stabiliser of such a matrix that is a subgroup of $\mathrm{GL}(2,\mathbb{Q})\ltimes\mathbb{Q}^2$, see proof of Proposition \ref{w_0not0,gamma2not0}. Since for $\mathfrak{g}_\mathbb{R}(0,-2,c)$ we have $p_3(y,z)=0$, $a_{57}=0$ and $q_1(y,z)=0$, then we must have the same for $P$. Furthermore there exists $g$ in the stabiliser of $\pi_1(P)$ such that $g\cdot P$ satisfies $\mathrm{coeff}(p_1(y,z),y)=\mathrm{coeff}(p_2(y,z),yz)=\mathrm{coeff}(q_3(y,z),z^3)=0$. Solving these equations we see that 
\[g\cdot P=\begin{pmatrix} 2 & 1 & 0 & 0 & 0 & 0 & 0 & 0 \\ 0 & 0 & 1 & 0 & 0 & 3 & 0 & c \end{pmatrix}.\]
One can see that $[\pi_1(P),\pi_2(P),\pi_3(P)]$ is in the same form as $(\ref{n1})$, hence the associated Lie algebra is $\mathfrak{g}_\mathbb{Q}(0,-2,c)$. 

Finally if $\mathfrak{h}$ is a $\mathbb{Q}$-form of $\mathfrak{g}_\mathbb{R}(a,b,c)$ such that $\pi_1(P(a,b,c))=0$ then $a=0,b=-2$ and $c=-6$. Then $P(0,-2,-6)$ is generated by the following two matrices
\[u=\begin{pmatrix} 2 & 0 & 0 \\ 0 & -1 & 0 \\ 0 & 0 & -1\end{pmatrix}\text{ and } v=\begin{pmatrix}
0 &0 & 0 \\ 0 & 0 & -9 \\ 0 & 1 & 0 
\end{pmatrix}.\]
As $P$ is a $\mathbb{Q}$-form of $P(0,-2,-6)$, for all $X\in P$ there exist $\alpha,\beta\in\mathbb{R}$ such that $\alpha u+\beta v=X$. Since $X$ is a rational matrix, $\alpha$ and $\beta$ must be rational. Hence $P$ is generated by $u$ and $v$ over $\mathbb{Q}$. That is $\mathfrak{h}$ is isomorphic to $\mathfrak{g}_\mathbb{Q}(0,-2,-6)$. 
\end{proof}

\begin{rema}
Apart from the family $\mathfrak{g}_\mathbb{R}(a,b,c)$, it is clear from the presentation we have given of representatives of the isomorphism classes of the other nilpotent Lie algebras, that all the others admit $\mathbb{Q}$-forms. Furthermore from the simple remark that there are uncountably many non isomorphic Lie algebras $\mathfrak{g}_\mathbb{R}(a,b,c)$ but that there are countably many $\mathbb{Q}$-forms we know that most Lie algebras $\mathfrak{g}_\mathbb{R}(a,b,c)$ that do not admit $\mathbb{Q}$-forms.
\end{rema}

The following proposition answers the question of $\mathbb{Q}$-isomorphism classes of $\mathbb{Q}$-forms in the Lie algebras $\mathfrak{g}_\mathbb{R}(0,a,b,c)$.

\begin{prop}\label{Q form alpha=0}
The $\mathbb{Q}$-isomorphism classes of $\mathbb{Q}$-forms of the family $\mathfrak{g}_\mathbb{R}(0,a,b,c)$ are \\ $\mathfrak{g}_\mathbb{Q}(0,0,0,0)$ and $\mathfrak{g}_\mathbb{Q}(0,0,\varepsilon,1)$ with $\varepsilon\in\mathbb{Q}$, where $\mathfrak{g}_\mathbb{Q}(0,0,\varepsilon',1)\cong \mathfrak{g}_\mathbb{Q}(0,0,\varepsilon,1)$ if and only if there exists $\alpha\in\mathbb{Q}^*$ such that $\varepsilon'=\alpha^2\varepsilon$. 
\end{prop}
\begin{proof}
We will first proceed as in the previous proposition and prove that if $\mathfrak{g}_\mathbb{R}(0,a,b,c)$ admits a $\mathbb{Q}$-form $\mathfrak{h}$, then $\mathfrak{h}\cong\mathfrak{g}_\mathbb{Q}(0,a',b',c')$ with $a',b',c'\in\mathbb{Q}$. Secondly we will study the isomorphism classes of $\mathfrak{g}_\mathbb{Q}(0,a',b',c')$. Let $\mathfrak{h}$ be a $\mathbb{Q}$-form of $\mathfrak{g}_\mathbb{R}(0,a,b,c)$. Then by definition $\mathfrak{h}\otimes\mathbb{R}\cong\mathfrak{g}_\mathbb{R}(0,a,b,c)$. Hence $\mathfrak{h}=\mathfrak{g}(V)$, see Definition \ref{QuotientAlg}. One can see that the $2$-dimensional plane associated with $\mathfrak{g}_\mathbb{R}(0,a,b,c)$ is 
\[P(0,a,b,c)=\begin{pmatrix} 0 & 0 & -b & a & 1 & 0 & 0 & 0 \\ 0 & 0 & a & c & 0 & 0 & 1 & 0 \end{pmatrix},\] 
whose projections are
\[\left[\begin{pmatrix} 0 & 0 & 0 \\ 0 & a & c \\ 0 & b & -a \end{pmatrix},\  -(a^2+bc)x^3+(-cy^2+2ayz+by^2)x,\ x^3+(by^2-2ayz-cz^2)x\right].\]
Let $P\in\mathrm{Gr}(2,\mathfrak{sl}(3,\mathbb{Q}))$ be the $2$-plane in $\mathfrak{sl}(3,\mathbb{Q})$ corresponding to $V$ such that $\mathfrak{h}=\mathfrak{g}(V)$, then $P$ is in the same real orbit as $P(0,a,b,c)$. This implies that $0$ is an eigenvalue of $\pi_1(P)$. We can then act on it in order to have $\pi_1(P)=\left[\begin{pmatrix} 0 & * & * \\ 0 & * & * \\ 0 & * & * \end{pmatrix}\right]$. Assuming $P$ is such that $\pi_1(P)$ has rank $2$ allows us to reduce the action to $\mathrm{GL}(2,\mathbb{Q})\ltimes \mathbb{Q}^2$. The plane $P(0,a,b,c)$ has the further property that $p_3(y,z)=0$ and $a_{57}=\mathrm{coeff}(\pi_3(P(0,a,b,c)),x^3)\ne 0$, hence the same is true for $P$. Furthermore as in the previous proposition we can use the action to achieve $p_1(y,z)=0$ for $\pi_2(P)$. Working in the affine chart of the Grassmannian defined by $a_{57}\ne 0$ and solving the above equations we find that $P$ is in the orbit of
\begin{equation}\begin{pmatrix} 0 & 0 & b' & a' & 1 & 0 & 0 & 0 \\ 0 & 0 & a' & c' & 0 & 0 & 1 & 0 \end{pmatrix}\end{equation}
with $a',b',c'\in\mathbb{Q}$.
This means that $\mathfrak{h}\cong\mathfrak{g}_\mathbb{Q}(0,a',b',c')$.

Suppose now that $\pi_1(P)$ has rank $1$. We have seen in Proposition \ref{alpha=0} that in this case $\mathfrak{g}_\mathbb{R}(0,a,b,c)$ is isomorphic to $\mathfrak{g}_\mathbb{R}(0,0,0,1)$. We can then bring the associated plane via the $\mathbb{R}$-action to
\begin{equation}\label{n2}\left[\begin{pmatrix} 0 & 1 & 0 \\ 0 & 0 & 0 \\ 0 & 0 & 0 \end{pmatrix},\  -zx^2,\ -y^2z-z^3\right].\end{equation}
We can then assume that the plane $P$ associated to $\mathfrak{h}$ is such that $\pi_1(P)$ is the same as in $(\ref{n2})$. Then we are left with the action of the stabiliser of such matrix, that is included in $\mathrm{GL}(2,\mathbb{Q})\ltimes\mathbb{Q}^2$. Since for $\mathfrak{g}_\mathbb{R}(0,0,0,1)$ we have $p_3(y,z)=p_2(y,z)=0$, $a_{57}=0$ and $q_1(y,z)=q_2(y,z)=0$, and since the stabiliser acts separately on each of these terms, we must have the same also for $P$. Furthermore using the action we may put $\mathrm{coeff}(p_1(y,z),y)=\mathrm{coeff}(\pi_2(P),x^3)=\mathrm{coeff}(q_3(y,z),yz^2)=0$. Then 
\[[\pi_1(P),\pi_2(P),\pi_3(P)]=\left[\begin{pmatrix} 0 & 1 & 0 \\ 0 & 0 & 0 \\ 0 & 0 & 0 \end{pmatrix},\  -zx^2,\ -y^2z-\frac{1}{c}z^3\right].\]
This means that $\mathfrak{h}\cong\mathfrak{g}_\mathbb{Q}\left(0,0,0,\frac{1}{c}\right)$.

Finally if $\pi_1(P)=0$ then $\mathfrak{h}$ is a $\mathbb{Q}$-form of $\mathfrak{g}_\mathbb{R}(0,0,0,0)$ that is generated by the following two matrices
\[u=\begin{pmatrix} 0 & 0 & 0 \\ 1 & 0 & 0 \\ 0 & 0 & 0\end{pmatrix}, v=\begin{pmatrix}
0 &0 & 0 \\ 0 & 0 & 0 \\ 1 & 0 & 0 
\end{pmatrix}.\] 
This means that $\mathfrak{h}$ is isomorphic to $\mathfrak{g}_\mathbb{Q}(0,0,0,0)$. In conclusion we have proved that $\mathfrak{h}\cong\mathfrak{g}_\mathbb{Q}(0,a,b,c)$. As we have seen in Proposition \ref{alpha=0}, finding representatives for the isomorphism classes corresponds to finding normal forms for the adjoint action of $\mathrm{GL}_2(\mathbb{Q})$ over $\mathbb{P}(\mathfrak{sl}(2,\mathbb{Q}))$. Using the theory of rational canonical forms we can see that the normal forms are 
\[\begin{pmatrix} 1 & 0 \\ 0 & -1\end{pmatrix}, \begin{pmatrix} 0 & \varepsilon \\ 1 & 0 \end{pmatrix}\text{ with } \varepsilon\notin\mathbb{Q}^2\]
or the $0$ matrix. Hence as representatives of the $\mathbb{Q}$-isomorphism classes of the family $\mathfrak{g}_\mathbb{Q}(0,a,b,c)$ we have $\mathfrak{g}_\mathbb{Q}(0,0,0,0)$ and $\mathfrak{g}_\mathbb{Q}(0,0,\varepsilon,1)$, noticing that $\mathfrak{g}_\mathbb{Q}(0,0,\varepsilon,1)\cong\mathfrak{g}_\mathbb{Q}(0,1,0,0)$ if $\varepsilon\in\mathbb{Q}^2$. Finally it is clear that $\mathfrak{g}_\mathbb{Q}(0,0,\varepsilon',1)\cong \mathfrak{g}_\mathbb{Q}(0,0,\varepsilon,1)$ if and only if the ratio of $\varepsilon$ and $\varepsilon'$ is the square of a rational number.
\end{proof}

\subsubsection{Skjelbred and Sund method}

Now that the question of commensurability classes of lattices in the Lie groups associated with the Lie algebras $\mathfrak{g}_\mathbb{R}(\alpha,a,b,c)$ is settled we want to study the commensurability question for the Lie groups associated to the Lie algebras $L_j$ and $N_j$. We have seen that in order to study this question we have to understand the $\mathbb{Q}$-isomorphism classes of rational subalgebras of these Lie algebras. To do so we introduce a method that was developed by Skjelbred and Sund in \cite{Skje} and applied by Gong in his thesis \cite{Gong} in order to classify $7$-dimensional real nilpotent Lie algebras. To state the theorem we need to introduce some terminology. 

\begin{defi}
Let $\mathbb{K}$ be a field and $\mathfrak{g}$ a Lie algebra over $\mathbb{K}$. A map $B:\bigwedge^2\mathfrak{g}\rightarrow \mathbb{K}$ such that 
\[B([x_1,x_2],x_3)+B([x_2,x_3],x_1)+B([x_3,x_1],x_2)=0\]
is called a \emph{cocycle} and the space of cocycles is denoted by $\mathrm{Z}^2(\mathfrak{g},\mathbb{K})$. 
A map $B:\bigwedge^2\mathfrak{g}\rightarrow \mathbb{K}$ such that there exists $g\in\mathrm{Hom}(\mathfrak{g},\mathbb{K})$ and
\[B(x,y)=g([x,y])\]
is called a \emph{coboundary} and the space of coboundaries is denoted by $\mathrm{B}^2(\mathfrak{g},\mathbb{K})$. It can be noticed that the coboundaries form a subspace of the cocycles and finally the quotient space 
\[\mathrm{H}^2(\mathfrak{g},\mathbb{K})=\mathrm{Z}^2(\mathfrak{g},\mathbb{K})/\mathrm{B}^2(\mathfrak{g},\mathbb{K})\]
is called the \emph{2-nd cohomology group} of $\mathfrak{g}$ with coefficients in $\mathbb{K}$.
\end{defi}

\begin{defi}
For $B:\bigwedge^2\mathfrak{g}\rightarrow \mathbb{K}$ we can define the set
\[\mathfrak{g}^\perp_B=\{x\in\mathfrak{g}\ |\ B(x,\mathfrak{g})=0\}.\]
And if $B=(B_1,\ldots,B_k):\bigwedge^2\mathfrak{g}\rightarrow \mathbb{K}^k$ we can define $\mathfrak{g}^\perp_B=\mathfrak{g}^\perp_{B_1}\cap \ldots\cap \mathfrak{g}^\perp_{B_k}$.
\end{defi}

\begin{defi}
Let us denote by $\mathrm{G}_k(\mathrm{H}^2(\mathfrak{g},\mathbb{K}))$ the $k$-th Grassmannian of $\mathrm{H}^2(\mathfrak{g},\mathbb{K})$. Furthermore if $B$ is a cocycle let us denote by $\widetilde{B}$ its image in cohomology. Thus we define a subspace of the $k$-th Grassmannian of the cohomology of $\mathfrak{g}$ as follows
\[U_k(\mathfrak{g})=\{\widetilde{B}_1\mathbb{K}\oplus\ldots\oplus\widetilde{B}_k\mathbb{K}\in\mathrm{G}_k(\mathrm{H}^2(\mathfrak{g},\mathbb{K}))\ |\ \mathfrak{g}_{B=(B_1,\ldots,B_k)}^\perp\cap \mathcal{Z}(\mathfrak{g})=0 \}\]
where $\mathcal{Z}(\mathfrak{g})$ is the center of $\mathfrak{g}$. 
\end{defi}

If we call $\mathrm{Aut}(\mathfrak{g})$ the automorphism group of the Lie algebra $\mathfrak{g}$ we can notice that we have an action of $\mathrm{Aut}(\mathfrak{g})$ on $\mathrm{H}^2(\mathfrak{g},\mathbb{K})$ induce by the following action on cocycles: 
\[\text{if } \varphi\in \mathrm{Aut}(\mathfrak{g})\text{ and }B\in\mathrm{Z}^2(\mathfrak{g},\mathbb{K})\text{ then }\varphi\cdot B(x,y)=B(\varphi(x),\varphi(y)).\] Furthermore it can be proved that this action induces an action on $U_k(\mathfrak{g})$. We are now ready to state the theorem of Skjelbred and Sund.

\begin{theo}[{\cite[Theorem 3.5]{Skje}}]
Let $\mathfrak{g}$ be a Lie algebra over a field $\mathbb{K}$. The isomorphism classes of Lie algebras $\widetilde{\mathfrak{g}}$, with center $\widetilde{\mathfrak{Z}}$ of dimension $k$, with $\widetilde{\mathfrak{g}}/\widetilde{\mathfrak{Z}}\cong \mathfrak{g}$ and without abelian factors are in bijective correspondence with elements in $U_k(\mathfrak{g})/\mathrm{Aut}(\mathfrak{g})$.
\end{theo}

\begin{rema}
If $B\in U_k(\mathfrak{g})/\mathrm{Aut}(\mathfrak{g})$ is a representative of one orbit then the corresponding $k$-dimensional central extension of $\mathfrak{g}$ is defined as the direct sum of the vector spaces $\mathfrak{g}(B)=\mathfrak{g}\oplus\mathbb{K}^k$ and with Lie brackets $[(x,u),(y,v)]=([x,y],B(x,y))$.
\end{rema}

We will now apply this method to the aforementioned Lie algebras.  

\begin{prop}\label{Q form degenerate case}
All the Lie algebras $L_i$ and $N_j$ have just one $\mathbb{Q}$-form, up to $\mathbb{Q}$-isomorphism, except $L_6^\mathbb{R}(\varepsilon),N_3^\mathbb{R}(\varepsilon)$ and $N_{9}^\mathbb{R}(\varepsilon)$. For these Lie algebras their  $\mathbb{Q}$-forms are $L_6^\mathbb{Q}(\varepsilon),N_3^\mathbb{Q}(\varepsilon)$ and  $N_{9}^\mathbb{Q}(\varepsilon)$ respectively, as defined in Appendix \ref{LieBracket}.
\end{prop}
\begin{proof}
First consider the Lie algebras presented in Appendix \ref{LieBracket}  that are the sum of a $6$-dimensional Lie algebra and an abelian ideal. Since in \cite{Graaf} all $6$-dimensional Lie algebras over any field of characteristic different from $2$ are classified up to isomorphism the result follows from their analysis. Now consider the Lie algebras in Appendix \ref{LieBracket} that are the sum of a $7$-dimensional Lie algebra and an abelian ideal. Lie algebras of dimension $7$ over $\mathbb{R}$ are classified by Gong in \cite{Gong}. Hence we can follow his classification for the isomorphism classes that concern us. For almost all the Lie algebras in dimension $7$ the analysis that is done for $\mathbb{R}$ in \cite{Gong} works for $\mathbb{Q}$ without any problem so we will point out only the cases where there is a difference. Let again fix some terminology. If $\{e_i\}_i$ is a basis for $\mathfrak{g}$ we let $\Delta_{ij}=(e_i\wedge e_j)^*$  be the elemts of the basis for $(\bigwedge^2\mathfrak{g})^*$. We will write the elements of the cohomology group $\mathrm{H}^2(\mathfrak{g},\mathbb{K})$ simply as cocycles thinking them as equivalence classes. Finally for the action of $\mathrm{Aut}(\mathfrak{g})$ on $\mathrm{H}^2(\mathfrak{g},\mathbb{K})$ if $g\cdot(\sum \alpha_{ij}\Delta_{ij})=\alpha'_{ij}\Delta_{ij}$ we will write $\alpha_{ij}\mapsto \alpha'_{ij}$.\\

The Lie algebras $N_9^\mathbb{R}(\varepsilon)$ are $2$-dimensional central extensions of the Lie algebra $\mathfrak{g}$ defined on the basis $\{x_1,\ldots,x_5\}$ by $[x_1,x_2]=x_3$. An element of the cohomology group of $\mathfrak{g}$ reads as $a\Delta_{13}+b\Delta_{14}+c\Delta_{15}+d\Delta_{23}+e\Delta_{24}+f\Delta_{25}+g\Delta_{45}$ and we will just write $B=[a,b,c,d,e,f,g]$. The action of the automorphism group of $\mathfrak{g}$ on it is as follows: 
\begin{align*}
& a\mapsto aa_{11}\delta+da_{21}\delta;\\
& b\mapsto a_{11}(aa_{34}+ba_{44}+ca_{54})+a_{21}(da_{34}+ea_{44}+fa_{54})+g(a_{41}a_{54}-a_{51}a_{44});\\
& c\mapsto a_{11}(aa_{35}+ba_{45}+ca_{55})+a_{21}(da_{35}+ea_{45}+fa_{55})+g(a_{41}a_{55}-a_{51}a_{45});\\
& d\mapsto aa_{12}\delta+da_{22}\delta;\\
& e\mapsto a_{12}(aa_{34}+ba_{44}+ca_{54})+a_{22}(da_{34}+ea_{44}+fa_{54})+g(a_{42}a_{54}-a_{52}a_{44});\\
& f\mapsto a_{12}(aa_{35}+ba_{45}+ca_{55})+a_{22}(da_{35}+ea_{45}+fa_{55})+g(a_{42}a_{55}-a_{52}a_{45});\\
& g\mapsto g(a_{44}a_{55}-a_{54}a_{45})
\end{align*}
where $\delta=a_{11}a_{22}-a_{12}a_{21}=a_{33}$.
Looking at the definition of the family of Lie algebras $N_9^\mathbb{R}(\varepsilon)$ in Appendix \ref{LieBracket} we can see that it corresponds to the $2$-dimensional subspace of the cohomology of $\mathfrak{g}$ generated by $(\Delta_{13}+ \Delta_{14}+\varepsilon\Delta_{25})\wedge (\Delta_{15}+\Delta_{23})$ with $\varepsilon\in\{0,1,-1\}$. Let $V_1$ be the subspace of $\mathrm{H}^2(\mathfrak{g},\mathbb{R})$ generated by $\{\Delta_{13},\Delta_{14},\Delta_{15},\Delta_{23},\Delta_{24},\Delta_{25}\}$, $V_2$ the subspace generated by $\{\Delta_{14},\Delta_{15},\Delta_{24},\Delta_{25},\Delta_{45}\}$ and $V_3=V_1\cap V_2$. Notice that $V_1,V_2$ and $V_3$ are submodules for the group action. If $L_{\varepsilon}$ is a $2$-dimensional subspace of $\mathrm{H}^2(\mathfrak{g},\mathbb{R})$ associated to $N_9^\mathbb{R}(\varepsilon)$ then in Gong's analysis $L$ is characterised by $L_\varepsilon\subseteq V_1$ and $L_\varepsilon\cap V_3=0$. Let us now consider a $\mathbb{Q}$-form $\mathfrak{h}$ of $N_9^\mathbb{R}(\varepsilon)$. Then $\mathfrak{h}$ is a $2$-dimensional central extension of $\mathfrak{g}$ and let $P$ be the $2$-dimensional subspace of $\mathrm{H}^2(\mathfrak{g},\mathbb{Q})$ associated to it. We can then assume $P\subseteq V_1$ and $P\cap V_3=0$ so that the elements of a basis for $P$ are $A=[1,b,c,d,e,f,0]$ and $B=[0,b_1,c_1,1,e_1,f_1,0]$. Following Gong's analysis we can bring the second element of the basis to either $B_1=[0,0,1,1,0,0,0]$ or $B_2=[0,0,0,1,0,0,0]$. For the case $B_1=[0,0,1,1,0,0,0]$ we choose $a_{21}=a_{34}=a_{54}=0$, $a_{35}=-\frac{a_{12}a_{55}}{a_{22}}$, $a_{11}a_{55}=a_{11}a_{22}^2=1$ in order to fix $B_1$. Considering then the action on $A$ we have 

\begin{align*}a=1 & \mapsto a_{11}\delta;
b\mapsto a_{11}ba_{44};  
c\mapsto a_{11}(a_{35}+ba_{45}+ca_{55});
d\mapsto a_{12}\delta; \\
& e\mapsto a_{12}ba_{44}+a_{22}ea_{44};
f\mapsto a_{12}(a_{35}+ba_{45}+ca_{55})+a_{22}(ea_{45}+fa_{55});
g\mapsto 0.
\end{align*}

Notice that $b$ and $e$ cannot be both $0$ otherwise there will be a non trivial element in the intersection $\mathfrak{g}^\perp_{(A,B)}\cap \mathcal{Z}(\mathfrak{g})$.
Now assuming $b\not=0$ we can make it $1$ taking $a_{44}=\frac{1}{ba_{11}}$ and solving for $a_{12}$ make $e$ equal $0$. Solving for $a_{45}$ we can make $c=d$ and then subtracting a multiple of $B$ make them equal $0$.  Then taking $a_{12}=a_{35}=a_{45}=0$ we obtain 

\begin{align*}a=1\mapsto a_{11}\delta;
b=1\mapsto a_{11}a_{44};
c=0\mapsto 0;
d=0\mapsto 0;
e=0\mapsto 0;
f\mapsto a_{22}fa_{55};
g\mapsto 0.
\end{align*}

The representatives of the orbits are then $A_\varepsilon=[1,1,0,0,0,\varepsilon,0]$ with $\varepsilon\in\mathbb{Q}$ and $A_{\varepsilon'}$ is in the same orbit of $A_{\varepsilon}$ if and only if there exists $\alpha\in\mathbb{Q}^*$ such that $\varepsilon'=\alpha^6\varepsilon$. For the case $B_2=[0,0,0,1,0,0,0]$ we can then assume $A=[1,b,c,0,e,f,0]$. In order to fix $B_2$ we need $a_{21}=a_{34}=a_{35}=0$ and $a_{11}a_{22}^2=1$. Then the action on $A$ is as follows 

\begin{align*}a & \mapsto a_{11}\delta;
b\mapsto a_{11}(ba_{44}+ca_{54});  
c\mapsto a_{11}(ba_{45}+ca_{55});
d\mapsto a_{12}\delta; \\
& e\mapsto a_{12}(ba_{44}+ca_{54})+a_{22}(ea_{44}+fa_{54});
f\mapsto a_{12}(ba_{45}+ca_{55})+a_{22}(ea_{45}+fa_{55});
g\mapsto 0.
\end{align*}

Again one of $b$ or $e$ should be non zero. Make then $b=1$ and $e=0$. Make also $c=0$ by solving for $a_{45}$. Then $f\ne 0$, since otherwise the $2$-cocycle will contain a non trivial element of the center in its kernel, and we get $A=[1,1,0,0,0,1,0]$. Then one can see that $[1,1,0,0,0,1,0]\wedge B_2$ is in the same orbit as $A_0\wedge B_1$.\\

We are now left to consider the Lie algebras in Appendix \ref{LieBracket} that are not decomposable, namely $L_6^\mathbb{R}(1),N_{10},N_{11},N_{12}$ and $N_{13}$. The Lie algebra $L_6^\mathbb{R}(1)$ is a $2$-dimensional central extension of the Lie algebra defined on the basis $\{x_1,\ldots,x_6\}$ by $[x_1,x_2]=x_4,[x_1,x_3]=x_5$. The elements in its cohomology group can be represented by $a\Delta_{14}+b\Delta_{15}+c\Delta_{16}+d\Delta_{23}+e\Delta_{24}+f(\Delta_{25}+\Delta_{34})+g\Delta_{26}+h\Delta_{35}+i\Delta_{36}$. The action of the automorphism group is as follows:

\begin{align*}
a\mapsto &  aa_{11}^2a_{22}+ba_{11}^2a_{32}+ea_{11}a_{22}a_{21}+fa_{11}(a_{21}a_{32}+a_{31}a_{22})+ha_{11}a_{31}a_{32};\\
b\mapsto &  aa_{11}^2a_{23}+ba_{11}^2a_{33}+ea_{11}a_{21}a_{23}+fa_{11}(a_{21}a_{33}+a_{31}a_{23})+ha_{11}a_{31}a_{33};\\
c\mapsto &  aa_{11}a_{46}+ba_{11}a_{56}+ca_{11}a_{66}+ea_{21}a_{46}+f(a_{21}a_{56}+a_{31}a_{46})+ga_{21}a_{66}+ha_{31}a_{56}+ia_{31}a_{66};\\
d\mapsto &  d(a_{22}a_{33}-a_{23}a_{32})+e(a_{22}a_{43}-a_{23}a_{42})+f(a_{22}a_{53}-a_{23}a_{52}+a_{32}a_{43}-a_{33}a_{42})\\
&+g(a_{22}a_{63}-a_{23}a_{62})+h(a_{32}a_{53}-a_{33}a_{52})+i(a_{32}a_{63}-a_{33}a_{62});\\
e\mapsto & ea_{11}a_{22}^2+2fa_{11}a_{22}a_{32}+ha_{11}a_{32}^2;\\
f\mapsto &  ea_{11}a_{22}a_{23}+f(a_{11}a_{22}a_{33}+a_{11}a_{32}a_{23})+ha_{11}a_{32}a_{33};\\
g\mapsto &  ea_{22}a_{46}+f(a_{22}a_{56}+a_{32}a_{46})+ga_{22}a_{66}+ha_{32}a_{56}+ia_{32}a_{66};\\
h\mapsto &  ea_{11}a_{23}^2+2fa_{11}a_{23}a_{33}+ha_{11}a_{33}^2;\\
i\mapsto &  ea_{23}a_{46}+f(a_{23}a_{56}+a_{33}a_{46})+ga_{23}a_{66}+ha_{33}a_{56}+ia_{33}a_{66}.
\end{align*}

The family $L_{6}^\mathbb{R}(1)$ corresponds to $(\Delta_{14}-\Delta_{36})\wedge (\Delta_{15}+\Delta_{26})$. Call $V_1$ the subspace of $\mathrm{H}^2(\mathfrak{g},\mathbb{R})$ generated by $\{\Delta_{14},\Delta_{15},\Delta_{16},\Delta_{23},\Delta_{26},\Delta_{36}\}$ and $V_2$ the one generated by $\{\Delta_{14},\Delta_{15},\Delta_{16},\Delta_{23}\}$. Notice that $V_1$ and $V_2$ are submodules under the action of $\mathrm{Aut}(\mathfrak{g})$. If $L$ is a $2$-dimensional subspace of $\mathrm{H}^2(\mathfrak{g},\mathbb{R})$ associated to $L_{6}^\mathbb{R}(1)$ then $L$ is characterised by $L\subseteq V_1$ and $L\not\subseteq V_2$. Hence the two generators of $L$ have the form $A=[a_1,b_1,c_1,d_1,0,0,g_1,0,i_1]$ with $i_1\ne 0$ and $B=[a,b,c,d,0,0,g,0,i]$. Furthermore we must have either $a$ or $a_1$ not $0$ since otherwise we would have a non trivial element in $\mathfrak{g}^\perp_{(A,B)}\cap \mathcal{Z}(\mathfrak{g})$. Then we can bring the first generator of a $2$-dimensional subspace of $\mathrm{H}^2(\mathfrak{g},\mathbb{Q})$ related to a $\mathbb{Q}$-form of $L_{6}^\mathbb{R}(1)$ to $A=[1,0,0,0,0,0,0,0,-1]$. In order to leave $A$ stable we need $a_{23}=a_{32}=a_{62}=0, a_{46}=\frac{a_{31}a_{66}}{a_{11}}, a_{11}^2a_{22}=a_{33}a_{66}=1$. Then the action on $B=[0,b,c,d,0,0,g,0,i]$ is as follows

\begin{align*}
a\mapsto  0; 
b\mapsto & ba_{11}^2a_{33}; 
c\mapsto   ba_{11}a_{56}+ca_{11}a_{66}+ga_{21}a_{66}+ia_{31}a_{66}; 
d\mapsto  da_{22}a_{33}+ga_{22}a_{63}; \\
e\mapsto & 0; 
f\mapsto  0;
g\mapsto ga_{22}a_{66}; 
h\mapsto  0; 
i\mapsto  i.
\end{align*}

Notice that $i$ should be $0$ since otherwise the corresponding Lie algebra over $\mathbb{R}$ is not in the same isomorphism class of $L_{6}^\mathbb{R}(1)$. Notice furthermore that $b$ and $g$ should be non zero since otherwise we would have a non trivial element in $\mathfrak{g}^\perp_{(A,B)}\cap \mathcal{Z}(\mathfrak{g})$. Then solving for $a_{63},a_{56}$ put $d=c=0$. Then taking $a_{63}=a_{56}=a_{21}=0$ we are left with 

\begin{align*}a\mapsto 0;
b\mapsto ba_{11}^2a_{33};
c\mapsto 0;
d\mapsto 0;
e\mapsto 0;
f\mapsto 0;
g\mapsto ga_{22}a_{66};
h\mapsto  0;
i\mapsto 0.
\end{align*}

Then a representative of the orbit is $B_\varepsilon=[0,1,0,0,0,0,\varepsilon,0,0]$ and $B_{\varepsilon'}$ is in the same orbit as $B_{\varepsilon}$ if and only if there exists $\alpha\in\mathbb{Q}^*$ such that $\varepsilon'=\alpha^2\varepsilon$. Since over $\mathbb{R}$ the Lie algebra associated to $A\wedge B_\varepsilon$ should be isomorphic to $L_6^\mathbb{R}(1)$ then $\varepsilon>0$.

The Lie algebras $N_{10}$ and $N_{11}$ are $2$-dimensional central extensions of the Lie algebra defined on the basis $\{x_1,\ldots,x_6\}$ by $[x_1,x_2]=x_3$.  An element of the cohomology group of $\mathfrak{g}$ reads as $a\Delta_{13}+b\Delta_{14}+c\Delta_{15}+d\Delta_{16}+e\Delta_{23}+f\Delta_{24}+g\Delta_{25}+h\Delta_{26}+i\Delta_{45}+l\Delta_{46}+m\Delta_{56}$. The action of the automorphism group of $\mathfrak{g}$ on it is as follows: 
\begin{align*}
 a\mapsto  aa_{11}\delta+  e & a_{21}\delta;\\
 b\mapsto  a_{11}(aa_{34}+ & ba_{44}+ca_{54}+da_{64})+a_{21}(ea_{34}+fa_{44}+ga_{54}+ha_{64})\\ & +i(a_{41}a_{54}-a_{51}a_{44})+l(a_{41}a_{64}-a_{44}a_{61})+m(a_{51}a_{64}-a_{54}a_{61});\\
c\mapsto   a_{11}(aa_{35}+ & ba_{45}+ca_{55}+da_{65})+a_{21}(ea_{35}+fa_{45}+ga_{55}+ha_{65}) \\ & +i(a_{41}a_{55}-a_{51}a_{45})+l(a_{41}a_{65}-a_{45}a_{61})+m(a_{51}a_{65}-a_{55}a_{61});\\
 d\mapsto  a_{11}(aa_{36}+ & ba_{46}+ca_{56}+da_{66})+a_{21}(ea_{36}+fa_{46}+ga_{56}+ha_{66})
 \\ & +i(a_{41}a_{56}-a_{51}a_{46})+l(a_{41}a_{66}-a_{46}a_{61})+m(a_{51}a_{66}-a_{56}a_{61});\\
 e\mapsto  aa_{12}\delta+  e & a_{22}\delta; \\
 f\mapsto   a_{12}(aa_{34}+ & ba_{44}+ca_{54}+da_{64})+a_{22}(ea_{34}+fa_{44}+ga_{54}+ha_{64}) \\ & +i(a_{42}a_{54}-a_{52}a_{44})+l(a_{42}a_{64}-a_{44}a_{62})+m(a_{52}a_{64}-a_{54}a_{62});\\
 g\mapsto   a_{12}(aa_{35}+ & ba_{45}+ca_{55}+da_{65})+a_{22}(ea_{35}+fa_{45}+ga_{55}+ha_{65}) \\ & +i(a_{42}a_{55}-a_{52}a_{45})+l(a_{42}a_{65}-a_{45}a_{62})+m(a_{52}a_{65}-a_{55}a_{62});\\
 h\mapsto   a_{12}(aa_{36}+ & ba_{46}+ca_{56}+da_{66})+a_{22}(ea_{36}+fa_{46}+ga_{56}+ha_{66}) \\ & +i(a_{42}a_{56}-a_{52}a_{46})+l(a_{42}a_{66}-a_{46}a_{62})+m(a_{52}a_{66}-a_{56}a_{62});\\
 i\mapsto   i(a_{44}a_{55}- & a_{54}a_{45})+l(a_{44}a_{65}-a_{45}a_{64})+m(a_{54}a_{65}-a_{55}a_{64});\\
 l \mapsto  i(a_{44}a_{56}- & a_{54}a_{46})+l(a_{44}a_{66}-a_{46}a_{64})+m(a_{54}a_{66}-a_{56}a_{64});\\
 m \mapsto  i(a_{45}a_{56}- & a_{55}a_{46})+l(a_{45}a_{66}-a_{46}a_{65})+m(a_{55}a_{66}-a_{56}a_{65})
\end{align*}
where $\delta=a_{11}a_{22}-a_{12}a_{21}$. The Lie algebras  $N_{10}$ and $N_{11}$ correspond to $(\Delta_{13}+\Delta_{25})\wedge (\Delta_{14}+\Delta_{26})$ and $(\Delta_{13}+\Delta_{25})\wedge (\Delta_{14}+\Delta_{23}+\Delta_{26})$ respectively. Call $V_1$ the subspace of $\mathrm{H}^2(\mathfrak{g},\mathbb{R})$ generated by $\{\Delta_{13},\Delta_{14},\Delta_{15},\Delta_{16},\Delta_{23},\Delta_{24},\Delta_{25},\Delta_{26}\}$ and $V_2$ the one generated by $\{\Delta_{14},\Delta_{15},\Delta_{16},\Delta_{24},\Delta_{25},\Delta_{26}\}$. Notice that $V_1$ and $V_2$ are submodules under the action of $\mathrm{Aut}(\mathfrak{g})$. If $L_1$ and $L_2$ are $2$-dimensional subspaces of $\mathrm{H}^2(\mathfrak{g},\mathbb{R})$ associated to $N_{10}$ and $N_{11}$ respectively then they are both characterised by $L_i\subseteq V_1$ and $L_i\not\subseteq V_2$. Hence the two generators of the two dimensional subspaces of $\mathrm{H}^2(\mathfrak{g},\mathbb{R})$ associated to a $\mathbb{Q}$-form of both $N_{10}$ and $N_{11}$ should have the form $A=[a_1,b_1,c_1,d_1,e_1,f_1,g_1,h_1,0,0,0]$ and $B=[a,b,c,d,e,f,g,h,0,0,0]$ with either $a$ or $a_1$ not $0$. Then under the group action we can bring $A$ to $A=[1,0,0,0,0,0,1,0,0,0,0]$. In order to stabilise it we need $a_{12}=a_{54}=a_{56}=a_{34}=a_{36}=0,a_{35}=-\frac{a_{21}a_{55}}{a_{11}},a_{11}^2a_{22}=1,a_{22}a_{55}=1$. Then the action on $B=[0,b,c,d,e,f,g,h,0,0,0]$ is as follows

\begin{align*}
  &a\mapsto   e  a_{21}\delta;
 b\mapsto  a_{11}(  ba_{44}+da_{64})+a_{21}(fa_{44}+ha_{64});\\
&c\mapsto   a_{11}(   ba_{45}+ca_{55}+da_{65})+a_{21}(ea_{35}+fa_{45}+ga_{55}+ha_{65}); \\
 &d\mapsto  a_{11}(  ba_{46}+da_{66})+a_{21}(fa_{46}+ha_{66});
 e\mapsto   e  a_{22}\delta; 
 f\mapsto    a_{22}(fa_{44}+ha_{64}); \\ 
& g\mapsto     a_{22}(ea_{35}+fa_{45}+ga_{55}+ha_{65}); 
 h\mapsto     a_{22}(fa_{46}+ha_{66}); 
 i\mapsto   0;
 l \mapsto  0;
 m \mapsto   0. 
\end{align*}

Since we cannot have all $b,c,d,e,f,g,h$ equal  $0$ let us suppose that $b\ne 0$ then solving for $a_{45}$ and $a_{46}$ we can put $c=d=0$. Putting now $a_{45}=a_{46}=a_{21}=0$ we are left with 
\begin{align*}
 a\mapsto    & 0; b\mapsto  a_{11} ba_{44};
c\mapsto   0; 
 d\mapsto  0;
 e\mapsto  e   a_{22}\delta; 
 f\mapsto   a_{22}(fa_{44}+ha_{64}); \\
 g\mapsto  & a_{22}(ga_{55}+ha_{65}); 
 h\mapsto  a_{22}ha_{66}; 
 i\mapsto   0;
 l \mapsto  0;
 m \mapsto  0. 
\end{align*}
 Now if $h$ is $0$ the Lie algebra is not isomorphic to a $\mathbb{Q}$-form of neither $N_{10}$ nor $N_{11}$. Then solving for $a_{64}$ and $a_{65}$ we can put $f=g=0$. Taking now $a_{64}=a_{65}=0$ we have 
 \begin{align*}
  a\mapsto     0; b\mapsto  a_{11} ba_{44};
 c\mapsto  0; 
  d\mapsto  0;
  e\mapsto  e   a_{22}\delta; 
  f\mapsto   0; 
  g\mapsto  0; 
  h\mapsto  a_{22}ha_{66}; 
  i\mapsto   0;
  l \mapsto  0;
  m \mapsto  0. 
 \end{align*}
Then depending on whether $e$ is $0$ or not we can arrive to the standard form $B_1=[0,1,0,0,0,0,0,1,0,0,0]$ that corresponds to $L_{10}$ or $B_2=[0,1,0,0,1,0,0,1,0,0,0]$ that corresponds to $L_{11}$.

The Lie algebras $N_{12}$ and $N_{13}$ are $3$-dimensional extensions of the abelian Lie algebra of dimension $5$, $\mathbb{R}^5$. Elements of the cohomology group of $\mathbb{R}^5$ can be represented as $B=a\Delta_{12}+b\Delta_{13}+c\Delta_{14}+d\Delta_{15}+e\Delta_{23}+f\Delta_{24}+g\Delta_{25}+h\Delta_{34}+i\Delta_{35}+l\Delta_{45}$. If we denote by $g=(a_{ij})$ an element of the automorphism group of $\mathbb{R}^5$, that in this case is just $\mathrm{GL}(5,\mathbb{R})$, then its action on the cohomology group is as follows:\\
\begin{align*}
a\mapsto a\Sigma_{12}^{12}+b\Sigma_{13}^{12}+c\Sigma_{14}^{12}+d\Sigma_{15}^{12}+e\Sigma_{23}^{12}+f\Sigma_{24}^{12}+g\Sigma_{25}^{12}+h\Sigma_{34}^{12}+i\Sigma_{35}^{12}+l\Sigma_{45}^{12};\\
b\mapsto a\Sigma_{12}^{13}+b\Sigma_{13}^{13}+c\Sigma_{14}^{13}+d\Sigma_{15}^{13}+e\Sigma_{23}^{13}+f\Sigma_{24}^{13}+g\Sigma_{25}^{13}+h\Sigma_{34}^{13}+i\Sigma_{35}^{13}+l\Sigma_{45}^{13};\\
c\mapsto a\Sigma_{12}^{14}+b\Sigma_{13}^{14}+c\Sigma_{14}^{14}+d\Sigma_{15}^{14}+e\Sigma_{23}^{14}+f\Sigma_{24}^{14}+g\Sigma_{25}^{14}+h\Sigma_{34}^{14}+i\Sigma_{35}^{14}+l\Sigma_{45}^{14};\\
d\mapsto a\Sigma_{12}^{15}+b\Sigma_{13}^{15}+c\Sigma_{14}^{15}+d\Sigma_{15}^{15}+e\Sigma_{23}^{15}+f\Sigma_{24}^{15}+g\Sigma_{25}^{15}+h\Sigma_{34}^{15}+i\Sigma_{35}^{15}+l\Sigma_{45}^{15};\\
e\mapsto a\Sigma_{12}^{23}+b\Sigma_{13}^{23}+c\Sigma_{14}^{23}+d\Sigma_{15}^{23}+e\Sigma_{23}^{23}+f\Sigma_{24}^{23}+g\Sigma_{25}^{23}+h\Sigma_{34}^{23}+i\Sigma_{35}^{23}+l\Sigma_{45}^{23};\\
f\mapsto a\Sigma_{12}^{24}+b\Sigma_{13}^{24}+c\Sigma_{14}^{24}+d\Sigma_{15}^{24}+e\Sigma_{23}^{24}+f\Sigma_{24}^{24}+g\Sigma_{25}^{24}+h\Sigma_{34}^{24}+i\Sigma_{35}^{24}+l\Sigma_{45}^{24};\\
g\mapsto a\Sigma_{12}^{25}+b\Sigma_{13}^{25}+c\Sigma_{14}^{25}+d\Sigma_{15}^{25}+e\Sigma_{23}^{25}+f\Sigma_{24}^{25}+g\Sigma_{25}^{25}+h\Sigma_{34}^{25}+i\Sigma_{35}^{25}+l\Sigma_{45}^{25};\\
h\mapsto a\Sigma_{12}^{34}+b\Sigma_{13}^{34}+c\Sigma_{14}^{34}+d\Sigma_{15}^{34}+e\Sigma_{23}^{34}+f\Sigma_{24}^{34}+g\Sigma_{25}^{34}+h\Sigma_{34}^{34}+i\Sigma_{35}^{34}+l\Sigma_{45}^{34};\\
i\mapsto a\Sigma_{12}^{35}+b\Sigma_{13}^{35}+c\Sigma_{14}^{35}+d\Sigma_{15}^{35}+e\Sigma_{23}^{35}+f\Sigma_{24}^{35}+g\Sigma_{25}^{35}+h\Sigma_{34}^{35}+i\Sigma_{35}^{35}+l\Sigma_{45}^{35};\\
l\mapsto a\Sigma_{12}^{45}+b\Sigma_{13}^{45}+c\Sigma_{14}^{45}+d\Sigma_{15}^{45}+e\Sigma_{23}^{45}+f\Sigma_{24}^{45}+g\Sigma_{25}^{45}+h\Sigma_{34}^{45}+i\Sigma_{35}^{45}+l\Sigma_{45}^{45}
\end{align*}
where $\Sigma_{ij}^{st}=a_{is}a_{jt}-a_{it}a_{js}$.
The Lie algebras $N_{12}$ and $N_{13}$ correspond to the three dimensional subspace of the cohomology group of $\mathbb{R}^5$ represented by $\Delta_{12}\wedge\Delta_{13}\wedge(\Delta_{14}+\Delta_{25})$ and $\Delta_{12}\wedge(\Delta_{13}+\Delta_{24})\wedge(\Delta_{14}+\Delta_{25})$ respectively. We can notice that for both subspaces two of the generators are such that one is a symplectic form when restricted to the subspace $\spn\{e_1,e_2,e_4,e_5\}$ with $\{e_i\}$ the standard basis for $\mathbb{R}^5$ and the other is a Lagrangian subspace with respect to the first. Then if $P$ is the three dimensional subspace of $\mathrm{H}^2(\mathbb{Q}^4,\mathbb{Q})$ corresponding to a $
\mathbb{Q}$-form  $\mathfrak{h}$ of $N_{12}$ or $N_{13}$ we can bring two of the generators to $A=[1,0,0,0,0,0,0,0,0,0]$ and $C=[0,0,1,0,0,0,1,0,0,0]$. Now in order to fix $A$ and $C$ we need $a_{12}=a_{13}=a_{14}=a_{15}=a_{23}=a_{24}=a_{25}=a_{53}=a_{54}=a_{43}=0$,  $a_{45}=\frac{-a_{21}a_{55}}{a_{11}}$ and $a_{11}a_{22}=a_{11}a_{44}=a_{22}a_{55}=1$. Then if $B=[0,b,0,d,e,f,g,h,i,l]$ we get 

\begin{align*}
&a\mapsto 0;
b\mapsto ba_{11}a_{33}+ea_{21}a_{33}-ha_{33}a_{41}-ia_{33}a_{51};\\
& c\mapsto ba_{11}a_{34}+ca_{11}a_{44}-da_{21}a_{45}a_{55}+ea_{21}a_{34}+fa_{21}a_{44}-g\frac{a_{21}^2a_{55}}{a_{11}}+h\Sigma_{34}^{14}-ia_{34}a_{51}-la_{44}a_{51};\\
&d\mapsto ba_{11}a_{35}+da_{11}a_{55}+ea_{21}a_{35}-f\frac{a_{21}^2a_{55}}{a_{11}}+ga_{21}a_{55}+h\Sigma_{34}^{15}+i\Sigma_{35}^{15}+l\Sigma_{45}^{15};\\
& e\mapsto ea_{22}a_{33}-ha_{33}a_{42}-ia_{33}a_{52};
f\mapsto ea_{22}a_{34}+fa_{22}a_{44}+h\Sigma_{34}^{24}-ia_{34}a_{52}-la_{44}a_{52};\\
& g\mapsto ea_{22}a_{35}-f\frac{a_{21}a_{22}a_{55}}{a_{11}}+ga_{22}a_{55}+h\Sigma_{34}^{25}+i\Sigma_{35}^{25}+l\Sigma_{45}^{25};\\
& h\mapsto ha_{33}a_{44}; 
i\mapsto -h\frac{a_{21}a_{33}a_{55}}{a_{11}}+ia_{33}a_{55};
l\mapsto h\Sigma_{34}^{45}+ia_{34}a_{55}+la_{44}a_{55}.
\end{align*}

Now since the subspace spanned by $\{\Delta_{12},\Delta_{13},\Delta_{14},\Delta_{15},\Delta_{23},\Delta_{24},\Delta_{25}\}$ is a submodule for the action of the stabiliser of $A$ and $C$ we have $h=i=l=0$. Then $e$ should be $0$ as well otherwise the Lie algebra is isomorphic to a $\mathbb{Q}$-form of a Lie algebra that does not interest us. Then $b$ should be non zero otherwise the cocycle will contain a non trivial element of the center in its kernel. Now solving for $a_{34}$ make $c=g$ and then by subtracting a scalar multiple of $C$ make them equal $0$. Furthermore solving for $a_{35}$ make $d=0$. Taking now $a_{34}=a_{35}=a_{21}=0$ we are left with 

\begin{align*}
a\mapsto 0;
b\mapsto ba_{11}a_{33};
 c\mapsto 0;
d\mapsto 0;
e\mapsto 0;
f\mapsto fa_{22}a_{44};
g\mapsto 0;
 h\mapsto 0; 
i\mapsto 0;
l\mapsto 0.
\end{align*}

Hence depending on whether $f$ is $0$ or not a representative of the $\mathbb{Q}$-orbit is either  $B_1=[0,1,0,0,0,0,0,0,0,0]$ or $B_2=[0,1,0,0,0,1,0,0,0,0]$ that correspond to $N_{12}$ and $N_{13}$ defined over $\mathbb{Q}$.

\end{proof}

Putting together Proposition \ref{Q-form general case}, \ref{Q form alpha=0} and \ref{Q form degenerate case} we see that the Proof of Theorem \ref{ClassLattic} is achieved.

\begin{rema}
A priori the method just presented could have been applied also to classify the family of Lie algebras that we found in the non degenerate case, namely $\mathfrak{g}(\alpha,a,b,c)$. Nevertheless, the method was not that easy to apply for that case hence we decided to use an ad hoc method. Indeed also in \cite{Gong} for the case of central extensions of the free Lie algebra of rank $2$ over $3$ generators the author uses another method. 
\end{rema}

\section{Topological considerations}\label{topology}

Let $\Gamma$ be a subgroup of $\mathcal{H}(n,1)$ acting properly discontinuously and cocompactly on $\mathfrak{a}(\mathbb{C}^{n+1})$. Then from \cite[ Theorem 1.3]{Margulis} $\Gamma$ is virtually polycyclic. From a theorem of Selberg every finitely generated linear group contains a torsion free subgroup of finite index. Hence, up to replacing $\Gamma$ by a finite index subgroup, we can consider $M=\Gamma\backslash\mathfrak{a}(\mathbb{C}^{n+1})$ to be a compact flat Hermite-Lorentz manifold. From Theorem \ref{FG} there exists a subgroup $H\le \mathcal{H}(n,1)$ that acts simply transitively on $\mathfrak{a}(\mathbb{C}^{n+1})$ and $\Gamma\cap H$ has finite index in $\Gamma$ and it is a lattice in $H$. Hence, up to finite cover, $M$ is diffeomorphic to $(\Gamma\cap H)\backslash H$. \\
Let us suppose that $H=U(\gamma_2,\gamma_3,b_2,b_3)$ is a unipotent group so that it has the form of Proposition \ref{UnipSimplTrans}.

\begin{prop} 
Let $\Gamma$ be a lattice in the group $U:=U(\gamma_2,\gamma_3,b_2,b_3)$. The manifold $\Gamma\backslash U$ is a fiber bundle over a real torus $\Gamma/(\Gamma\cap \mathcal{C}^2U)\backslash U/\mathcal{C}^2U$ of dimension $\frac{2n+1}{3}\le p\le 2n+2$ with fibers that are real tori of dimension $q=2n+2-p$. Furthermore this fibration split into two fiber bundles as follows
\[\Gamma\backslash U\rightarrow \Gamma/(\Gamma\cap \mathcal{C}^3U)\backslash U/\mathcal{C}^3U\rightarrow \Gamma/(\Gamma\cap \mathcal{C}^2U)\backslash U/\mathcal{C}^2U.\]
\end{prop}
\begin{proof}
Let $\mathcal{C}^iU$, with $i\le 3$, be the elements of the lower central series, we have the following commutative diagram

\begin{equation*}
\begin{tikzcd}
0 \arrow[r]  & \mathcal{C}^3  U \arrow[r] \arrow[d, hook]  & U \arrow[r, "\pi_1"] \arrow[d, "\mathrm{Id}"] & U  /\mathcal{C}^3U \arrow[r] \arrow[d, "\pi_2"] & 0 \\
0 \arrow[r] &\mathcal{C}^2  U \arrow[r] & U  \arrow[r, "\pi"] & U  /\mathcal{C}^2U \arrow[r] & 0
\end{tikzcd}
\end{equation*}

From \cite[Corollary 1 of Theorem 2.3]{Rag} if $\Gamma$ is a lattice in $U$ nilpotent simply connected Lie group then $\Gamma\cap \mathcal{C}^iU$ is a lattice in $\mathcal{C}^iU$. Hence considering the induces maps by $\pi_1$ and $\pi_2$ on the quotients by $\Gamma$ we have that $\Gamma\backslash U$ can be seen as a sequence of fiber bundles. Furthermore the derived group $\mathcal{C}^2U$ is abelian of dimension $q$ with $0\le q\le 2+2\mathrm{rank}(\gamma_2) +1$. From Lemma \ref{uppper bound k} we then have $q\le \frac{4n+5}{3}$. Hence the quotient $U/\mathcal{C}^2U$ is an abelian Lie group of dimension $p$ with $\frac{2n+1}{3}\le p\le 2n+2$. Then $\mathcal{C}^2U$ is isomorphic to $\mathbb{R}^q$ and $U/\mathcal{C}^2U$ is isomorphic to $\mathbb{R}^p$.
Since, as we have seen, $\Gamma$ intersects $\mathcal{C}^2U$ in a lattice  we have that $\Gamma\cap \mathcal{C}^2U$ is isomorphic to $\mathbb{Z}^q$ and $\pi(\Gamma)$ is isomorphic to $\mathbb{Z}^p$. Then finally considering the fiber bundle induce by $\pi$ the manifold $\Gamma/(\Gamma\cap \mathcal{C}^2U)\backslash U/\mathcal{C}^2U$ is a torus and the fibers $\Gamma\cap \mathcal{C}^2U\backslash\mathcal{C}^2U$ are real tori.
\end{proof}

\begin{rema}
Notice that the above proposition is just a translation of the fact that the group $U(\gamma_2,\gamma_3,b_2,b_3)$ is $3$-step nilpotent.
\end{rema}

\appendix \section{}\label{LieBracket}

For $\mathbb{K}=\mathbb{R}$ the following is a non redundant list, up to isomorphism, of the $8$-dimensional nilpotent Lie algebras that appear as Lie algebras of unipotent simply transitive subgroups of $\mathrm{U}(3,1)\ltimes \mathbb{C}^{3+1}$. They are found putting together Proposition \ref{w_0=0}, \ref{gamma_2=0}, \ref{alpha=0} and \ref{w_0not0,gamma2not0}. Furthermore taking $\mathbb{K}=\mathbb{Q}$ this is also the complete non redundant list of the $\mathbb{Q}$-isomorphism classes of $\mathbb{Q}$-forms in the aforementioned Lie algebras and hence of the abstract commensurability classes of nilpotent crystallographic subgroups of $\mathrm{U}(3,1)\ltimes \mathbb{C}^{3+1}$. They are found putting together Proposition \ref{Q-form general case}, \ref{Q form alpha=0} and \ref{Q form degenerate case}. We present these Lie algebras defined over the field $\mathbb{K}$ in a compact version that is valid for both $\mathbb{K}=\mathbb{R}$ or $\mathbb{K}=\mathbb{Q}$. The presentation is given in the basis $\{x_1,\ldots,x_8\}$ and we will write only the non zero Lie brackets. For the Lie algebras that decompose as a direct sum of an abelian ideal and a smaller dimensional Lie algebra we have written in brackets the corresponding names in the lists of de Graaf \cite{Graaf} for dimension up to $6$ and of Gong \cite{Gong} for dimension $7$.\\
\space
For the case $\pi(\gamma_3(i\xi)-J\gamma_3(\xi))=0$, see Proposition \ref{w_0=0}, we have
\begin{itemize}
	\item $L_1$: abelian, 
	\item $L_2$: $[x_1,x_2]=x_3$ $\qquad (L_{3,2})$,
	\item $L_3$: $[x_1,x_2]=x_4,[x_1,x_3]=x_5$ $\qquad (L_{5,8})$,
	\item $L_4$: $[x_1,x_2]=x_4, [x_1,x_3]=x_5, [x_1,x_4]=x_6, [x_1,x_5]=x_7$ $\qquad (247A)$, 
	\item $L_5$: $[x_1,x_2]=x_4, [x_1,x_3]=x_5, [x_1,x_4]=x_6, [x_1,x_5]=x_7, [x_2,x_3]=x_6$ $\qquad(247L)$,
	\item $L_6^\mathbb{K}(\varepsilon)$:  $[x_1,x_2]=x_4, [x_1,x_3]=x_5, [x_1,x_4]=x_7, [x_1,x_5]=x_8, [x_2,x_6]=\varepsilon x_8, [x_3,x_6]=-x_7$ with $\varepsilon\in \mathbb{K}_{>0}$ and $L_{6}(\varepsilon)\cong L_6(\varepsilon')$ if and only if there exists $\alpha\in\mathbb{K}^*$ such that $\varepsilon'=\alpha^2\varepsilon$.
\end{itemize}
\space
For the case $\pi(\gamma_3(i\xi)-J\gamma_3(\xi))\not =0$ and $\gamma_2=0$, see Proposition \ref{gamma_2=0}, we have 

\begin{itemize}
	\item $N_1$: $ [x_1,x_2]=x_3,[x_1,x_3]=x_5,[x_2,x_4]=x_6$ $\qquad (L_{6,19}(0))$,
    \item $N_2$: $[x_1,x_2]=x_3,[x_1,x_3]=x_5,[x_1,x_4]=x_6$ $\qquad (L_{6,25})$, 
    \item $N_3^\mathbb{K}(\varepsilon)$: $[x_1,x_2]=x_3,[x_1,x_3]=x_5,[x_1,x_4]=\varepsilon x_6, [x_2,x_3]=x_6,[x_2,x_4]=x_5$ with $\varepsilon\in\mathbb{K}$ and $N_3(\varepsilon)\cong N_3(\varepsilon')$ if and only if there exists $\alpha\in\mathbb{K}^*$ such that $\varepsilon'=\alpha^2\varepsilon$ $\qquad (L_{6,24}(\varepsilon))$,
    \item $N_4$: $[x_1,x_2]=x_3,[x_1,x_3]=x_5,[x_1,x_4]=x_6,[x_2,x_4]=x_5$ $\qquad (L_{6,23})$,  
    \item $N_5: [x_1,x_2]=x_3, [x_1,x_3]=x_6, [x_1,x_5]=x_7, [x_2,x_4]=x_6$ $\qquad(257A)$,
    \item $N_6$: $[x_1,x_2]=x_3, [x_1,x_3]=x_6, [x_1,x_4]=x_7, [x_2,x_5]=x_7$ $\qquad(257B)$,
    \item $N_7$: $[x_1,x_2]=x_3, [x_1,x_3]=x_6, [x_2,x_4]=x_6, [x_2,x_5]=x_7$ $\qquad(257C)$,
    \item $N_8$: $[x_1,x_2]=x_3, [x_1,x_3]=x_6, [x_1,x_4]=x_7, [x_2,x_4]=x_6, [x_2,x_5]=x_7$ $\qquad(257D)$,
    \item $N_9^\mathbb{K}(\varepsilon)$: $[x_1,x_2]=x_3, [x_1,x_3]=x_6, [x_1,x_4]=x_6, [x_1,x_5]=x_7, [x_2,x_3]=x_7, [x_2,x_5]=\varepsilon x_6$ with  $\varepsilon\in\mathbb{K}$ and $N_{9}(\varepsilon)\cong N_{9}(\varepsilon')$ if and only if there exists $\alpha\in\mathbb{K}^*$ such that $\varepsilon'=\alpha^6\varepsilon$. Over $\mathbb{R}$ to classify equivalence classes the parameter $\varepsilon$ can then take three values $\varepsilon\in\{0,1,-1\}$ that correspond to $(257I),(257J_1)$ and $(257J)$ respectively,
    \item $N_{10}: [x_1,x_2]=x_3, [x_1,x_3]=x_7, [x_1,x_4]=x_8, [x_2,x_5]=x_7, [x_2,x_6]=x_8$
    \item $N_{11}: [x_1,x_2]=x_3, [x_1,x_3]=x_7, [x_1,x_4]=x_8, [x_2,x_3]=x_8, [x_2,x_5]=x_7, [x_2,x_6]=x_8 $
    \item $N_{12}: [x_1,x_2]=x_6, [x_1,x_3]=x_7, [x_1,x_4]=x_8, [x_2,x_5]=x_8$
    \item $N_{13}: [x_1,x_2]=x_6, [x_1,x_3]=x_7, [x_1,x_4]=x_8, [x_2,x_4]=x_7, [x_2,x_5]=x_8$.
\end{itemize}
\space
For the case $\pi(\gamma_3(i\xi)-J\gamma_3(\xi))\not =0$ and $\gamma_2\ne 0$, see Proposition  \ref{alpha=0}, \ref{w_0not0,gamma2not0}, \ref{Q-form general case} and \ref{Q form alpha=0}, we have 
\begin{itemize}
	\item $\mathfrak{g}_\mathbb{K}(0,0,0,0)$: $[x_1,x_2]=x_4, [x_1,x_3]=x_5, [x_2,x_3]=x_6, [x_1,x_4]=x_7, [x_1,x_5]=x_8$
	\item $\mathfrak{g}_\mathbb{K}(0,0,\varepsilon,1)$: $[x_1,x_2]=x_4, [x_1,x_3]=x_5, [x_2,x_3]=x_6, [x_1,x_4]=x_7, [x_1,x_5]=x_8, [x_2,x_6]=\varepsilon x_8,$ $ [x_3,x_6]=x_7$  with $\varepsilon\in\mathbb{K}$ such that $\mathfrak{g}_\mathbb{K}(0,0,\varepsilon',1)\cong \mathfrak{g}_\mathbb{K}(0,0,\varepsilon,1)$ if and only if there exists $\alpha\in\mathbb{K}^*$ such that $\varepsilon'=\alpha^2\varepsilon$
     \item $ \mathfrak{g}_\mathbb{R}(a,b,c):\  [x_1,x_2]=x_4, [x_1,x_3]=x_5, [x_2,x_3]=x_6, [x_1,x_4]=x_7, [x_1,x_5]=x_8, [x_2,x_4]=x_7,$ $ [x_2,x_5]=- x_8,  [x_3,x_4]=-x_8, [x_3,x_5]=3x_7, [x_2,x_6]=ax_7+bx_8, [x_3,x_6]=cx_7-ax_8 $ \\ with   $a,b,c\in\mathbb{R}, (a,b)\ne (0,-2) \text{ and }\mathfrak{g}_\mathbb{R}(a,b,c)\cong \mathfrak{g}_\mathbb{R}(-a,b,c)$
  \item $\mathfrak{g}_\mathbb{Q}(e,f,g,h,j,k,l)$ with $e,f,g,h,j,k,l\in\mathbb{Q} $. This is a general family of Lie algebras of the form: \\$[x_1,x_2]=x_4, [x_1,x_3]=x_5, [x_2,x_3]=x_6, [x_1,x_4]=x_7, [x_1,x_5]=x_8, [x_2,x_4]=fx_7+hx_8, [x_2,x_5]=-gx_7-f x_8,  [x_3,x_4]=-gx_7-fx_8, [x_3,x_5]=ex_7+gx_8,[x_2,x_6]=jx_7+kx_8, [x_3,x_6]=lx_7-jx_8$ .\\
  Each isomorphism class of this family is represented by one of the Lie algebras of the following lists. For all of them we will ask \\
  $ 3f^2g^2+6efgh-4f^3h-4g^3e-e^2h^2<0 $ and $\begin{pmatrix} hl+2gj-fk \\ -eh+fg+j\\ 2g^2-2fh+k\end{pmatrix}\wedge \begin{pmatrix} -gl+ek-2fj\\ 2eg-2f^2+l\\ eh-fg-j \end{pmatrix}\ne 0$.
  \begin{itemize}
     \item $\mathfrak{h}^1_\mathbb{Q}([e],[g],j,k,l):$ with $j,k,l\in\mathbb{Q}$ and $(e,g)\in\mathbb{Q}^2$ representatives of the equivalence class defined by $(e,g)\sim (e',g')$ if and only if there exists $\sigma_1,\sigma_2\in\mathbb{Q}^*$ such that $e'=\sigma_1^3 e$ and $g'=\sigma_1\sigma_2^2g$.  This family is defined by $f=h=0$.
	\item $\mathfrak{h}^2_\mathbb{Q}([e],m,j,k,l,[t]):$ with $j,k,l\in\mathbb{Q}$, $e\in\mathbb{Q}^*$ a representative of the equivalent class defined by $e\sim e'$ if and only if there exists $\mu\in\mathbb{Q}^*$ such that $e'=\mu^9e$, $m\in\mathbb{Q}\cup\{\infty\}$ and $t\in\mathbb{Q}\setminus\mathbb{Q}^3$ a representative of the equivalence class defined by $t\sim t'$ if and only if there exists $\mu\in\mathbb{Q}$ such that $t'=\mu^3t^j$ with $j=1,2$.
	\begin{itemize}
	\item if $m\in\mathbb{Q}$: this family is defined by $f=0,g=-emt$ and $h=-et(m^3t+1)$,
	\item if $m=\infty$: this family is defined by $f=0,g=0$ and $h=-et^2$.
     \end{itemize}
	\item $\mathfrak{h}^3_\mathbb{Q}([e],m,j,k,l,[t]):$  with  $j,k,l\in\mathbb{Q}$, $e\in\mathbb{Q}^*$ a representative of the equivalent class defined by $e\sim e'$ if and only if there exists $\mu\in\mathbb{Q}^*$ such that $e'=\mu^9e$, $m\in\mathbb{Q}\cup\{\infty\}$ and $t\in\mathbb{Q}$ such that $t^2-4>0$, $t^2-4\notin\mathbb{Q}^2$ and $t$ not in the image of the function $f(x)=x^3-3x$ over the rational. Finally we choose $t$ as a representative of the equivalence class defined by $t\sim t'$ if and only if there exists $\alpha,\beta\in\mathbb{Q}$ with $\alpha^2+t\alpha\beta+\beta^2=1$ such that $t'=-3t\alpha^2\beta+t\beta^3+6\alpha+\alpha^3t^2-8\alpha^3$.
	\begin{itemize}
	\item $m\in\mathbb{Q}$: this family is defined by $f=-2em,g=e(3m^2-tm-1)$ and $h=et(3m^2-m^3t-1)$,
	\item $m=\infty$: this family is defined by $f=-2e,g=3e$ and $h=-et^2$.
	\end{itemize}
	
	\end{itemize}
	Let us remark that the presentation of the Lie algebras depend on the choice of representative of the class $[t]$.
	\item $\mathfrak{g}_\mathbb{K}(0,-2,c): [x_1,x_2]=x_4, [x_1,x_3]=x_5, [x_2,x_3]=x_6, [x_1,x_4]=x_7, [x_1,x_5]=x_8, [x_2,x_4]=x_7,$ $[x_2,x_5]=- x_8,  [x_3,x_4]=-x_8, [x_3,x_5]=3x_7, [x_2,x_6]=-2x_8, [x_3,x_6]=cx_7$ with $c\in\mathbb{K}$.
\end{itemize}

\begin{rema}
Since it is not written in Gong's thesis we point out that the isomorphism between $N_9^\mathbb{R}(-1)$ and $(37B)$ is given by $x_1'=x_2-x_4,x_2'=x_1-x_3,x_3'=x_1+x_3,x_4'=x_2+x_4x_5'=2(x_7-x_5), x_6'=2x_6, x_7'=2(x_5+x_7)$.
\end{rema}

\bibliographystyle{amsplain}

\begin{thebibliography}{10}

\bibitem{Medina} A. Aubert, A. Medina, "Groupes de Lie pseudo-riemanniens plats.", Tohoku Math. J. (2) 55 (2003);	
\bibitem{Auslander} L. Auslander, "Simply transitive groups of affine motions.", Amer. J. Math. 99 (1977);
\bibitem{Ausl2} L. Auslander, "On radicals of discrete subgroups of Lie groups.", Amer. J. Math. 85 (1963);
\bibitem{AuslMark} L. Auslander, L. Markus, "Flat Lorentz 3-manifolds.", Mem. Amer. Math. Soc. No. 30 (1959); 
\bibitem{Zeghib} A. Ben Ahmed, A. Zeghib, "On homogeneous Hermite-Lorentz spaces.", Asian J. Math. 20 (2016);
\bibitem{Boucetta} M. Boucetta, H. Lebzioui, "On Flat Pseudo-Euclidean Nilpotent Lie Algebras", https://arxiv.org/abs/1711.06938;
\bibitem{Carr}  Y. Carri\` ere, F. Dal'bo, "G\' en\' eralisations du premiere  th\' eor\` eme de  Bieberbach sur les groupes cristallographiques" , Enseignement  Math. 35  (1989);
\bibitem{Cornulier} Y. Cornulier, "Gradings on Lie algebras, systolic growth, and cohopfian properties of nilpotent groups.", Bull. Soc. Math. France 144 (2016);
\bibitem{Cor}  L. J. Corwin, F. P. Greenleaf, "Representations of nilpotent Lie groups and their applications. Part I. Basic theory and examples", Cambridge Studies in Advanced Mathematics, 18. Cambridge University Press, Cambridge, (1990),
\bibitem{Fried} D. Fried, "Flat spacetimes", J. Differential Geom. 26 (1987);
\bibitem{FG} D. Fried, W. Goldman, "Three-Dimensional affine crystallographic groups.",  Adv. in Math. 47 (1983);
\bibitem{FGH}  D. Fried, W. Goldman, M. W. Hirsch, " Affine manifolds with nilpotent holonomy.", Comment. Math. Helv. 56 (1981);
\bibitem{Goldman} W. Goldman, Y. Kamishima, "The fundamental group of a compact flat Lorentz space form is virtually polycyclic.", J. Differential Geom. 19 (1984);
\bibitem{Gong} M. P. Gong, "Classification of nilpotent Lie algebras of dimension 7", Thesis (Ph.D.)–University of Waterloo (Canada), (1998);
\bibitem{Graaf} W. A. de Graaf, "Classification of 6-dimensional nilpotent Lie algebras over fields of characteristic not 2",  J. Algebra 309 (2007);
\bibitem{Greiter} G. Greiter, "A simple proof for a theorem of {K}ronecker", Amer. Math. Monthly, 85, (1978);
\bibitem{Margulis} F. Grunewald, G. Margulis, "Transitive and quasitransitive actions of affine groups preserving a generalized Lorentz-structure",  J. Geom. Phys. 5 (1988);
\bibitem{Knapp} A. W. Knapp, "Lie groups beyond an introduction", Progress in Mathematics, Boston (2002);
\bibitem{Marques} S. Marques, K. Ward, "Cubic Fields: A Primer", https://arxiv.org/abs/1703.06219 (2017);
\bibitem{Mostow} G. D. Mostow, "Factor spaces of solvable groups.", Ann. of Math. (2) 60, (1954);
\bibitem{Rag} M. S. Raghunathan, "Discrete subgroups of Lie groups", Springer, Berlin, (1972);
\bibitem{Salamon} S. M. Salamon, "Complex structures on nilpotent Lie algebras.", J. Pure Appl. Algebra 157 (2001);
\bibitem{Sch1} J. Scheuneman, "Fundamental groups of compact complete locally affine complex surfaces. II.", Pacific J. Math. 52 (1974);
\bibitem{Skje} T. Skjelbred, T. Sund, "On the classification of nilpotent Lie algebras", Technical report, Mathematisk institutt, Universitetet i Oslo, (1977);
\bibitem{Vaisman} I. Vaisman, "Symplectic geometry and secondary characteristic classes. ", Progress in Mathematics, 72. Birkhäuser Boston, Inc., Boston, MA, (1987).

\end{thebibliography}

\end{document}